%% file: main-sphere.tex
\documentclass[a4paper,12pt]{amsart}
\usepackage{lipsum}
\usepackage{amssymb,amsthm}
\usepackage[foot]{amsaddr}
\usepackage{amsmath}
\usepackage{tikz}
\usetikzlibrary{quotes,angles}
\usepackage{graphicx}
\usepackage{subcaption}
\usepackage[shortlabels]{enumitem}
\usetikzlibrary{arrows}
\usepackage{ulem}
\usepackage{float}
\usepackage{graphicx}
\usepackage{subcaption}
\usepackage{hyperref}
\usepackage[section]{placeins}
\graphicspath{{sphere_figs/}}
\usetikzlibrary{decorations.pathmorphing}
\tikzset{snake it/.style={decorate, decoration=snake}}
\usepackage{pgfplots}

\usepackage{xspace}%

\usepackage{ulem} 
\pgfplotsset{compat=1.18}
\makeatletter
\newcommand*{\rom}[1]{\expandafter\@slowromancap\romannumeral #1@}
\makeatother

\numberwithin{equation}{section}

\theoremstyle{plain}
\newtheorem{theorem}{Theorem}
\numberwithin{theorem}{section}
\newtheorem{proposition}[theorem]{Proposition}
\newtheorem{lemma}[theorem]{Lemma}
\newtheorem{corollary}[theorem]{Corollary}
\theoremstyle{definition}
\newtheorem{definition}[theorem]{Definition}
\theoremstyle{remark}
\newtheorem{remark}[theorem]{Remark}
\newtheorem{example}[theorem]{Example}
\theoremstyle{remark}

\theoremstyle{remark}

\newcommand{\eps}{\epsilon}
\newcommand{\smo}{\setminus \mathbf{0}}
\newcommand{\zero}{\mathbf{0}}

\newcommand{\Norm}[1]{\left\lVert#1\right\rVert}      
\newcommand{\norm}[1]{\left\lvert#1\right\rvert}      

\newcommand{\abs}[1]{\left|#1\right|}                 
\newcommand{\paren}[1]{\left(#1\right)}               
\newcommand{\sparen}[1]{\left\{#1\right\}}      

\renewcommand{\d}{\,\mathrm{d}}  

\newcommand{\dd}{\mathrm{d}}  

\newcommand{\supp}{\operatorname{supp}} 
\newcommand{\vp}{\varphi}

\newcommand{\Ac}{\mathcal{A}}

\newcommand{\Cc}{\mathcal{C}}

\newcommand{\Dc}{\mathcal{D}}
\newcommand{\Ec}{\mathcal{E}}
\newcommand{\Fc}{\mathcal{F}}

\newcommand{\Nc}{\mathcal{N}}

\newcommand{\Pc}{\mathcal{P}}

\newcommand{\Sc}{\mathcal{S}}

\newcommand{\WF}{\mathrm{WF}}                         
\newcommand{\wf}{\mathrm{WF}}                         

\newcommand{\partyf}[2]{\frac{\partial #2}{\partial y_{#1}}}

\newcommand{\vv}{{\mathbf{v}}}

\newcommand{\bpm}{\begin{pmatrix}}
\newcommand{\epm}{\end{pmatrix}}

\newcommand{\vx}{{\mathbf{x}}}
\newcommand{\vxo}{\mathbf{x}_0}
\newcommand{\xio}{\xi_0}

\newcommand{\vy}{{\mathbf{y}}}
\newcommand{\vz}{{\mathbf{z}}}
\newcommand{\vs}{\mathbf{s}}
\newcommand{\vd}{\mathbf{d}}

\newcommand{\vxi}{{\boldsymbol{\xi}}}
\newcommand{\vxio}{{\boldsymbol{\xi}_0}}

\newcommand{\om}{\omega}

\newcommand{\vsig}{{\boldsymbol{\sigma}}} 

\newcommand{\snm}{S^{n-1}}

\newcommand{\ip}[1]{\left\langle {#1}\right\rangle}
\newcommand{\nablay}{\nabla_\vy}

\newcommand{\rfr}{\mathfrak{r}}

\newcommand{\psido}{$\Psi$do\xspace}

\newcommand{\rr}{{{\mathbb R}}}

\newcommand{\rtwo}{{{\mathbb R}^2}}
\newcommand{\rthree}{{{\mathbb R}^3}}
\newcommand{\rn}{{{\mathbb R}^n}}

\newcommand{\hvx}{\hat{\vx}}
\newcommand{\hxi}{\hat{\xi}}
\newcommand{\hlambda}{\hat{\lambda}}
\newcommand{\homega}{\hat{\omega}}
\newcommand{\hsigma}{\hat{\sigma}}

\newcommand{\hAc}{\widehat{\mathcal{A}}}

\newcommand{\CcO}{\Cc_\Omega}
\newcommand{\YO}{Y_\Omega}

\newcommand{\dR}{\dot{\mathbb{R}}}
\newcommand{\drn}{\dot{\mathbb{R}^n}}

\newcommand{\naby}{\nabla_\vy}

\newcommand{\st}{\hskip 0.3mm : \hskip 0.3mm}

\newcommand{\be}{\begin{equation}}
\newcommand{\bea}{\begin{eqnarray}}
\newcommand{\eea}{\end{eqnarray}}
\newcommand{\bean}{\begin{eqnarray*}}
\newcommand{\eean}{\end{eqnarray*}}

\newcommand{\bel}[1]{\begin{equation}\label{#1}}
\newcommand{\ee}{\end{equation}}
\newcommand{\eel}[1]{{\label{#1}\end{equation}}}
\newcommand{\intt}{{\operatorname{int}}}
\newcommand{\cl}{{\operatorname{cl}}}
\newcommand{\bd}{{\operatorname{bd}}}

\DeclareMathOperator*{\argmin}{arg\,min}

\usepackage{fullpage}

\usepackage{kbordermatrix}
\renewcommand{\kbldelim}{(}
\renewcommand{\kbrdelim}{)}

\usepackage{datetime}
\usetikzlibrary{quotes, angles}

\newcommand\irregularcircle[2]{
  \pgfextra {\pgfmathsetmacro\len{(#1)+rand*(#2)}}
  +(0:\len pt)
  \foreach \a in {10,20,...,350}{
    \pgfextra {\pgfmathsetmacro\len{(#1)+rand*(#2)}}
    -- +(\a:\len pt)
  } -- cycle
}

\title[short]{Spherical Radon transforms with smoothly varying radii\\{\footnotesize\ddmmyyyydate\today~\currenttime}}
\author{James W. Webber\textsuperscript{$\dagger$}}
\author{Eric Todd Quinto$^\ddagger$}
\address[James W. Webber (corresponding author)]{Cleveland Clinic Lerner College of Medicine of Case Western Reserve University School of Medicine, EC-10 Cleveland Clinic, 9501 Euclid Ave, Cleveland, OH 44195}
\address[Eric Todd Quinto]{Department of Mathematics, Tufts
University, 177 College Ave, Medford, MA 02155}
\email[A1,A2]{jwebber5@bwh.harvard.edu\textsuperscript{$\dagger$}, todd.quinto@tufts.edu\textsuperscript{$\ddagger$}}

\providecommand{\keywords}[1]
{
  \small	
  \textbf{\textit{Keywords---}} #1
}

\begin{document}

\begin{abstract}
We present an analysis of a novel spherical Radon transform, $R$,  which defines the integrals of a function, $f$, in $\mathbb{R}^n$ over spheres with arbitrary center ($\vy$) and radii, $r(\vy)$, which vary smoothly with $\vy$. We first establish sufficient and necessary conditions on $r$ and $\text{supp}(f)$ so that $R$ satisfies the Bolker condition, and further conditions which allow $f$ to be recovered stably from $Rf$. We then apply this theory to a number of example applications in Compton Scatter Tomography (CST) and Ultrasound Reflection Tomography (URT). For each application considered, we also provide injectivity proofs and explicit inversion formulae, some of which are based on the generalized theory presented by Palamodov \cite{palamodov2012uniform}. We then combine our microlocal theory and injectivity results to prove stability estimates for our transforms. In addition, to validate our theory, we provide simulated image reconstructions.
\end{abstract}
\maketitle

 - spherical Radon transforms, inversion methods, microlocal analysis


\input{introduction.tex}

\input{definitions.tex}
\input{bolker.tex}

\input{example4.1.tex}

\input{example4.2-v2.tex}

\input{example4.3.tex}

\input{images_new.tex}
\input{conclusion.tex}

\section*{Acknowledgements:} 
The authors thank Jan Boman for the proof of Proposition
\ref{prop:real-valued}. The first author wishes to acknowledge funding
support from Brigham Ovarian Cancer Research Fund, The V Foundation,
Abcam Inc.,  Aspira Women's Health, The Honorable Tina Brozman Foundation, The Cleveland Clinic Foundation, and The National Cancer Institute R03CA283252-01. The second author thanks the
Simons Foundation for grant 708556 which partially supported this
research. 

\bibliographystyle{abbrv} 
\bibliography{RefRevolution}

\end{document}

%% file: introduction.tex
\section{Introduction} In this paper, we present microlocal and
injectivity analyses of a new spherical Radon transform, $R$, which
integrates a function, $f \in L^2(\mathbb{R}^n)$, of compact support
over spheres with arbitrary center, $\vy \in \mathbb{R}^n$, and
radius, $r(\vy) \in C^{\infty}(\mathbb{R}^n)$, which varies smoothly
with $\vy$. In particular, we provide simple to compute conditions of
$r$ so that $R$ satisfies the Bolker condition, and we give further
geometric requirements which allow for stable inversion of $R$.
Spherical Radon transforms have been covered extensively in the
literature \cite{andersson1988determination, kunyansky2007explicit,
rubin2002inversion, haltmeier2017spherical, Nguyen-Pham, Q2006:supp,
klein2003inverting, ambartsoumian2010inversion, SU:SAR2013, Caday:SAR,
Co1963, KrQu2011, AFKNQ:common-midpoint, felea2013microlocal,
agranovsky1996injectivity}. To the best of our knowledge, however, we
do not believe the proposed transform has been addressed from a
microlocal perspective. For example, much of the literature focuses on
geometries where the $\vy$ are constrained to an $(n-1)$-dimensional
surface, and $r$ varies independently of $\vy$, e.g., see
\cite{Nguyen-Pham}. Special cases of our geometry have been studied
\cite{norton,cormack1980radon}, e.g., in those papers, using our
notation, $r(\vy) = |\vy|$ although $r$ is not smooth at $\zero$
(Example \ref{ex:CQ} provides a closely related transform that does
satisfy most of our theory). However, the general case is not covered.
We aim to address this here.

In \cite{Nguyen-Pham}, a microlocal analysis of a spherical transform, $\mathcal{R}$, is considered. In that paper, the sphere centers lie on a smooth $(n-1)$-dimensional surface, $\mathcal{S}$, in $\mathbb{R}^n$, and the sphere radii are unconstrained. It is proven that $\mathcal{R}$ is a Fourier Integral Operator (FIO) away from $\mathcal{S}$ and that the left projection of $\mathcal{R}$, $\Pi_L$, drops rank on hyperplanes tangent to $\mathcal{S}$. The authors also show that artifacts exist in the reconstruction (e.g., using filtered backprojection), and these are generated by reflections through planes tangent to $\mathcal{S}$. They go one to consider examples of non-smooth $\mathcal{S}$ (e.g., surfaces with corners) and identify the image artifacts. This theory, in the case of smooth $\mathcal{S}$, was later generalized to ellipsoid an hyperboloid integral surfaces in \cite{webber2023ellipsoidal}.

In \cite{agranovsky1996injectivity} microlocal analysis is used to prove injectivity results for a spherical Radon transform in $\mathbb{R}^2$. In this paper, the sphere radii, $r$, vary freely and independently of $\vy$, and the authors identify injectivity sets for the $\vy$. It is shown using analytic wavefront set theory that the solution is unique as long as the set of sphere centers is not constrained to either a finite set or a special union of lines (Coxeter system) through the origin. In \cite{Q2006:supp}, injectivity proofs are presented for a Radon transform which integrates a function over spheres with centers on a smooth surface, $S$, in $\mathbb{R}^n$. More specifically, the authors provide sufficient conditions on $S$ (e.g., $S$ is a real analytic surface) and $\text{supp}(f)$ so that the solution is unique. In this paper, similar to \cite{agranovsky1996injectivity} , the sphere radii vary independently of $\vy$. In section \ref{examples}, we present injectivity proofs for $R$ for some specific example cases of interest in CST and URT. It is noted that the theory of \cite{agranovsky1996injectivity} and \cite{Q2006:supp} does not apply here since in our case, $r(\vy)$ is dependent on $\vy$.

Spherical Radon transforms and microlocal analysis have been applied extensively in Synthetic Aperture Radar (SAR) \cite{SU:SAR2013, Caday:SAR, Co1963, KrQu2011, AFKNQ:common-midpoint, felea2013microlocal}. In SAR, the set of $\vy$ (which represent a plane flight path) is usually constrained to a 1-D curve in $\mathbb{R}^2$. For example, in \cite{SU:SAR2013}, the authors compare curved and straight scanning paths, $\gamma$, from a microlocal perspective. Specifically, the authors identify mirror point type artifacts in the reconstruction, which happen through lines tangent to $\gamma$, making $\WF(f)$ hard to identify. They compare the artifacts with differing $\gamma$ and show that it is possible for the singularities of $f$ to cancel in the reconstruction, for both curved and straight flight paths. However, in the special case when $\gamma$ is the boundary of a convex domain, $\Omega$, and $\text{supp}(f) \subset \Omega$, the authors prove that $\WF(f)$ can be recovered. In these works, the sphere radii vary independently of $\vy$. Thus, such SAR applications and microlocal theory does not fit our framework. 

In section \ref{examples}, we present multiple applications of our theory to CST and URT. In particular, we introduce a novel rotational scanning modality in CST. This geometry is of note as we are able to prove that the reconstruction is stable in Sobolev space. In the CST literature, a number of scanning geometries have been proposed \cite{norton, truong2019compton, RigaudComptonSIIMS2017, webber2019compton, cebeiro2021three}. In those papers, there is inversion instability due to, e.g., limited wavefront coverage (i.e., not all the image edges are detectable) or additional singularities appearing in the reconstruction (artifacts) due to failure of the Bolker condition. For example, in \cite{norton} a Radon transform, $\mathcal{I}$, is considered which integrates $f$ over all spheres which pass through the origin. Here, while the Bolker condition is satisfied for $\text{supp}(f)$ bounded away from the origin, not all wavefronts of $f$ are detectable in the data, e.g., elements of the form $(\vx,\vx^{\perp}) \in \WF(f)$, where $\vx^{\perp}$ satisfies $\vx^{\perp}\cdot \vx = 0$. To the best of our knowledge, we present here the first scanning geometry in CST which yields a stable solution, and we formalize this idea using Sobolev spaces. We also provide explicit, closed-form inversion formulae using the generalized theory of \cite{palamodov2012uniform}. Thus, the proposed geometry, from a purely theoretical perspective (e.g., not accounting for physical effects), is optimized for stability when compared with other modalities from the literature.

In addition to CST application, we will also explore the case when
$r(\vy) = r$ is constant, which may have applications in URT where
spherical waves are used for imaging. Here, $r(\vy) = r$ represents
the scanning depth, and $\vy$ is the location of the sound wave
emitter/receiver. Similar examples have been explored previously in
the literature \cite[Chapter VI]{john2004plane}. Indeed, if the $\vy$ are
allowed to vary freely in $\mathbb{R}^n$, then the recovery of $f$
from $Rf$ in this case is a deconvolution problem. More detailed
support theorems can be found in \cite{agranovsky2011support},
although in that paper the authors consider spheres with centers on $\mathbb{R}^n\backslash K$, where $K$ is a compact region of interest. With this in mind, we provide more limited sphere center sets (represented by a subset $\Omega \subset \mathbb{R}^n$) which are sufficient for stable reconstruction, given a scanning target size. Using spherical harmonic decomposition and Volterra integral equations, we also provide novel injectivity proofs and inversion methods for $R$ in $n$ dimensions when the $\vy$ are constrained to a subset of $\{|\vy| > r\}$ and $\text{supp}(f) \subset \{|\vx| < r\}$. In particular, we only require that the $\vy$ be on one side of the scanning region (full $360^{\circ}$ scanning is not needed) which may be beneficial in some applications, e.g., in URT using a handheld scanner.

The remainder of this paper is organized as follows. In section \ref{sect:defns}, we recite some key definitions and prior theorems that will be used in our proofs. In section \ref{bolker}, we present our main microlocal theorems, and we provide necessary and sufficient conditions on $r(\vy)$ so that $R$ satisfies the Bolker condition. We then use this theory to derive conditions for stable reconstruction. In section \ref{examples}, we present a number of example applications of our analysis to CST and URT. Here, injectivity proofs and global Sobolev space estimates are also provided. In section \ref{images}, we present simulated image reconstructions to validate our theory. In particular, we focus on the scanning geometries related to CST and URT proposed in section \ref{examples}.

%% file: definitions.tex
\section{Definitions}\label{sect:defns}

In this section, we review some theory from microlocal analysis which will be used in our theorems. We first provide some
notation and definitions.  Let $X$ and $Y$ be open subsets of
{$\mathbb{R}^{n_X}$ and $\mathbb{R}^{n_Y}$, respectively.}  Let $\Dc(X)$ be the space of smooth functions compactly
supported on $X$ with the standard topology and let $\mathcal{D}'(X)$
denote its dual space, the vector space of distributions on $X$.  Let
$\Ec(X)$ be the space of all smooth functions on $X$ with the standard
topology and let $\mathcal{E}'(X)$ denote its dual space, the vector
space of distributions with compact support contained in $X$. Finally,
let $\Sc(\rn)$ be the space of Schwartz functions, that are rapidly
decreasing at $\infty$ along with all derivatives. See \cite{Rudin:FA}
for more information. 

We now list some notation conventions that will be used throughout this paper:
\begin{enumerate}
\item If $A\subset \rn$, then $\intt(A)$ and $\bd(A)$ are, respectively, the interior of $A$ and the boundary of $A$.
\item For a function $f$ in the Schwartz space $\Sc(\mathbb{R}^{n_X})$, we {write
\[
\mathcal{F}f(\xi) = \int_{\mathbb{R}^{n_X}} e^{-i x\cdot \xi}f(x)\ \dd x,\quad \mathcal{F}^{-1}f(x) = \frac{1}{(2 \pi)^{n_X}}\int_{\mathbb{R}^{n_X}} e^{i x\cdot \xi}f(\xi)\ \dd \xi
\]
for} the Fourier transform and inverse Fourier transform of $f$,
respectively, {and extend these in the usual way to tempered distributions $\Sc'(\mathbb{R}^{n_X})$} (see \cite[Definition 7.1.1]{hormanderI}). 

\item We use the standard  multi-index notation: if
$\alpha=(\alpha_1,\dots,\alpha_n)\in \sparen{0,1,2,\dots}^{n_X}$
is a multi-index and $f$ is a function on $\mathbb{R}^{n_X}$, then
\[\partial^\alpha f=\paren{\frac{\partial}{\partial
x_1}}^{\alpha_1}\paren{\frac{\partial}{\partial
x_2}}^{\alpha_2}\cdots\paren{\frac{\partial}{\partial x_{n_X}}}^{\alpha_{n_X}}
f.\] If $f$ is a function of $(\vy,\vx,\vsig)$ then $\partial^\alpha_\vy
f$ and $\partial^\alpha_\vsig f$ are defined similarly.

\item \label{item:T*Xident} We identify the cotangent
spaces of Euclidean spaces with the underlying Euclidean spaces. For example, the cotangent space, 
$T^*(X)$, of $X$ is identified with $X\times \mathbb{R}^{n_X}$. If $\Phi$ is a function of $(\vy,\vx,\vsig)\in Y\times X\times \rr^N$,
then we define $\dd_{\vy} \Phi = \paren{\partyf{1}{\Phi},
\partyf{2}{\Phi}, \cdots, \partyf{{n_X}}{\Phi} }$, and $\dd_\vx\Phi$ and $
\dd_{\vsig} \Phi $ are defined similarly. Identifying the cotangent space with the Euclidean space as mentioned above, we let $\dd\Phi =
\paren{\dd_{\vy} \Phi, \dd_{\vx} \Phi,\dd_{\vsig} \Phi}$.

\item For $\Omega\subset \rr^m$, we define $\dot{\Omega}
= \Omega\smo$.  When we write $T^*(X)\smo$, we mean the cotangent
space minus its zero section, $X\times \{\zero\}$. 

\end{enumerate}

\noindent The singularities of a function and the directions in which they occur
are described by the wavefront set \cite[page
16]{duistermaat1996fourier}, which we now define.
\begin{definition}
\label{WF} Let $X$ be an open subset of $\mathbb{R}^{n_X}$ and let $f$ be a
distribution in $\mathcal{D}'(X)$.  Let $(\vx_0,\vxi_0)\in X\times
\drn$.  Then $f$ is \emph{smooth at $\vx_0$ in direction $\vxio$} if
there exists a neighborhood $U$ of $\vx_0$ and $V$ of $\vxi_0$ such
that for every $\Phi\in \Dc(U)$ and $N\in\mathbb{R}$ there exists a
constant $C_N$ such that for all $\vxi\in V$ {and $\lambda >1$},
\begin{equation}\label{rapid decay}
\left|\Fc(\Phi f)(\lambda\vxi)\right|\leq C_N(1+\abs{\lambda})^{-N}.
\end{equation}
The pair $(\vx_0,\vxio)$ is in the \emph{wavefront set,} $\wf(f)$, if
$f$ is not smooth at $\vx_0$ in direction $\vxio$.
\end{definition}

Intuitively, the elements 
$(\vxo,\vxio)\in \WF(f)$ are the point-normal vector pairs at
which $f$ has singularities; $\vxo$ is the location of the 
  singularity, and  $\vxio$ is the direction in which the
  singularity occurs.
  A
 geometric example of the wavefront set is given by the characteristic function $f$ of a domain $\Omega \subset \mathbb{R}^{n_X}$ with smooth boundary, which is $1$ on $\Omega$ and $0$ on $\mathbb{R}^{n_X}\backslash\Omega$. Then the wavefront set is
\[
\WF(f) = \{(\vx,t\vv) \ : t \neq 0, \ \vx \in \partial \Omega, \ \mbox{$\vv$ is orthogonal to $\partial \Omega$ at $\vx$}\}.
\]
In other words, the wavefront set is the set of points in the boundary of $\Omega$ together with the nonzero normal vectors to the boundary. {The set of normals of the surface $\partial \Omega$ is a subset of the cotangent bundle $T^*X$, and here we are using the identification of $T^*X$ with $X \times \mathbb{R}^{n_X}$ mentioned above in point (\ref{item:T*Xident}).} The wavefront set is an important consideration in imaging since elements of the wavefront set will correspond to sharp features of an image.





Our next proposition is elementary but useful.


\begin{proposition}\label{prop:real-valued} Let $f$ be a real-valued
distribution, $\vxo\in\rn$, $\xio\in \rn\smo$.  Then, $(\vxo,\xio)\in
\wf(f)$ if and only if $(\vxo,-\xio)\in
\wf(f)$.\end{proposition}

\begin{proof}
Let $g$ be a real-valued distribution of compact support on $\rn$.
Then for $\xi\in \rn$, the complex conjugate of $\Fc(g)(\xi)$ is
$\Fc(g)(-\lambda\xi)$ by a simple calculation since $g$ is real. By
applying this to the estimate \eqref{rapid decay} for a distribution
$f$, one sees $f$ is smooth at $\vxo$ in direction $\xio$ if and only if
$f$ is smooth at $\vxo$ in direction $-\xio$.
\end{proof}


 \begin{definition}[{\cite[Definition 7.8.1]{hormanderI}}] \label{ellip}We define
 $S^m(Y \times X, \mathbb{R}^N)$ to be the
set of $a\in \Ec(Y\times X\times \mathbb{R}^N)$ such that for every
compact set $K\subset Y\times X$ and all multi--indices $\alpha,
\beta, \gamma$ the bound
\[
\left|\partial^{\gamma}_{\vy}\partial^{\beta}_{\vx}\partial^{\alpha}_{\vsig}a(\vy,\vx,\vsig)\right|\leq
C_{K,\alpha,\beta,\gamma}(1+\norm{\vsig})^{m-|\alpha|},\ \ \ (\vy,\vx)\in K,\
\vsig\in\mathbb{R}^N,
\]
holds for some constant $C_{K,\alpha,\beta,\gamma}>0$. 

 The elements of $S^m$ are called \emph{symbols} of order $m$.  Note
that this symbol class is  sometimes denoted $S^m_{1,0}$.  The symbol
$a\in S^m(Y \times X,\rr^N)$ is \emph{elliptic} if for each compact set
$K\subset Y\times X$, there is a $C_K>0$ and $M>0$ such that
\bel{def:elliptic} \abs{a(\vy,\vx,\vsig)}\geq C_K(1+\norm{\vsig})^m,\
\ \ (\vy,\vx)\in K,\ \norm{\vsig}\geq M.
\ee 
\end{definition}

\begin{definition}[{\cite[Definition
        21.2.15]{hormanderIII}}] \label{phasedef}
A function $\Phi=\Phi(\vy,\vx,\vsig)\in
\Ec(Y\times X\times(\mathbb{R}^N\smo))$ is a \emph{phase
function} if $\Phi(\vy,\vx,\lambda\vsig)=\lambda\Phi(\vy,\vx,\vsig)$, $\forall
\lambda>0$ and $\mathrm{d}\Phi$ is nowhere zero. The
\emph{critical set of $\Phi$} is
\bel{def:SPhi}\Sigma_\Phi=\{(\vy,\vx,\vsig)\in Y\times X\times(\mathbb{R}^N\smo)
: \dd_{\vsig}\Phi=0\}.\ee 
 A phase function is
\emph{clean} if the critical set $\Sigma_\Phi$ is a smooth manifold {with tangent space defined {by} the kernel of $\mathrm{d}\,(\mathrm{d}_\sigma\Phi)$ on $\Sigma_\Phi$. Here, the derivative $\mathrm{d}$ is applied component-wise to the vector-valued function $\mathrm{d}_\sigma\Phi$. So, $\mathrm{d}\,(\mathrm{d}_\sigma\Phi)$ is treated as a Jacobian matrix of dimensions $N\times ({n_Y + n_X}+N)$.}
\end{definition}
\noindent By the {Constant Rank Theorem} \cite[Theorem 4.12]{lee2012smooth} the requirement for a phase
function to be clean is satisfied if
$\mathrm{d}\paren{\mathrm{d}_\vsig
\Phi}$ has constant rank.

\begin{definition}[{\cite[Definition 21.2.15]{hormanderIII} and
      \cite[section 25.2]{hormander}}]\label{def:canon} Let $X$ and
$Y$ be open subsets of $\rn$. Let $\Phi\in \Ec\paren{Y \times X \times
{\rr}^N}$ be a clean phase function.  In addition, we assume that
$\Phi$ is \emph{nondegenerate} in the following sense:
\[\text{$\dd_{\vy}\Phi$ and $\dd_{\vx}\Phi$ are never zero on
$\Sigma_{\Phi}$.}\]
  The
\emph{canonical relation parametrized by $\Phi$} is defined as
\begin{equation}\label{def:Cgenl} \begin{aligned} \Cc=&\sparen{
\paren{\paren{\vy,\dd_{\vy}\Phi(\vy,\vx,\vsig)};\paren{\vx,-\dd_{\vx}\Phi(\vy,\vx,\vsig)}}:(\vy,\vx,\vsig)\in
\Sigma_{\Phi}}{.}
\end{aligned}
\end{equation}
\end{definition}

\begin{definition}\label{FIOdef}
Let $X$ and $Y$ be open subsets of {$\mathbb{R}^{n_X}$ and $\mathbb{R}^{n_Y}$, respectively.} {Let an operator $A :
\Dc(X)\to \mathcal{D}'(Y)$ be defined by the distribution kernel
$K_A\in \mathcal{D}'(Y\times X)$, in the sense that
$Af(\vy)=\int_{X}K_A(\vy,\vx)f(\vx)\mathrm{d}\vx$. Then we call $K_A$
the \emph{Schwartz kernel} of $A$}. A \emph{Fourier
integral operator (FIO)} of order $m + N/2 - (n_X+n_Y)/4$ is an operator
$A:\Dc(X)\to \mathcal{D}'(Y)$ with Schwartz kernel given by an
oscillatory integral of the form
\begin{equation} \label{oscint}
K_A(\vy,\vx)=\int_{\mathbb{R}^N}
e^{i\Phi(\vy,\vx,\vsig)}a(\vy,\vx,\vsig) \mathrm{d}\vsig,
\end{equation}
where $\Phi$ is a clean nondegenerate phase function and $a$ is a
symbol in $S^m(Y \times X , \mathbb{R}^N)$. The \emph{canonical
relation of $A$} is the canonical relation of $\Phi$ defined in
\eqref{def:Cgenl}. $A$ is called an \emph{elliptic} FIO if its symbol is elliptic. An FIO is called a \emph{pseudodifferential operator} if {$X = Y$ and} its canonical relation $\Cc$ is contained in the diagonal, i.e.,
$\Cc \subset \Delta := \{ (\vx,\vxi;\vx,\vxi)\st \vx\in X,\,
\xi\in\dot{\rr^{n_X}}\}$.
\end{definition}






{Let $X$ and $Y$ be
sets and let $\Omega_1\subset X$ and $\Omega_2\subset Y\times X$. The composition $\Omega_2\circ \Omega_1$ and transpose $\Omega_2^t$ of $\Omega_2$ are defined
\[\begin{aligned}\Omega_2\circ \Omega_1 &= \sparen{\vy\in Y\st \exists \vx\in \Omega_1,\
(\vy,\vx)\in \Omega_2}\\
\Omega_2^t &= \sparen{(\vx,\vy)\st (\vy,\vx)\in \Omega_2}.\end{aligned}\]
We now state the H\"ormander-Sato Lemma \cite[Theorem 8.2.13]{hormanderI},
which provides  the relationship between the
wavefront set of distributions and their images under FIO{s}.

\begin{theorem}[H\"ormander-Sato Lemma]\label{thm:HS} Let $f\in \Ec'(X)$ and
let ${A}:\Ec'(X)\to \Dc'(Y)$ be an FIO with canonical relation $\Cc$.
Then, $\wf({A}f)\subset \Cc\circ \wf(f)$.\end{theorem}}

Let $A$ be an FIO, then its formal adjoint $A^*$ is also an FIO, and if $\Cc$ is the canonical relation of $A$, then the
canonical relation of $A^*$ is $\Cc^t$ \cite{Ho1971}. Many imaging techniques are based
on application of the adjoint operator $A^*$ and so to understand artifacts
we consider $A^* A$ (or, if $A$ does not map to $\Ec'(Y)$, then
$A^* \psi A$ for an appropriate cutoff $\psi$). Because of Theorem
\ref{thm:HS},
\begin{equation}
\label{compo}
\wf(A^* \psi A f) \subset \Cc^t \circ \Cc \circ \wf(f).
\end{equation}
The next two definitions provide tools to analyze the composition in equation \eqref{compo}.
\begin{definition}
\label{defproj} Let $\Cc\subset T^*(Y\times X)$ be the canonical
relation associated to the FIO ${A}:\mathcal{E}'(X)\to
\mathcal{D}'(Y)$. We let $\Pi_L$ and $\Pi_R$ denote the natural left-
and right-projections of $\Cc$, projecting onto the appropriate
coordinates: $\Pi_L:\Cc\to T^*(Y)$ and $\Pi_R : \Cc\to T^*(X)$.
\end{definition}

Because $\Phi$ is nondegenerate, the projections do not map to the
zero section.  
%
%
If $A$ satisfies our next definition, then $A^* A$ (or $A^* \psi A$) is a pseudodifferential operator
\cite{GS1977, quinto}.

\begin{definition}[Bolker condition]\label{def:bolker} Let
${A}:\Ec'(X)\to \Dc'(Y)$ be a FIO with canonical relation $\Cc$ then
{$A$} (or $\Cc$) satisfies the \emph{Bolker Condition} if
the natural projection $\Pi_L:\Cc\to T^*(Y)$ is an embedding
(injective immersion).\end{definition}

Thus, using \eqref{compo}, we see that under the Bolker condition the wavefront set of $A^* {\psi} Af$ will be contained in the wavefront set of $f$. Intuitively, the reconstructed image ($A^* {\psi} Af$) will only include singularities at the same positions and in the same directions as the original image ($f$). 



We use the following definition of Sobolev space.
\begin{definition}
We define the Sobolev space order $s\in \rr$ on $\rn$ to be
\begin{equation}\label{def:Halpha}
H^s(\mathbb{R}^n) = \paren{f\in S'(\mathbb{R}^n) : (1 + |\xi|^2)^{s/2} \mathcal{F}f \in L^2(\mathbb{R}^n)}.
\end{equation}
The Sobolev norm  of $f\in H^s(\rn)$ is
\bel{SobolevNorm}
\Norm{f}_s = \paren{\int_{\rn}  (1 + |\xi|^2)^{s}\norm{\mathcal{F}f(\xi)}^2\d \xi}^{1/2}\ee

Let $K$ be a closed subset of $\rn$. We define $H^s(K)$ to be the set
of all distributions in $H^s(\rn)$ with support in $K$. Note that
$H^s(K)$ is closed since $K$ is a closed set. Now, let $\Omega$ be an
open subset of $\rn$. We define $H^s_c(\Omega)$ to be the set of
distributions in $H^s (\rn)$ with compact support in $\Omega$.
Therefore, the closure of $H^s(\Omega)$ in $H^s(\rn)$ is
$H^s(\cl(\Omega))$.

\end{definition}

The definition of Sobolev space on $\rn$ is from \cite[page 200]{natterer},
and implicit in \eqref{def:Halpha} is the assumption that $\Fc f$ is a
locally integrable function.  There are several definitions of Sobolev
spaces on subsets of $\rn$, and ours is most convenient for our purposes,
and it  agrees with \eqref{def:Halpha}. on $\rn$.



\noindent We now state the lemma from \cite{stefanov2004stability} which will be used to prove our stability estimates.

\begin{lemma}
\label{stef}
Let $X$, $Y$, and $Z$ be Banach spaces, let $A : X\to Y$ be a closed linear operator with domain $\mathcal{D}(A)$, and $K : X\to Z$ be a compact linear operator. Further, let
$$\|f\|_X \leq c\paren{\|Af\|_Y + \|Kf\|_Z}$$
hold for any $f\in \mathcal{D}(A)$, and assume $A$ is injective. Then, there exists $c'$ such that
$$\|f\|_X\leq c'\|Af\|_Y$$
for any $f\in\mathcal{D}(A)$.
\end{lemma}

%% file: bolker.tex
\section{Spherical transform with smoothly varying \texorpdfstring{$r$}{r}} 
\label{bolker}
\noindent In this section, we introduce the spherical Radon transform we will be analyzing, and present our main microlocal analysis results. Based on this theory, we provide conditions for stable reconstruction.

We consider the generating function
\begin{equation}
\Psi(\vy,\vx) = |\vx-\vy|^2 - r^2(\vy),
\end{equation}
where $\vx,\vy \in \mathbb{R}^n$, and $r \in C^{\infty}\paren{\mathbb{R}^n}$ is a smooth radius function with $r>0$. The set
$$S(\vy) = \{ \vx \in \mathbb{R}^n : \Psi(\vy,\vx) = 0\}$$
defines a sphere in $\mathbb{R}^n$ with center $\vy$, radius
$r(\vy)$. 
Then, we define the spherical Radon transform
\begin{equation}
\label{radon_def}
\begin{split}
Rf(\vy) &= \int_{S(\vy)} f \mathrm{d}S \\
&= \int_{\mathbb{R}^n}|\nabla_{\vx}\Psi|\delta\paren{\Psi(\vy,\vx)}f(\vx)\mathrm{d}\vx\\
&=
\frac{1}{2\pi}\int_{-\infty}^{\infty}\int_{\mathbb{R}^n}|\nabla_{\vx}\Psi|e^{i\sigma\Psi(\vy,\vx)}f(\vx)\mathrm{d}\vx\,\dd
\sigma
\end{split}
\end{equation}
where $\mathrm{d}S$ is the surface measure on $S(\vy)$ and $\delta$ is
the Dirac delta function. The middle term in \eqref{radon_def} follows
from the theory of Palamodov \cite{palamodov2012uniform}. 



\begin{proposition}
\label{fio_prop} Let $r:\rn\to (0,\infty)$ be smooth.
Assume for all $\vy\in \rn$, $\abs{\nabla_{\vy}r}\neq 1$. Then $R$ is an elliptic FIO order $-(n-1)/2$ with phase
\begin{equation}
\Phi(\vy,\vx,\sigma) = \sigma\paren{|\vx-\vy|^2 -
r^2(\vy)}=\sigma\Psi(\vy,\vx).
\end{equation}

Global coordinates on the canonical relation, $\Cc$, for $R$ are given
by \bel{def:C}\rn\times S^{n-1}\times\dR\ni (\vy,\omega,\sigma)\mapsto
\paren{ \vy,-2\sigma r(\vy)\paren{\omega + \nabla_{\vy}r} ; 
\vy+r(\vy)\omega,-2\sigma r(\vy)\omega }\in \Cc.\ee
\end{proposition}

\begin{proof} We show   $\Phi$ satisfies the conditions to be a phase
function.
The total derivative of $\Phi$ is
\begin{equation}\label{dPhi}
\mathrm{d}\Phi = \paren{
-2\sigma\paren{(\vx-\vy)+r(\vy)\nabla_{\vy}r}, 2\sigma\paren{\vx-\vy},
\Psi(\vy,\vx)}\end{equation} Let $\Sigma_\Psi$ be the critical set of
$\Psi$ given by \eqref{def:SPhi}. Note that $\dd_\vx\Phi$ is never
zero on $\Sigma_\Phi$ since $\mathrm{d}\Phi = 0 \implies (\vx-\vy) = 0
\implies r(\vy) = 0$ but $r>0$.  Note that $\dd_\vy\Phi$ is never
zero because we are assuming
$\abs{\nabla_\vy r}\neq 1$. In addition, $\Phi$ is homogeneous of degree one in
$\sigma$. This shows $\Phi$ is a clean phase function. Then, along
with the second expression in \eqref{radon_def}, shows that $R$ is an FIO.

 The symbol of $R$ is 
\begin{equation}
a(\vy,\vx) = 2|\vx-\vy| = 2r(\vy)>0,
\end{equation}
which is elliptic of order zero since $a$ is smooth, does not depend on
$\sigma$, and is never zero. Therefore, $R$ is an elliptic FIO order
$0+1/2-2n/4 = -(n-1)/2$.

{To derive \eqref{def:C}, we note that when $\sigma \neq 0$,
$\dd_\sigma \Phi = 0$ if and only if $\Psi(\vy,\vx)=0$. Let $Z$ be the
zero set of $\Phi$.\footnote{$Z=\sparen{(\vy,\vx)\in Y\times X\st
\vx\in S(\vy)}$ is the incidence relation of the Radon transform $R$,
and the Schwartz Kernel of $R$ is a distribution on $Y\times X$ that
integrates over $Z$ \cite{Gu1975,GS1977, quinto}.} If $(\vy,\vx)\in Z$
then $\vx\in S(\vy)$, so $\vx$ can be written $\vy +r(\vy)\omega$ for
some $\omega\in S^{n-1}$, and this gives global coordinates on $Z$.
\[(\vy,\omega)\mapsto (\vy,\vy+r(\vy)\omega).\]
Using these coordinates, when $\Psi(\vy,\vx)=0$, the derivatives
$\mathrm{d}_\vy \Psi$ and $\mathrm{d}_\vx \Psi$ become
$$\mathrm{d}_\vy \Phi = -2\sigma\paren{(\vx - \vy)  + r(\vy)\nabla_{\vy}r} = -2\sigma r(\vy)\paren{\omega + \nabla_{\vy}r},$$
and
$$\mathrm{d}_\vx \Phi = 2\sigma (\vx - \vy) = 2\sigma r(\vy)\omega.$$
Expression \eqref{def:C} now follows from Definition
\ref{def:canon}.}\end{proof}

\subsection{The Bolker Condition}
In our next two theorems, we establish necessary and
sufficient requirements on $R$ so that the Bolker condition is
satisfied.

\begin{theorem}
\label{imm_thm}
Let $r:\rn\to (0,\infty)$ be smooth. The left projection of $\Cc$, $\Pi_L$, is an immersion if and only if 
\bel{norm1} \norm{\nabla_{\vy}r}<1\ \ \forall \vy\in \rn.\ee
\end{theorem}
\begin{proof}
The left projection of $R$ is 
\begin{equation}\label{PiL}
\Pi_L(\vy,\omega,\sigma) = \paren{\vy,-2\sigma r(\vy)\paren{\omega + \nabla_{\vy}r}},
\end{equation}
Let us parametrize $\snm$ near $\omega$ using standard spherical
coordinates $\varphi_1,\ldots,\varphi_{n-1}$. Then, the derivative of
$\Pi_L$ is
\begin{equation}\label{DPiL}
D\Pi_L = \begin{pmatrix} I_{n\times n} & 0_{n \times n} \\ \cdot &
-2r(\vy) \left[\sigma\Theta, \paren{\omega + \nabla_{\vy}r}\right]
\end{pmatrix},
\end{equation}
where
$$\Theta = [\omega_{\varphi_1},\ldots,\omega_{\varphi_{n-1}}],$$
and $[\Theta,\omega]$ is an orthogonal matrix and the columns of
$\Theta$ are a basis for the orthogonal complement of $\omega$, which
we denote $\omega^\perp$. 

Using \eqref{DPiL} and since the columns of $\Theta$ are independent,
we see that $\det(D\Pi_L)$ is never zero if and only if
\bel{omperp}\omega+\naby r(\vy) \notin \omega^\perp\ \ \forall
(\vy,\om)\in
\rn\times \snm.\ee Note that \eqref{omperp} is independent of the coordinates
$\varphi_1,\dots, \varphi_{n-1}$. 

The only way for  \eqref{omperp} to hold is if $\naby r(\vy)\cdot
\om \neq -1,  \forall (\vy,\om)\in \rn\times \snm$ or equivalently,
$\abs{\naby r(\vy)}<1$
for all $\vy\in \rn$. 
\end{proof}

\begin{remark}
Under the assumption that $|\nabla_{\vy}r|<1$, $\Pi_L$ is also
surjective. To see this, let us fix $\vy$. Then, as
$|\nabla_{\vy}r|<1$, we can choose $\omega\in \snm$ so that $\omega +
\nabla_{\vy}r$ is in any direction, and vary $\sigma$ so that
$-2\sigma r(\vy)\paren{\omega + \nabla_{\vy}r}$ covers the whole of
$\mathbb{R}^n\backslash \{0\}$.
\end{remark}

Next, we discuss injectivity of $\Pi_L$. This will lead to a
discussion on surjectivity $\Pi_R$

\begin{theorem}\label{thm:PiL-2-to-1} Assume the norm inequality
\eqref{norm1} holds. Then, $\Pi_L:\Cc\to \Pi_L(\Cc)$ is two-to-one.
Specifically, let $(\vy,\eta)\in \Pi_L(\Cc)$ and let $(\vx,\xi)$ be chosen so
$\lambda = (\vy,\eta;\vx,\xi)$ is one of the preimages of
$(\vy,\eta)$. Then, $\vx\in S(\vy)$ and $\xi$ is normal to $S(\vy)$ at
$\vy$. 

To describe the second preimage, let $L_{\vy,\vx}$ be the line
containing $\vx$ and $\vz=\vy - r(\vy)\nabla r(\vy)$. This line
intersects $S(\vy)$ at $\vx$ and a second point
\bel{def:hvx}\hvx=\hvx(\vy,\vx).\ee Further, there is a unique
$\hxi\in \drn$ such that $\hlambda=(\vy,\eta;\hvx,\hxi)\in \Cc$. The
covectors $\lambda$ and $\hlambda$ are the two preimages of
$(\vy,\eta)$.
\end{theorem}

\begin{remark}
  The theorem statement above describes some of the key geometric
properties of the preimages of $(\vy,\eta)$. We will explain further
the importance of the ``mirror points" $\hvx(\vy,\vx)$ in section
\ref{sec:artifacts}.\end{remark}

\begin{proof}[Proof of Theorem \ref{thm:PiL-2-to-1}]
Let $\lambda = (\vy,\eta;\vx,\xi)$ be a preimage of $(\vx,\xi)$ in
$\Cc$. For $\lambda$ to be in $\Cc$, $\vx$ must be in $S(\vy)$ and
$\xi$ must be normal to $S(\vy)$ at $\vx$. This explains why the only
preimages of $(\vx,\xi)$ have $\vy$ coordinates on the line
$\ell_{\vy,\vx}$

Let $\omega =
\xi/\norm{\xi}$ and choose $\sigma$ such that $\xi = \sigma \omega$.
Then, $\lambda$ has coordinates $(\vy,\omega,\sigma)$.

By \eqref{norm1}, $\vy-r(\vy)\nablay r(\vy)$ is inside of $S(\vy)$ so
$L_{\vy,\vx}$ must intersect $S(\vy)$ in  one other point
besides $\vx$. Let $\hvx$ be the other point of intersection.

We now  show there is one and
 only one other preimage of $(\vy,\eta)$ and the second
 preimage of $(\vy,\eta)$ is of the form given in the theorem.

 Let $(\vy,\homega,\hsigma)$ be coordinates of a different preimage of
$(\vy,\eta)=\Pi_L(\vy,\omega,\sigma)$. Then
  \[\Pi_L(\vy,\omega,\sigma) = \Pi_L(\vy,\homega,\hsigma).\] Note that
$\sigma\neq \hsigma$ because, if they were equal,  then
$\omega=\homega$ as can be seen from \eqref{def:C}. 

 A calculation shows that \bel{inj} -\nabla_{\vy}r = t\omega + (1-t)
\homega, \text{ \ where \ } t = \sigma/(\sigma-\hsigma) \neq 0,1.\ee
The expression $t\mapsto t\omega + (1-t)\homega$ parameterizes the
line through $\omega$ and $\homega$. Since we have assumed
$\norm{\nabla_\vy r}< 1$ by \eqref{norm1}, $\homega\neq \omega$. When
we translate to points in $\rn$, this says that $\vx =
\vy+r(\vy)\omega$, $\hvx =\vy+r(\vy)\homega$ are on $L_{\vy\vx}$.
Furthermore, using coordinates $(\vy,\homega,\hsigma)$ in
\eqref{def:C}, we see $\hat\gamma$ is as given in the theorem.

Since $(\hvx, \hxi)=\Pi_R(\vy,\homega,\hsigma)$, $\hxi$ is conormal to
$S(\vy)$ at $\hvx$. This finishes the proof.\end{proof}


\begin{theorem}\label{inj_thm}  Let $r:\rn\to (0,\infty)$ be smooth.  Assume the norm inequality \eqref{norm1} holds. 
Let $\Omega$ be an open set in $\rn$ and let
\bel{def:CcO}\CcO=\sparen{(\vy, \eta,\vx,\xi)\in \Cc \st \vx\in
\Omega}\ee

Let $(\vy,\eta)\in \Pi_L(\Cc)$ and $\lambda = (\vy,\eta;\vx,\xi)$
and $\hlambda=(\vy,\eta;\hvx,\hxi)$ be
its two preimages in $\Cc$. Then the restricted map, $\Pi_L:\Cc_\Omega\to T^*(\rn)$ is injective if and only if for any such $(\vy,\eta)$
at most one of $\vx$ and $\hvx$ is in $\Omega$. 

In this case, $\Cc_\Omega$ satisfies the Bolker condition.\end{theorem}

\noindent The following is the first advantage of knowing the Bolker condition
holds. 

\begin{corollary}\label{cor:bolker} Under the assumptions of Theorem
\ref{inj_thm},\bel{def:DeltaO}\paren{\Cc_\Omega}^t\circ \Cc_\Omega
\subset \Delta_\Omega=\sparen{(\vx,\xi;\vx,\xi)\st \vx\in \Omega, \xi\in
\drn},\ee 
\end{corollary}

This is an immediate corollary of the Bolker condition. Once we
develop conditions under which one can compose $R$ and its adjoint, we
will use this to show the normal operator is a \psido.

\begin{proof}[Proof of Theorem \ref{inj_thm}]
Let $(\vy,\eta)\in \Pi_L(\Cc_\Omega)$ and let $\gamma$ and
$\hat\gamma$ be the two preimages of $\Pi_L$ in the full canonical
relation, $\Cc$ given in Theorem \ref{thm:PiL-2-to-1}. Bolker holds if
and only if at most one of $\lambda$ and $\hat\lambda$ is in
$\Cc_\Omega$. This is true if and at most one of $\vx$ and $\hvx$ is in
$\Omega$. By Theorem \ref{imm_thm} $\Pi_L$ is an immersion, so the
Bolker condition holds.\end{proof}

\input{bolker-S3.2-new.tex}


\subsection{Artifacts for the normal operator and geometric
conditions for stability}\label{sec:artifacts}

We now analyze the artifacts that can occur when the Bolker condition
does not hold.

In some cases, such as when $r(\vy)$ is constant,  $R^*$ can be
composed with $R$, but in others, such as the transform in section
\ref{sec:linearCST}, one needs a cutoff, $h$, to compose: $R^* h R$.  
\begin{equation}\label{def:R dagger}\begin{aligned}
&\text{We will let $R^\dagger = R$ when $R^*$ and $R$ can be
composed and
$R^\dagger = hR$}\\
&\text{for a suitably chosen cutoff, $h$, otherwise.}\end{aligned}
\end{equation}

Note for $\vy\in Y$, $N^*(S(\vy))$ is the conormal bundle of the
sphere $S(\vy)$. Our next corollary puts together our results in the last section to
describe the microlocal artifacts that can be generated by $R^*R^\dagger$.

\begin{corollary}\label{cor:artifact points} Let $(\vx,\xi)\in
\Pi_R(\Cc)$.
Then, there is either one or two preimages of $(\vx,\xi)$ under
$\Pi_R$.

Let $\lambda=(\vy,\eta,\vx,\xi)\in \Cc$ be a preimage. Then,
$(\vx,\xi)\in N^*(S(\vy))$, and $\Pi_L(\lambda)=(\vy,\eta)$. 

For each such $(\vy,\eta)$, there is a second preimage 
under $\Pi_L$, $\hlambda = (\vy,\eta, \hvx,\hxi)\in \Cc$ where $\hvx =
\hvx(\vy,\vx)$ is given by \eqref{def:hvx}.

If $f\in \Ec'(\rn)$, $(\vx,\xi)\in \wf(f)\cap
\Pi_R(\Cc)$, and $(\vx,\xi)\in N^*(S(\vy))$, then $(\hvx(\vy,\vx),\hxi)$ can be in $\wf(R^*R^\dagger f)$, even
if $(\hvx(\vy,\vx),\hxi)\notin \wf(f)$.
\end{corollary}

\begin{remark}\label{rem:artifact} We now explain the importance
of this corollary to the prediction of microlocal artifacts.

Let $f\in \Ec'(\rn)$, $\vx\in \rn$, $(\vx,\xi)\in \Pi_R(\Cc)\cap
\wf(f)$. Let $\vy\in Y$ such that $(\vx,\xi)\in N^*(S(\vy))$ (so, for
some $\eta$, $\lambda = (\vy,\eta,\vx,\xi)\in \Cc$). 

Corollary \ref{cor:artifact points} states that $(\hvx(\vy,\vx),\hxi)$
can be in $\wf(R^*R^\dagger f)$, even if $f$ is smooth near
$\hvx(\vy,\vx)$. For this reason, we will call $\hvx=\hvx(\vy,\vx) $
the \emph{artifact point} to $\vx$ on $S(\vy)$, or just \emph{an
artifact point to $\vx$}. 

Since there will be at most two points $\vy$ such that $(\vx,\xi)\in
N^*(S(\vy))$ by Theorem \ref{thm:PiR-atMost2}, there are at most two artifact points for each
$(\vx,\xi)\in \Pi_R(\Cc)$.

\bigskip

The proof of Theorem \ref{thm:PiL-2-to-1} shows that the point
$\hvx(\vy,\vx)$ in \eqref{def:hvx} is the second point of intersection
(besides $\vx$) of $S(\vy)$ and the line $L_{\vy,\vx}$ containing
$\vx$ and the point $\vz = \vy-r(\vy)\nabla  r(\vy)$. 

The point $\hvx$ is especially easy to describe when $r(\vy) = c$ is
constant. In this case, $\nablay r \equiv 0$ so if $\vx\in S(\vy)$,
then $L_{\vy,\vx}$ is the line through $\vy$ and $\vx$, so $\hvx = \vx
- 2(\vx-\vy)$ is the antipodal point on $S(\vy)$ to $\vx$. This
transform is well-studied, e.g., in \cite{john2004plane}. The
microlocal properties of this transform were worked out in
\cite{Q1993mor} on Riemannian manifolds. We will examine this
transform more in section \ref{sec:URT}.
\end{remark}

\begin{proof}[Proof of Corollary \ref{cor:artifact points}] The conclusion of the first paragraph of the corollary
follows directly from Theorem \ref{thm:PiR-atMost2}. The conclusion of
the second paragraph follows directly from the definitions of $\Pi_R$
and $\Cc$. The conclusion of the third paragraph is in Theorem
\ref{thm:PiL-2-to-1}.

That $\lambda = (\vy,\eta,\vx,\xi)\in \Cc$ for some $\eta$ if and only
if $(\vx,\xi)\in N^*(S(\vy))$ can be seen from \eqref{def:C} since
$(\vy,\vx,\xi)$ determine the coordinates on $\Cc$. The statement
about $\Cc^t\circ \Cc$ follows from the definition of composition of
sets and Theorem \ref{thm:PiL-2-to-1}. \end{proof}


Now that we have defined artifact points, we define artifact sets.

\begin{definition}\label{def:hAc} Let $r:\rn\to(0,\infty)$ be a smooth function that
satisfies \eqref{norm1}. Let $\Gamma\subset T^*(\rn)$. Define
$\hAc(\Gamma)$ to  be the set of all artifact points for all
$(\vx,\xi)\in \Gamma\cap \Pi_R(\Cc)$.  that is,
\bel{hAc}\begin{aligned}\hAc(\Gamma) = &\big\{\hvx(\vy,\vx)\st
(\vy,\vx)\in \rn\times
\rn,\\
&\hspace{1cm} \exists s\in \dot{\rr}, 
(\vx,s(\vx-\vy))\in N^*(S(\vy))\cap
\Gamma\cap\Pi_R(\Cc)
\big\}.\end{aligned}
\ee
If $S$ is a subset of $\rn$, we define $\Ac(S) =
\hat{\Ac}(S\times \drn)$.
\end{definition}

 Corollary \ref{cor:artifact points} justifies why $\hAc$ (resp.\
$\Ac$) in $\eqref{hAc}$ is the set of possible microlocal artifacts
from singularities in $\Gamma$ (resp.\ $S$).


\begin{proposition}
\label{prop_geo} Assume the norm inequality \eqref{norm1} holds. Then,
$R$ satisfies the Bolker condition above the open set $\Omega$ {if and
only if ${\Ac(\Omega) \cap\Omega = \emptyset}$}. {In this
case,}
$R^*{R^\dagger}$ is a pseudodifferential operator.

If, in addition, $\Pi_R:\CcO\to \dot{T^*}(\Omega)$ is surjective and
$R^*$ and $R$ can be composed, then $R^*R$ is elliptic. 
\end{proposition}


\begin{proof}

{First, $\Pi_L:\CcO\to T^*(\YO)$ is an immersion by Theorem
\ref{imm_thm} as \eqref{norm1} holds.  
Therefore, we just need to check injectivity. Let $(\vy,\eta)\in
\Pi_L(\CcO)$ and let \[\lambda = (\vy,\eta,\vx.\xi)\in \CcO.\] By
Corollary \ref{cor:artifact points}, there are two preimages of
$(\vy,\eta)$ in $\Cc$, $\lambda$ with base point (in the $T^*(\Omega)$
coordinate) $\vx\in \Omega$ and another, $\hat{\lambda}$, with base
point $\hvx(\vy,\vx)\in \Ac(\Omega)$. Therefore, $\Pi_L$ is injective
if and only if all artifact points $\hvx(\vy,\vx)$ for all $\vx\in
\Omega$ are not in $\Omega$. This proves the first claim in the
proposition.} Therefore, ${R^*R^\dagger}$ is a \psido.

To prove the   statement about  ellipticity, we first observe that
$\CcO^t\circ \CcO = \Delta_\Omega$ since $\Pi_R$ is surjective.   Then, since
the symbols of $R$ and $R^*$ are positive and $R^*$ and $R$ can be
composed, $R^*R$ is an elliptic pseudodifferential operator.
\end{proof}

\noindent We now introduce the following definition related to inversion stability of $R$.

\begin{definition}
\label{stable_def} Assume the norm inequality \eqref{norm1}  holds.
Let $\Omega$ be an open set in $\rn$ and let $\vx\in \Omega$.


 Consider the following conditions:
\begin{enumerate}
\item \label{1} For every $\omega \in S^{n-1}$,
there exists a $\vy\in Y$ such that $S(\vy)$ intersects $\vx$ with
$\vx - \vy$ parallel to $\omega$. 

\item\label{2} $\vx \notin \mathcal{A}(\Omega)$.
\end{enumerate}
If \eqref{1} and \eqref{2} above hold for a specific $\vx\in \Omega$,
then we say the reconstruction is \emph{weakly stable} at $\vx$. If
\eqref{1} and \eqref{2} hold for every $\vx \in \Omega$, then we say
the reconstruction is \emph{weakly stable} {on
$\Omega$}, or {$R:\Ec'(\Omega)\to \Dc'(\YO)$ is \emph{weakly
stable}}.
\end{definition}

\begin{corollary}\label{cor:weak stability}
Let $\Omega$ be an open set. Reconstruction  (from
$R$) on $\Omega$ is weakly stable if and only if both  $\Pi_R:\CcO\to \Omega$ is surjective
and the Bolker condition holds. 

In this case $R^*{R^\dagger}$ is a \psido, and if
$R^*$ and $R$ can be composed, then $R^*R$ is an elliptic \psido\end{corollary}

\begin{proof}  Let $\vx\in \Omega$.
Condition \eqref{1} is satisfied at $\vx$ if and only if for each
$\xi\in \dot{T^*_\vx}(\Omega)$, there is a $\vy\in \YO$ such that
$\vx\in S(\vy)$ and $\xi$ is conormal to $S(\vy)$ at $\vx$. This is
equivalent to
\[\forall \xi\in \dot{T^*_\vx}(\Omega),\ \exists \vy\in \YO,\
\Pi_R(\vy,\xi/\norm{\xi},\norm{\xi}/(-2r(\vy))) = (\vx,\xi)\] where we
are using coordinates \eqref{def:C} on $\CcO$. Therefore, \eqref{1}
holds for all $\vx\in \Omega$ if and only if $\Pi_R:\CcO\to
\dot{T^*}(\Omega)$ is surjective. Proposition \ref{prop_geo} implies
the equivalence of \eqref{2} with the Bolker condition, noting that we are also assuming a-priori that \eqref{norm1} holds.

The last statements follow directly from  Proposition \ref{prop_geo}.
\end{proof}


{When $R$ is injective and the reconstruction is weakly stable on
$\Omega$, as in Definition \ref{stable_def}, we will sometimes be
able to derive Sobolev continuity estimates for the forward and inverse
operators on $H^s(B)$ for any compact subset $B\subset\Omega$. We will
prove such estimates for each of the examples in the next section, using appropriately chosen smooth cutoffs, $h$, to compose $R^*$ and $R$ when needed.
That is, we will show the solution is bounded in Sobolev space
order $(n-1)/2$. If such a global Sobolev estimate exists, then we say
the solution is (strongly) \emph{stable} {on $\Omega$.} We now
explore some examples in the following section.

%% file: bolker-S3.2-new.tex
\subsection{Surjectivity of \texorpdfstring{$\Pi_R$}{Pi sub R} and analysis of the
normal operator} In order to better understand the normal
operator, we now discuss properties
of $\Pi_R$.
Note, we still assume $r:\rn\to (0,\infty)$ is smooth and satisfies
\eqref{norm1} on $\rn$ throughout this section.  

We will generalize our setup in a useful way. However, we consider
smaller domains; let $\Omega$ be an open subset of $\rn$ and let
$\CcO$ be the set of covectors above $\Omega$ (see \eqref{def:CcO}).
Let $Y_\Omega$ be the set of projections of $\CcO$ to the first
coordinate, that is \bel{def:YO}\YO = \sparen{\vy\in \rn\st \exists
(\eta,\vx,\xi)\in \drn\times \Omega\times \drn,\ (\vy,\eta,\vx,
\xi)\in \CcO}.\ee This definition, \eqref{def:YO}, implies that $\YO$
contains the centers of all spheres $S(\vy)$ that intersect $\Omega$
since $(\vy,\eta,\vx,\xi)\in \CcO$ if and only if $\vx\in S(\vy)$.
Because $\Pi_L$ and $\Pi_R$ are immersions by Theorem \ref{imm_thm},
$\YO$ is an open set.

\begin{theorem}\label{thm:PiR-atMost2} Let $r:\rn\to (0,\infty)$ be
smooth and satisfy \eqref{norm1} and let $\Omega$ be an open subset of
$\rn$. Let $(\vx,\xi)\in \Pi_R(\CcO)$. Then, there are at most two
preimages in $\CcO$ of $(\vx,\xi)$. For each preimage,
$\lambda=(\vy,\eta,\vx,\xi)$, $\vy$ is on the line
\[\ell_{(\vx,\xi)}=\vx+\rr\xi.\] If there are two preimages then they
are on opposite sides of $\vx$ on $\ell_{(\vx,\xi)}$.
\end{theorem}


Example \ref{ex:CQ} provides a smooth $r:\rn\to (0,\infty)$
satisfying the  inequality \eqref{norm1} for which there are no
preimages of $\Pi_R$ above $\vx = 0$.

\begin{proof} {Let $(\vx,\xi)\in \Pi_R(\CcO)$, then there is at least one
preimage $\lambda_1=(\vy_1,\eta_1;\vx,\xi)$ to $(\vx,\xi)$. Then by
the definition of $\CcO$, \eqref{def:CcO}$, \vx\in S(\vy_1)$ and $\xi$
is normal to $S(\vy_1)$ at $\vx$. This shows $\vy_1$ is on the line
$\ell_{(\vx,\xi)}$.  

Let $\omega = \xi/\norm{\xi}$, and let $t_1\in \rr$ so that
$\vy_1=\vx+t_1\omega$. Assume there is a second preimage with
base point $\vy_2=\vx+t_2\omega$. We will draw a contradiction if
$\vy_1$ and $\vy_2$ are on the same side of $\vx$ on
$\ell_{(\vx,\xi)}$. Without loss of generality, we can assume
$0<t_1<t_2$. Note that $\vy_1$ and $\vy_2$ are in $\YO$ since $\YO$
contains the centers of all spheres $S(\vy)$ that meet $\Omega$.

We let \bel{def:rfr} \rfr:[0,\infty)\to[0,\infty)\ \ \rfr(t) =
r(\vx+t\omega).\ee Note that $\rfr$ is differentiable and that
$\rfr(t_j) = \norm{\vy_j-\vx}=t_j$ for $j=1,2$ since $\vx\in S(\vy_j)$
and $\omega$ is a unit vector.

A Mean Value Theorem argument on $\rfr$  shows for some $c\in
(t_1,t_2)$ that  
\[t_2-t_1=\rfr(t_2)-\rfr(t_1)=\rfr'(c)(t_2-t_1),\] so
$\abs{\rfr'(c)}=1$ and therefore, $\abs{\nablay r(\vx+c\omega)}\geq
\abs{\rfr'(c)}=1$. This contradicts assumption
\eqref{norm1}.}

Note, there clearly can not be any more than two preimages since there are only two sides of $\ell_{(\vx,\xi)}$ about $\vx$.\end{proof}


\noindent Our next theorem gives conditions under which $\Pi_R:\CcO\to
\dot{T^*}(\Omega)$ is surjective.

\begin{theorem}\label{thm:PiR-bound} Let $r:\rn\to(0,\infty)$ satisfy
the norm inequality \eqref{norm1} on $\rn$. Let $\Omega$ be an open
subset of $\rn$.  Assume there is a $C\in (0,1)$ such that the
\emph{strong norm inequality}\bel{strong norm} \norm{\nabla 
r(\vy)}<C\ \forall \vy\in \YO.\ee Let $(\vx, \xi)\in \Omega\times
\drn$, and let $\ell_{\vx,\omega}$ be the line through $\vx$ and
parallel $\xi$. Then there are two points $\vy_j\in
\ell_{\vx,\omega}\cap \YO$, for $j=1,2$ are on opposite sides of $\vx$
and $\vx\in S(\vy_j)$.

There are exactly two preimages
$\lambda_j$, $j=1,2$ of $(\vx,\xi)$ in $\CcO$. Furthermore, for some
$\eta_j\in \drn$, the two preimages, $\lambda_j$, $j=1,2$ satisfy
\[\lambda_j=(\vy_j,\eta_j;\vx,\xi)\in \CcO.\]   Therefore, $\Pi_R$ is
surjective.
\end{theorem}

\noindent Note that the theorem holds for $\Omega  = \rn$ if \eqref{strong
norm} holds on $\rn$.


\begin{proof}
{Let $\Omega$ and $C$ be as in hypotheses of  the theorem and let
$(\vx,\xi)\in \Omega\times \drn$.
Let $\omega = \xi/\norm{\xi}$ and consider the function $\rfr$ given
by \eqref{def:rfr}. Then, $\rfr:[0,\infty)\to[0,\infty)$ is a
contraction mapping by the Chain Rule since $\norm{\nablay r(\vy)}\leq
C$. Since $[0,\infty)$ is complete, by the Contraction Mapping
Theorem, there is a unique $t_1\in [0,\infty)$ such that
$\rfr(t_1)=t_1$, and therefore $\vx\in S(\vx+t_1\omega)$. We let $\vy_1
= \vx+t_1\omega$. Note that $\vy_1\in \YO$ by the definition of $\YO$

We now make a similar argument for the function 
\[[0,\infty)\ni 
t\mapsto
\rfr(\vx-t\omega)\] to get a unique $t_2>0$ and a unique
$\vy_2=\vx-t_2\omega$ for which $\vx\in S(\vy_2)$ and $\vy_2\in \YO$.  

Let $\sigma \in \dR$ satisfy $\xi = \sigma\omega$. Using the
coordinates in \eqref{def:C} and then \eqref{def:CcO}, the points
\[\lambda_j =(\vy_j,\omega,-\sigma/(2r(\vy_j)), \ \ j=1,2\] define the two preimages of
$(\vx,\xi)$, $\lambda_j$, for $j=1,2$, and $\Pi_R$ is surjective and
two-to-one.}
\end{proof}



\begin{theorem}\label{thm:elliptic psido}  Let $r:\rn\to(0,\infty)$ be
smooth and satisfy \eqref{norm1} on $\rn$. Let $\Omega$ be an open
subset of $\rn$. Assume $\Omega$ satisfies the hypotheses of Theorem
\ref{inj_thm} and the strong norm inequality \eqref{strong norm} holds
on $\YO$. Then $R^*R:\Ec'(\Omega)\to \Dc'(\Omega)$ is an elliptic
\psido of order $1-n$.
\end{theorem}

Note that, if the hypotheses of Theorem \ref{thm:elliptic psido} hold,
then $R^*$ can be composed with $R$ without a cutoff. We provide details on this in the proof below.

 \begin{proof} Theorem \ref{thm:PiR-bound} shows that $\Pi_R$ is
surjective under our assumptions. Combining this with the result of
Corollary \ref{cor:bolker} shows that $\Cc_\Omega^t\circ \Cc_\Omega =
\Delta_\Omega$.

Now, we show $R:\Ec'(\Omega)\to\Ec'(\rn)$. This will allow us to
compose $R^*$ and $R$ without a cutoff. Let $\vx\in \Omega$
As shown in the proof
of Theorem \ref{thm:PiR-bound}, for each $(\vx, \omega)\in
\Omega\times \snm$, there is a unique  $t>0$ such that  $\vx\in
S(\vx+t(\vx,\omega)\omega))$. We denote this value by
$t=t(\vx,\omega)$.  Now,  $t=t(\vx,\omega)$ satisfies
\[t-r(\vx+t\omega)=0.\] 
Because of \eqref{strong norm}, \[\frac{\partial}{\partial
t}\paren{t-r(\vx+t\omega)} = 1 - \ip{\nablay r(\vx+t\omega),\omega}
\neq 0
\] and by
the Implicit Function Theorem, $t=t(\vx,\omega)$ must
be smooth. Although the Implicit Function Theorem is local, because
$(\vx,\omega)\mapsto t(\vx,\omega)$ is a well defined function
globally, this function is smooth globally.


Since the function $t(\vx,\omega)$ is continuous, $t(\vx,\omega)$ is
bounded on compact subsets of $\rn\times \snm$. Therefore, if $f\in
\Ec'(\Omega)$, then the set of sphere centers meeting $\supp(f)$,
$\{\vy\st S(\vy)\cap \supp(f)\neq \emptyset\}$, is bounded. This shows
that $R:\Ec'(\Omega)\to \Ec'(\rn)$.

Therefore, $R^*$ can be composed with $R$, and $R^*R:\Ec'(\Omega)\to
\Dc'(\Omega)$. Since $\Pi_R$ is surjective, $\Cc_\Omega^t\circ \Cc_\Omega =
\Delta_\Omega$, and since the weights for $R$ and $R^*$ are smooth and
positive, $R^*R$ is an elliptic pseudodifferential operator (e.g.,
\cite{quinto}).\end{proof}

Note that once we know the normal operator, $R^*R$, is an elliptic
pseudodifferential operator and that $R^*R$ is injective, we can
sometimes prove the reconstruction is stable in Sobolev scales using
Lemma \ref{stef}. Later, in section \ref{examples}, we will prove
such Sobolev estimates for $R$ in examples of interest in Compton CT
and ultrasound imaging.

\begin{example}\label{ex:CQ} We now provide an example where
ellipticity of $R^*R$ fails on $\rn$. Specifically, we provide an
example $r$ in which $\norm{\nabla r(\vy)}<1$ for all $\vy\in \rn$ but
no sphere $S(\vy)$ contains zero. Let $r(\vy) = \sqrt{|\vy|^2 + 1}$.
This provides an example when the hypotheses of Theorem
\ref{thm:PiR-atMost2} hold and there are no preimages of $\Pi_R$ above
$\zero$. This also shows that some stronger hypothesis than
\eqref{norm1} is needed for $R^*R$ to be an elliptic \psido.

Since $\abs{r(\vy)}>\norm{\vy}$ for all $\vy\in \rn$, no sphere
$S(\vy)$ contains $\zero$, so
$\Pi_R^{-1}\paren{\sparen{\zero}\times\drn}=\emptyset$ and $\Pi_R$ is
not surjective.  It is also easy to check that $|\nabla r| < 1$.

This shows that $R^*R$ cannot be an elliptic \psido. If $\vp\in
\Dc(\rn)$ then $R^*\vp R$ is a \psido, but it is not elliptic.

\end{example}

%% file: example4.1.tex
\section{Applications in CST and URT} \label{examples} In this
section, we explore some motivating examples in CST and URT. In each
example considered, we explain the applications of our microlocal
theory and then provide injectivity proofs and inversion formulae. In
some cases, we also supply Sobolev estimates of the form explained
above.

\subsection{Linear translation CST}\label{sec:linearCST} Here, we
consider a variant of the transform proposed in
\cite{webber2024generalized} (see figure 3(a)), which describes the
measurement acquisition for a linear array CST system. The system
geometry is displayed in figure \ref{fig1}. The physical modeling in
2-D CST constrains photon scatter to circular arcs which pass through
the source and detector (see figure \ref{fig1}), and thus, the Compton
scatter intensity can modeled as integrals of $f$ over circular arcs
\cite{RigaudComptonSIIMS2017}. As pictured in figure \ref{fig1}, the
density $f$ is supported on $\{x_2 >2\}$ (above the $x_1$ axis) and
the source and detector array is translated in the $x_1$ direction
(equivalently one could imagine this as translation of $f$).
Meanwhile, $f$ is illuminated by photons emitted from $\vs$ and
measured at $\vd$. The scattered energy determines $y_2$, and the
translation position determines $y_1$, where $\vy = (y_1,y_2)$ is the
circle center. In this case, the radius function is given by $r(\vy) =
\sqrt{y_2^2+\alpha^2}$, where $2\alpha > 0$ is the distance between
$\vs$ and $\vd$, which ensures that $S(\vy)$ intersects $\vs =
(y_1-\alpha,0)$ and $\vd = (y_1+\alpha,0)$. When $n=2$ and $r(\vy) =
\sqrt{y_2^2+\alpha^2}$, $Rf$ models the Compton scatter data measured
by the scanner system in figure \ref{fig1}. Let 
\[\Omega = \rr\times (0,\infty)\]As we assume
$\text{supp}(f)\subset \Omega$, integrating $f$ over the circular arc
$S(\vy)\cap \Omega$ and the circle $S(\vy)$ is equivalent. 
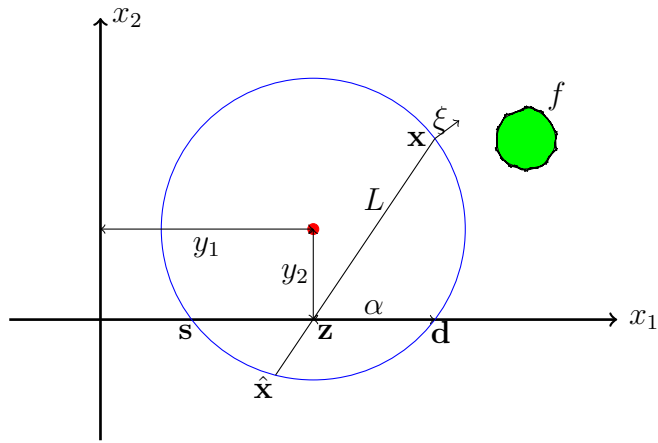
\begin{figure}[!h]
\centering
\begin{tikzpicture}[scale=0.8]
\draw [red,fill=red] (0,1.5) circle (0.09); \draw[fill=green,rounded
corners=1mm] (3.5,3) \irregularcircle{0.5cm}{0.3mm}; \node at (4,3.7)
{$f$}; \draw [->,line width=1pt] (-5,0)--(5,0)node[right] {$x_1$};
\draw [->,line width=1pt] (0-3.5,-2)--(0-3.5,5)node[right] {$x_2$};
\draw [blue] (0,1.5) circle (2.5); \draw [<->]
(0-3.5,1.5)--(3.5-3.5,1.5); \node at (1.75-3.5,1.5-0.3) {$y_1$}; \draw
[<->] (0,0)--(2,0); \node at (1,0.2) {$\alpha$}; \draw [<->]
(0,0)--(0,1.5); \node at (-0.3,0.75) {$y_2$}; \node at (-2.1,-0.2)
{$\vs$}; \node at (2.1,-0.2) {$\vd$}; \node at (1.7,3) {$\vx$}; \node
at (0.2,-0.2) {$\vz$}; \draw (2,3)--(-0.62,-0.92); \node at (1,2)
{$L$}; \node at (-0.82,-1.12) {$\hat{\vx}$}; \draw [->]
(2,3)--(2.4,3.3)node[left]{$\xi$};
\end{tikzpicture}
\caption{Example CST scanner design. Photons scatter on circular arcs
bound by the source $\vs$ and detector $\vd$. } \label{fig1}
\end{figure}

\noindent We consider the $n=2$ case as this is of practical interest
in CST. When $n\geq 3$, the integral surfaces in CST are tori
\cite{webberholman}, and thus do not fit this general framework.

\begin{theorem}
\label{thm_cst_1} Let $r(\vy) = \sqrt{y_2^2+\alpha^2}$ and let $\Omega
= \rr\times (0,\infty)$. Then, $R:\Ec'(\Omega)\to \Dc'(\rtwo)$ is an
elliptic FIO which satisfies the Bolker condition.
\end{theorem}
\begin{proof}
We have, $\nabla_{\vy}r = (0,y_2/\sqrt{y_2^2+\alpha^2})$, and
$|\nabla_{\vy}r| = |y_2|/\sqrt{y_2^2+\alpha^2}<1$. Thus, $R$ is an
elliptic FIO by Proposition \ref{fio_prop} and $\Pi_L$ is an immersion
by Theorem \ref{imm_thm}. The point $\vz = \vy - r\nabla_{\vy}r =
(y_1,0)$, and thus if $\vx \in \Omega$, the corresponding artifact
locations are $\hat{\vx} \in \{x_2 < 0\}$, as illustrated in figure
\ref{fig1}. Therefore, $\mathcal{A}(\Omega) \subset \{x_2 <0\}$. As
$\Omega= \{x_2>0\}$, $\Omega\cap \mathcal{A}(\Omega) = \emptyset$ and
thus $\Pi_L$ is also injective by Proposition \ref{prop_geo}.
\end{proof}

\begin{remark}
If we constrain $y_2 < 0$ (i.e., we focus only on scatter at $<
90^{\circ}$, or forward scatter), then Theorem \ref{thm_cst_1} follows
from the generalized theory of \cite{webber2024surface}, as when $y_2
< 0$, the circular arcs $S(\vy) \cap \{x_2 > 0\}$ are graphs. Theorem
\ref{thm_cst_1} and the microlocal theory presented here allows us to
consider all $y_2 \in \mathbb{R}$ (i.e., both forward and backscatter
simultaneously). This has further implications on detection of
wavefronts and stability, which we now explore.
\end{remark}

\begin{theorem}
\label{wave_cst_1} Let $n = 2$ and $r(\vy) = \sqrt{y_2^2+\alpha^2}$.
Then $\Pi_R:\Cc_\Omega\to T^*(\Omega)\smo$ is surjective.

If $\xi =(0,\xi_2),\ \xi_2\neq 0$, then for every $\vx\in \Omega$,
there is one preimage of $(\vx,\xi)$, and there two preimages for
every other covector in $\dot{T^*}(\Omega)$.
\end{theorem}

\begin{proof}   
From \eqref{def:C}, we see that $\Pi_R:\Cc_\Omega\to T^*(\Omega)$ is
given in coordinates by 
\[\rtwo\times S^1\times \dR\ni(\vy,
\omega,\sigma)\mapsto\paren{\vy,-2\sigma r(\vy)(\omega + \nabla 
r);\vy+r(\vy)\omega, -2\sigma r(\vy)\omega}.\]

We parameterize $S^1$ by
\[\xi=\xi(\varphi) = (\cos (\varphi),\sin(\varphi))\] for $\varphi \in
[-\pi/2,3\pi/2]$.

Let $(\vx,\xi)\in \Omega\times S^1 $. If $\xi$ is an outer normal at
$\vx$ to a circle, the circle can be written $S(\vx-t\xi)$ for some
$t>0$, and we now solve for $t>0$ for which this circle is in our data
set:
\begin{equation}
\label{quad} \vx-t\xi = \paren{y_1,\sqrt{t^2-\alpha^2}}.
\end{equation} For any $\vx\in\{x_2>0\}$ and $\xi \in S^1\backslash
\{(0,-1)\}$, we can solve the quadratic in the second argument of
\eqref{quad} 
\begin{equation}\label{def:t}
t(\varphi,x_2) =
\frac{\alpha^2+x^2_2}{ x_2\sin(\vp) +
\sqrt{\alpha^2\cos^2(\vp)+x_2^2 }}.
\end{equation}
When $\varphi = 3\pi/2$ then there is no solution for $t(3\pi/2,\vx)$.
This is clear from \eqref{def:t} and geometrically because any circle
containing a point $\vx\in \Omega$ with outer normal $(-1,0)$ will be
disjoint from the horizontal axis.

Therefore, for $\xi\neq (0,-1)$, $\xi$ is the outer unit normal of the
circle in the data set with center $\vy = \vx -t(\vp,\vx)\xi$ and
radius $r(\vy)$. Using \eqref{def:C}, we see that \[(\vy,-2\sigma
r(\vy)(\xi+\nabla_\vy r); \vx,\xi)\] is one preimage of $(\vx,\xi)$ in
$\Cc(\Omega_0)$ when $\sigma = -1/2 r(\vy)$. However, because there is
no circle in the data set with outer normal $(0,-1)$, there is no
preimage of $(\vx,(0,-1))$ in which $(0,-1)$ is an outer normal to the
circle.

If we do the same calculation for inner normals, we solve $\vy =
\vx+\tilde{t}(\vp,\vx)$ for $\tilde{t}>0$ and see for $\vp\in
(-3\pi/2,\pi/2)$ that $\tilde{t}(\vp, \vx) = t(\vp +\pi,\vx)$. A
similar argument to the one for outer unit normals shows that for
every $\xi\in S^1$ except $(0,1)$, there is a preimage of $\Pi_R$ for
which $\xi$ is an inner normal to the circle with center $\vy =
\vx+\tilde{t}(\vp,\vx)$. 

Finally, since $\Cc_\Omega$ is a conic set, this
shows that $\Pi_R(\Cc_\Omega)=T^*(\Omega)\smo$. This completes the
proof.\end{proof}

 \begin{remark}\label{rem:not strong norm} We note that since
$\lim_{|y_2|\to \infty}|\nabla_\vy r| = \lim_{|y_2|\to
\infty}|y_2|/\sqrt{y_2^2+\alpha^2} = 1$, the strict norm inequality
\eqref{strong norm} does not hold. We have thus addressed detection of
wavefronts and surjectivity of $\Pi_R$ in the $n=2$ case using Theorem
\ref{wave_cst_1}. Also, see Remark \ref{rem:strongish norm} below.
\end{remark}

\begin{corollary}
Let $n = 2$ and $r(\vy) = \sqrt{y_2^2+\alpha^2}$ with $f\in
\Ec'(\Omega)$. Then, the reconstruction of $f$ from $Rf$ is weakly stable
everywhere in $\Omega$.
\end{corollary}
\begin{proof}
This follows immediately from Theorems \ref{thm_cst_1} and
\ref{wave_cst_1} and Definition \ref{stable_def}.
\end{proof}

Let $b>0$ be arbitrary and let $a\in (0,b)$. We consider functions
supported in
\bel{def:Om0}\Omega_0 =
\{(x_1,x_2)\in \rtwo\st a< x_2 <b\}.\ee Since we are interested in
compactly supported functions and distributions in $\Omega=\{x_2>0\}$,
any such function will be supported in $\Omega_0$ for some choice of
$a,b$.


To ensure that our transform maps $\Ec'(\Omega_0)$ to $\Ec'(\rtwo)$,
we include a cutoff for $R$. Let $\eps>0$ and let $h:\rr\to[0,1]$ be a
smooth function supported in $(-\infty,b+\eps/2)$ and equal to one on
$(-\infty,b+\eps/4]$. For $f\in \Ec'(\Omega_0)$ let \bel{def:tR} R^\dagger
f(\vy) = h(y_2)Rf(\vy).\ee 
Now, let \bel{def:omega1}\Omega_1 =
\sparen{(y_1,y_2)\in \rtwo\st y_2<b+\eps},\ee then in Theorem
\ref{ell_thm_1}, we will show $R:\Ec'(\Omega_0)\to\Ec'(\Omega_1)$.

We let \bel{def:lambda}\lambda : (-\pi/2,3\pi/2) \times\{x_2>0\}\to
\mathbb{R}^2, \qquad\lambda(\varphi,\vx) = \vx - t(\varphi,x_2)\xi,\ee
where $t(\varphi,x_2)$ is defined in \eqref{def:t}. Then,
$\lambda(\varphi, \vx)$ is the center of the circle in our data set
with outward unit normal $\xi(\varphi)$ at $\vx$. In figure
\ref{fig2.1}, we give example plots of $\lambda([0,\pi], \vx)$ for
various $\vx$ when $\alpha = 1$. In the proof of Theorem
\ref{ell_thm_1}, we show that the image on any open neighborhood of
$[0,\pi]$ is sufficient to cover all wavefronts at $\vx$. 
 
We define a backprojection operator \bel{def:R*} R^*g(\vx) =
\int_{-\pi/2}^{3\pi/2}
g\paren{\lambda(\varphi,\vx)}\mathrm{d}\varphi\ee In general, this is
not well-defined, because $\vy(\vp,\vx)\to \infty$ as $\vp\to
-\pi/2,3\pi/2$. However, we will show in Theorem \ref{ell_thm_1},
$R^*:\Ec'(\Omega_1)\to \Dc'(\Omega_0)$.

Now, we define the \emph{smoothed normal
operator}
\begin{equation}\label{def:N}
\mathcal{N}f = R^* R^\dagger f.
\end{equation}


In our next theorem, we will show that $R^*$ and $R^\dagger$ can be composed
and we will provide other important properties of $\Nc$.

%
%


\begin{theorem}
\label{ell_thm_1} Let $n=2$. Then, $R^\dagger:\Ec'(\Omega_0)\to
\Ec'(\Omega_1)$, and $R^*:\Ec'(\Omega_1)\to \Dc'(\Omega_0)$. 

 The smoothed normal operator,
$\mathcal{N}:\Ec'(\Omega_0)\to\Dc'(\Omega_0)$, is an elliptic
pseudodifferential operator (\psido) of order $-1$, and
$\Nc:\Ec'(\Omega_0)\to \Dc'(\Omega_0)$

Furthermore, if $B$ is a compact subset of $\Omega_0$ then there is a
compact $B'\subset\Omega_1$ such that if $\supp(f)\subset B$ then
$\supp(R^\dagger f)\subset B'$.
\end{theorem}

\begin{proof}  
 The FIO $R^\dagger$ has the same phase as $R$ since multiplication by $h$
(a smooth function of $\vy$) affects only the amplitude of $R$.
Therefore, $R^\dagger$ and $R$ have the same canonical relation, $\Cc$. This
also implies, $R^*$ is an FIO (note that $R^*g (\vx)$ is an integral
of $g$ over all circles in the data set containing $\vx$, so it fits
the framework of the dual operator in \cite{GS1977, quinto}). 

We now prove the mapping properties of our operators. First, we write
$\lambda(\vp,\vx) = (y_1,y_2)$ where $\lambda$ is defined in
\eqref{def:lambda}. Note that $\lambda$ is a smooth function on
$(-\pi/2,3\pi/2)\times \Omega$ and $y_2 = y_2(\vp,x_2)$ is independent
of $x_1$. A straightforward calculation using the last expression in
\eqref{def:t} as well as \eqref{def:lambda} shows that for each
$x_2>0$, $y_2(\cdot,x_2)$ has global minimum for $\vp\in
(-\pi/2,3\pi/2)$ at $\vp = \pi/2$. Then, using the expression
${t\paren{\pi/2,x_2} = \frac{\alpha^2+x_2^2}{2x_2},}$
we see $y_2(\vp,x_2)$ has a global minimum on
$(-\pi/2,3\pi/2)\times [a,b]$. A simple geometric argument shows it is
\[\lambda(\pi/2,a) =-\frac{\alpha}{a}=:m.\]

First, we explain why $R^\dagger = hR:\Ec'(\Omega_0)\to\Ec'(\Omega_1)$.
By the choice of $m$ and $h$, the only circles in the data set with
center in $\Omega_1$ that meet $\Omega_0$ have $y_2\in[m,b+\eps/2]$,
so \bel{def:Sc}\supp(R^\dagger f)\subset \Sc=\{(y_1,y_2)\in \rtwo\st m\leq
y_2\leq b+\eps/2\}.\ee

We claim there is an $L>0$ such that for $\vy\in\Sc$, every circle in
the data set centered at $\vy$ has radius bounded above by $L$. The
reason is that for $\vy\in \Sc$, $y_2\in [m,b+\eps/2]$ and for any
$\vy$, $t^2 = y_2^2+\alpha^2$, and $t$ is the radius of the circle in
the data set centered at $\vy$.

Let $f\in \Ec'(\Omega_0)$. Now, since $f $ has compact support in
$\Omega_0$, then $R^\dagger f$ must have compact support as every circle in
the data set for $\vy\in \Sc$ that meets $\vx$ has radius no larger
than $L$. Therefore, $R^\dagger$ is an FIO, and $R^\dagger:\Ec'(\Omega_0)\to
\Ec'(\Omega_1)$.

Let $B$ be a compact subset of $\Omega_0$, then the argument above
provides a subset, $B'$ of $\Omega_1$ such that if $\supp(f)\subset B$
then $\supp(R^\dagger f)\subset B'$. Finally, $B'$ is compact because $h$ is
supported in $(-\infty, b+\eps/2]$. 

To finish characterizing mapping properties, we show
$\Nc:\Ec'(\Omega_0)\to \Ec'(\rtwo)$. Let $f\in \Ec'(\Omega_0)$. We
have shown that $\supp(R^\dagger f)\subset \Sc$. A straightforward geometric
exercise using the definition of the circles in the data set (see
Figure \ref{fig1}) shows that $\supp(\Nc f)\subset \Sc'$ where
\bel{def:Sc'}\Sc' = \sparen{(x_1,x_2)\in \rtwo\st 2m-a\leq x_2\leq
b+\eps+\sqrt{(b+\eps)^2+\alpha^2}},\ee and an argument similar to the
one showing $R^\dagger f$ has compact support can be used here to show that
$\Nc f$ has compact support.

Now, we show that $\Nc:\Ec'(\Omega_0)\to \Dc'(\Omega_0)$ is an
elliptic \psido. First, $\Nc:\Ec'(\Omega_0)\to \Dc'(\Omega_0)$ is a
\psido since the canonical relation of $\mathcal{N}$ is $\Cc^t\circ
\Cc$ which is a subset of the diagonal by Theorem \ref{thm_cst_1}. The
order of $\mathcal{N}$ is $-1$ since the order of $R^\dagger$ and $R^*$ is
$-1/2$. 

 This proof of ellipticity of $\Nc$ follows the same arguments as the
symbol calculations in \cite[pp.\ 337-338]{quinto} and \cite[Section
5.2]{GKQR:ma}. 

 Let $(\vx,\xi)\in T^*(\Omega_0)\smo$. We calculate the symbol of
$\Nc$ at $(\vx,\xi)$. Since wavefront sets are homogeneous sets, we
will assume $\xi$ is a unit vector. 

We consider three cases.

\textit{Case 1:} Assume $\xi = \xi(\vp)$ for some $\vp\in
(-\pi/2,\pi/2)\cup(\pi/2,3\pi/2)$, i.e., $\xi\neq (0,\pm 1)$.

Then, there are two circles in the data set that are normal to
$(\vx,\xi)$ and their centers are: \bel{def:ys}\vy_1 =
\gamma(\vp,\vx),\ \ \vy_2 = \gamma(\vp-\pi,\vx)\ee where, if
$\vp-\pi<-\pi/2$, we take angle $\vp+\pi$ for $\vy_2$. There are two
points in $C$ above $(\vx,\xi)$: \bel{preimages} \tau_1 =
(\vy_1,r(\vy_1), \eta_1,\vx,\xi),\ \ \tau_2 = (\vy_2,r(\vy_2),
\eta_2,\vx,\xi)\ee where $\eta_j$ is calculated using \eqref{def:C}.

To calculate the symbol of $R^*R^\dagger$ at $(\vx,\xi)$, we pull back
$(\vx,\xi)$ to $C$ using $\Pi_R$, getting the two preimages $\tau_1$
and $\tau_2$. Since the measure defining $R^*$ is positive, the symbol
of $R^*$ at each of these points is positive.

To get the symbol of $R^*R^\dagger$ at $(\vx,\xi)$ we multiply the symbols
of $R^*$ at $\tau_1$ and $\tau_2$ by the symbols of $R^\dagger$ at these
points. Since $R$ is elliptic with positive measure and $h\geq 0$, the
symbol of $R^\dagger$ is nonnegative at both points. Since $x_2<b$ and
$h(y_2)=1$ for $y_2\leq b$, at least one of $h(\vy_1)$ and $h(\vy_2)$
is equal to one so the sum of the symbols of $R^\dagger$ at $\tau_1$ and
$\tau_2$ is positive. Since the sum is positive $\Nc$ is elliptic
above $(\vx,\xi)$.

\textit{Case 2:} Assume $\xi = (0,1)=\xi(\pi/2)$ there is only one
preimage of $\xi$, $\tau_1$, and a similar proof is done but without
the addition of $\tau_2$.

\textit{Case 3:} If $\vp_0$ equals $-\pi/2$ or $3\pi/2$ then we do the
calculation from \textit{Case 2}, writing $(\vx,\xi) =
(\vx,-\xi(\pi/2))$, and that just multiplies the $\eta$ component by
$(-1)$. Therefore, $\Nc$ is an elliptic \psido. This finishes the
proof.
\end{proof}

 \begin{remark}\label{rem:strongish norm} Note that Theorem
\ref{ell_thm_1} does not follow directly from of Theorem
\ref{thm:elliptic psido} since the $R^*$ is not the dual of $R^\dagger$, but
\eqref{strong norm} can be used in parts of the proof since the
circles with centers in $\Omega_1$ that meet $\Omega_0$ have bounded
$y_2$ coordinate so the Contraction Mapping Theorem can be used for
them.\end{remark}

\subsubsection{Injectivity} The injectivity of $R$ in the special case
when $n=2$ and $r(\vy) = \sqrt{y_2^2+\alpha^2}$ is proven in
\cite[Theorem 4.5]{web} using Volterra integral equation theory. See
the example in \cite[figure 3(a)]{web}. Specifically, we have the
theorem:

\begin{theorem}
\label{inj_lin} Let $f\in L^2_c(\Omega_0)$, and $r(\vy) =
\sqrt{y_2^2+\alpha^2}$. If $Rf(\vy) = 0$ for $\vy\in \{\frac{a^2 -
\alpha^2}{2a} \leq y_2 \leq \frac{b^2 - \alpha^2}{2b}\}$, then $f=0$.
\end{theorem}

%
%

%

%


We use these theorems to show that reconstruction from $R$ is stable
in the following sense.

\begin{corollary}
\label{corr_stab_1} Let $B$ be a compact 
subset of $\Omega_0$, and let $s\in \rr$. Let $f\in H^s(B)$. Further,
let $B_1$ and $B_2$ be compact sets that satisfy $B\subset
\intt(B_1),\ B_1\subset \intt(B_2)$, and $B_2\subset \Omega_0$. Let
$\phi_1$ and $\phi_2$ be in $\Dc(\Omega_0)$ with $\phi_1 = 1$ on $B$,
$\supp(\phi_1)\subset B_1$, and $\phi_2=1$ on $B_1$, and
$\supp(\phi_2)\subset B_2$. Then, $\mathcal{N}' = \phi_2 \mathcal{N}
\phi_1$ has a stable inverse in $H^s(B)$ of order $1$. That is, there
exists $C >0$, with
\begin{equation}
\left\|f\right\|_{H^s(B)} \leq
C\left\|\mathcal{N}'f\right\|_{H^{s+1}(B_2)}.
\end{equation}
\end{corollary}


\begin{proof}
Let $\Pc$ be an order $-1$ \psido that is a parametrix of $\Nc$ on
$\Omega_0$. Note, $\Pc$ exists by Theorem \ref{ell_thm_1}. Without
loss of generality, one may assume $\Pc$ is properly supported. Then,
for $f\in \Ec'(\Omega_0)$ one can write \bel{parametrix1}\Pc\Nc f = I
f+ K f\ee for some smoothing operator $K$.

Let $f\in \Ec'(\Omega_0)$. As $f$ supported in $B$, we have $\phi_1 f
= f$ and $\phi_2\Pc\Nc \phi_1 f$ is equal to $f$ modulo smoothing
operators. By pseudolocality of \psido, one can write
\eqref{parametrix1} as
\begin{equation}
\label{parametrix2}
\begin{split}
\Pc'\Nc' f &=
f + K'f,\ \ \text{where}\\
\Pc' = \paren{\phi_2\Pc }, \ \
\Nc'&= \paren{\phi_2\Nc\phi_1},\ \ K' =
\phi_2K\phi_1+\phi_2\Pc(\phi_2^2-1)\Nc\phi_1,
\end{split}
\end{equation}
and  $K'$ is smoothing for distributions supported on $B$ since
$(\phi_2^2-1)$ is zero on $B$. 

Since $\Pc'$ and $\Nc'$ are properly supported \psido, for each
$s\in \rr$, \newline $\Pc':H^s(B)\to H^{s-1}(B_2)$
and $\Nc':H^s(B)\to H^{s+1}(B_2)$, and
$K':H^s(B)\to H^{s}(B_2)$ are continuous and $K'$ is compact since it
is smoothing for $f\in H^s(B)$.


By the triangle inequality, for some $C>0$, we have
\begin{equation}
\begin{split}
\|f\|_{{H^s(B)}}  &\leq
\|\Pc'\mathcal{N}'f\|_{H^s(B_2)} + \|K'f\|_{H^s(B_2)} \\
&\leq C\paren{\|\mathcal{N}'f\|_{H^{s+1}(B_2)} + \|K'f\|_{H^s(B_2)}}. \\
\end{split}
\end{equation}
We now use Lemma \ref{stef} to prove the result. Since $\mathcal{N}'$
is continuous, it is a closed operator. As noted, $K'$ is smoothing for
distributions supported on $B$ and is thus compact from $H^s(B)$ to
$H^s(B_2)$. It remains to prove that $\mathcal{N}'$ is injective from
$H^s(B)$ to $H^s(B_2)$. Let $f\in H^s(B)$ and $\mathcal{N}'f = 0$.
Then, $\mathcal{N}f = 0$ on $B_1$ given the support of $\phi_2$. Since
$\supp(f)\subset B_1$,
\begin{equation}
0 = \langle f, \mathcal{N}f\rangle_{H^s(\mathbb{R}^2)} =
\left\|\sqrt{h} Rf\right\|^2_{H^s(\mathbb{R}^2)}.
\end{equation}
 This implies $Rf = 0$ on $\{-\infty<y_2<b+\eps/4\}$. 
 
{Let $f\in
H^s(B)$. Since $\supp(f)\subset B$, $\Nc$ is elliptic, and
$\Nc f = 0$ on $B_1$, $f$ is smooth. It follows that $f=0$ by Theorem
\ref{inj_lin}, since $b > \frac{b^2 - \alpha^2}{2b}$ and $f\in L^2_c(\Omega_0)$. Thus,
$\mathcal{N}'$ is injective on domain $H^s(B)$ and this allows us to
use Lemma \ref{stef}, and this completes the proof.}\end{proof}



%% file: example4.2-v2.tex
\subsection{A novel rotational CST geometry}\label{sec:circularCST}

In this section, we apply our microlocal theory
to a new 2-D scanning geometry in CST. Consider the circular CST
geometry in figure \ref{fig2*}. The source and detector ($\vs$ and
$\vd$) lie at opposite ends of a line segment, length $2\alpha$ for
some $\alpha>0$, which is tangent to $S^1$. The line segment is
rotated (staying tangent to $S^1$) about the origin to generate data.
In this case, the CST data determines the integrals of $f$ over
circles which intersect both $\vs$ and $\vd$ and $f$ is assumed to be
supported on the interior of $S^1$, i.e., the open unit ball in
$\mathbb{R}^2$. The proposed geometry is somewhat analogous to that of
\cite{ABKQ2013}, although in that paper the integral curves are
ellipses and the application is ultrasound tomography.
\begin{figure}[!h]
\centering
\begin{tikzpicture}[scale=0.8]
\draw[fill=green,rounded corners=1mm] (1,1) \irregularcircle{0.5cm}{0.5mm};
\node at (0.5,1.5) {$f$};
\draw [->,line width=1pt] (-6,0)--(6,0)node[right] {$x$};
\draw [->,line width=1pt] (0,-6)--(0,6)node[right] {$y$};
\draw [thick] (0,0) circle (4);
\draw [<->] (0,0)--(-2.83,2.83);
\node at (-1.415,1.8) {$1$};
\node at (-4.1,-0.2-4) {$\vs$};
\node at (4.1,-0.2-4) {$\vd$};
\draw [blue] (0,-2) circle (4.4721);
\draw [red,fill=red] (0,-2) circle (0.1);
\draw [<->] (-4,-4)--(0,-4);
\node  at (-2,-4.2) {$\alpha$};
\draw [<->] (4,-4)--(0,-4);
\node  at (2,-4.2) {$\alpha$};
\node at (-0.3,-2) {$\vy$};
\draw [<->] (-4,-4)--(0,-2);
\node at (-2.25,-2.7) {$r(\vy)$};
\node at (0.25,-4.25) {$\vz$};
\end{tikzpicture}
\caption{Circular CST scanner design. For $\vy\neq \zero$, $\vz =
\frac{\vy}{\abs{\vy}}$.}
\label{fig2*}
\end{figure}
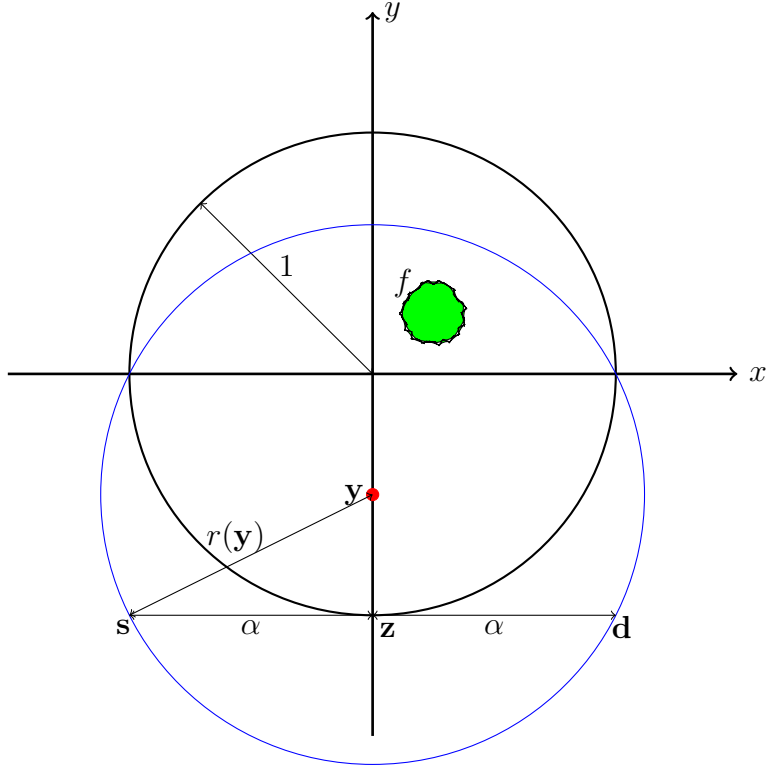

In this example, we have
\begin{equation}
\label{r_circ}
r(\vy)  = \sqrt{\alpha^2 + (1-|\vy|)^2}.
\end{equation}
Then, $Rf(\vy)$ models the intensity of Compton scattered photons in
the geometry of figure \ref{fig2*}.

These circles can be defined in terms of $\vy\neq 0$. Let $\vz =
\vy/\abs{\vy}$ and let $\ell$ be the line tangent to $S^1$ at $\vz$.
This determines $\vs$ and $\vd$, the two points on $\ell$ of distance
$\alpha$ from $\vz$ depicted in figure \ref{fig2*}, and this
determines the circle $S(\vy)$ centered at $\vy$ and containing $\vs$ and
$\vd$ in the data set
for this problem.

As the expression \eqref{r_circ} is valid for $n\geq 1$, we will work
in $n$-dimensions to keep the discussion more general, although the
case of interest in Compton tomography is $n=2$ since {such} Compton
transforms in $\rr^3$ would not be modeled by spherical integrals. {The integral surfaces in $\rr^3$ are tori \cite{rigaud20183d}.} To
state our theorems, we define the following sets that are analogous to
the sets in section \ref{sec:linearCST}. Define \bel{sets4.2} D =
\sparen{\vx\in \rn\st \abs{\vx}<1}, \ \ Y = \sparen{\vy\in \rn\st
\abs{\vy} >\frac{\alpha^2}{4}}, \ee then $D$ is the open unit ball,
and $Y$ is the set of all spheres $S(\vy)$ that intersect $D$.

We now have the theorem:

\begin{theorem}
\label{thm_cst_2}
Let $r(\vy)  = \sqrt{\alpha^2 + (1-|\vy|)^2}$. Then,
$R:\Ec'(D)\to \Dc'(Y)$ is an elliptic FIO which satisfies the Bolker
condition.
\end{theorem}
\begin{proof}
We have,
\begin{equation}
\nabla_{\vy}r = -\frac{\vy}{|\vy|}\cdot \frac{(1-|\vy|)}{\sqrt{\alpha^2 + (1-|\vy|)^2}},
\end{equation}
and
\[|\nabla_{\vy}r| = \frac{|1-|\vy||}{\sqrt{\alpha^2 + (1-|\vy|)^2}}
<1.\]
Thus, $R$ is an elliptic FIO by proposition \ref{fio_prop}, and the
left projection of $R$, $\Pi_L$, is an immersion by Theorem
\ref{imm_thm}. Furthermore, if $\vy\in Y$, we have that the point $\vz$ in the
statement of Theorem \ref{thm:PiL-2-to-1}
$$\vz = \vy - r(\vy)\nabla r(\vy)= \frac{\vy}{|\vy|} \in S^{n-1}.$$
Therefore for any point $\vx\in S(\vy)\cap D$, the artifact point in
\eqref{def:hvx}, $\hvx(\vy,\vx)$, the other point on the line
containing $\vx$ and $\vz$, is outside $D$.
Therefore, restricted to $\Cc_D$, $\Pi_L$ is also injective by Theorem
\ref{inj_thm}, and the Bolker condition is satisfied.
\end{proof}



Let $b\in (0,1)$ and define 
\bel{def:DbYb} D_b = \sparen{\vy\st \norm{\vy}<b}, \ \ Y_b = \sparen{\vy\st
\frac{\alpha^2}{4}<\norm{\vy}< \frac{\alpha^2+1-b^2}{2(1-b)}},\ee and
$Y_b$ is the set of sphere centers in the data set that meet $D_b$,
that is the set $Y_{D_b}$ given  in Definition \ref{def:YO}  and used
in Theorem \ref{thm:elliptic psido}

 We now show that $R^*R$ is an elliptic \psido.

\begin{theorem}\label{PDO_cst_2}
Let $r(\vy) = \sqrt{\alpha^2 + (1-|\vy|)^2}$. Then the normal
operator $R^*R:\Ec'(D)\to \Dc'(D)$ is an elliptic pseudodifferential
operator of order $-1$.
\end{theorem}

\begin{proof}  
Let $b\in (0,1)$.  As noted,  $Y_b$ is the set
$Y_{D_b}$ in \eqref{def:YO}. Since $Y_b$ is bounded and $r$ is smooth
on $\rn$, there is a $C>0$ such that $\norm{\nabla r(\vy)}<C$ for all
$\vy\in Y_b$. 

Therefore, we can apply Theorem \ref{thm_cst_2} and Theorem
\ref{thm:elliptic psido} to infer $R^*R:\Ec'(D_b)\to \Dc'(D_b)$ is an
elliptic \psido of order $-1$. Then, $R^*R:\Ec'(D)\to \Dc'(D)$ is an
elliptic \psido because every distribution in $\Ec'(D)$ is in
$\Ec'(D_b)$ for some $b\in (0,1)$.
\end{proof}



 \begin{remark}\label{rem:center points} A calculation can be used to
directly prove that $\Pi_R$ is surjective. Let $\vx \in D$ and $\xi\in
\drn$. Let $S(\vx + t\xi)$ intersect $\vx$ for some $\xi \in S^{n-1}$
and $t \in \mathbb{R}$. Then,
\begin{equation}
|\vx - (\vx + t\xi)|^2 = r(\vx + t\xi)^2.
\end{equation}
It follows that
\begin{equation}
\begin{split}
t^2 &= \alpha^2 + (1-|\vx+t\xi|)^2\\
&= \alpha^2 + 1 - 2|\vx+t\xi| + |\vx|^2 +2t(\vx \cdot \xi) + t^2\\
&\hspace{-1cm}\implies 4\paren{ |\vx|^2 + 2t(\vx \cdot \xi) + t^2 } = \paren{ 2t(\vx \cdot \xi) + (1+\alpha^2 +|\vx|^2) } ^2.
\end{split}
\end{equation}
Rearranging this yields the quadratic to solve for $t$
\begin{equation}
\label{quad2} p(t) =4\paren{ 1 - (\vx\cdot \xi)^2 } t^2 + 4(\vx \cdot
\xi)\paren{ 1 - (\alpha^2 + |\vx|^2) }t + 4|\vx|^2 - (1+\alpha^2
+|\vx|^2) ^2 = 0.
\end{equation}
Then, a calculus exercise shows there are two solutions for $t$, one
positive, and one negative, and they define centers of the two spheres
in $Y$ that are normal to $(\vx,\xi)$.

This parameterization allows one to find the sphere centers for each
$S(\vy)$ that contains $\vx$. As $R$ is a Radon transform,  $R^*g(\vx)$ integrates over
this set of spheres. See Figure \ref{fig2.1} for an example when $n=2$. 
\end{remark}

\begin{figure}[!h]
\centering \includegraphics[width=0.9\linewidth, height=5cm,
keepaspectratio]{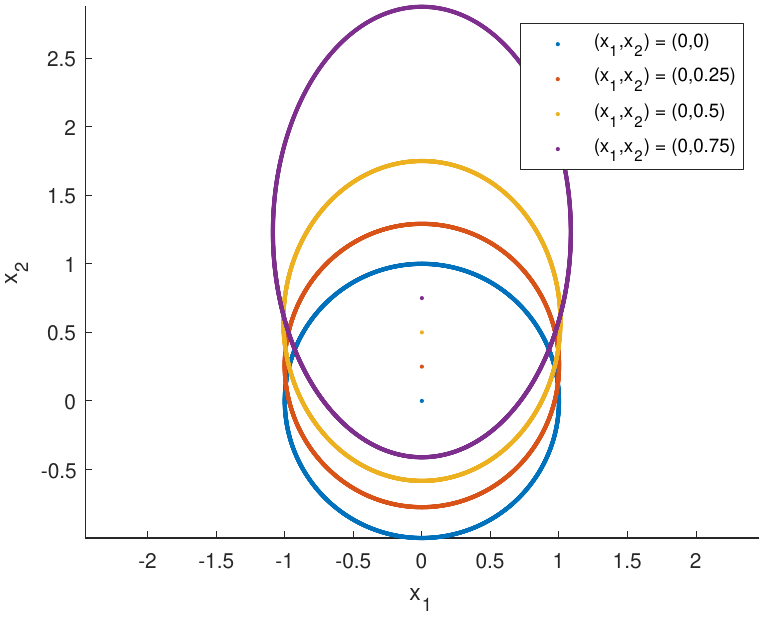} \caption{Plot of circle centers, $\vy$
that $R^*$ integrates over for varying $\vx$. Here, $n=2$. The
colored dots are the $\vx$ coordinates of the points listed in the figure, and
the corresponding colored  curves are the set of $\vy$ such that $S(\vy)$
intersects $\vx$. }
\label{fig2.1}
\end{figure}

\subsubsection{An inversion formula for this spherical transform in
$\rn$} In this section, we apply the generalized theory of
\cite{palamodov2012uniform} to derive inversion formulae for the
spherical transform in $\rn$, $R$, when $r(\vy) = \sqrt{\alpha^2 +
(1-|\vy|)^2}$ in $\rn$. This is, of course, the generalization of our
Compton transform in $\rtwo$, although it does not model Compton data
for $n>2$. 
To
apply the theory of \cite{palamodov2012uniform}, we will adapt our
notation slightly to align with that of \cite{palamodov2012uniform}.
Let us write $\vy = t \Theta$, where $t\in \mathbb{R}$ and $\Theta \in
S^{n-1}$. The spheres of integration in $\rn$ have the defining
equation
\begin{equation}
|\vx - t\Theta|^2 = \alpha^2 + (1-t)^2,
\end{equation}
which is equivalent to 
\begin{equation}
\lambda: = \frac{\sqrt{1+\alpha^2}}{2t} = \frac{p - (p\vx)\cdot
\Theta}{1-|p\vx|^2},
\end{equation}
where $p = 1/\sqrt{1+\alpha^2}$. Here, $\lambda,\Theta$ vary the
sphere centers and radii. For this section, we let $D$ be the open
unit ball in $\rn$ and 
$0 = \lambda - \Psi_1(\vx,\Theta) = \lambda - \frac{p - (p\vx)\cdot
\Theta}{1-|p\vx|^2}$ be the defining equation of our spheres.

Let $\Psi_2(\vy,\Theta) = \frac{p - \vy\cdot \Theta}{1-|\vy|^2}$.
Then, from \cite[Section 6]{palamodov2012uniform}, we have the definition of the
equidistant sphere transform
\begin{equation}
\mathcal{R}f(\lambda,\Theta) = \int_{\mathbb{R}^n}|\nabla_{\vy}\Psi_2|\delta\paren{ \lambda - \Psi_2(\vy,\Theta) }f(\vy)\mathrm{d}\vy
\end{equation}
In \cite[Section 6]{palamodov2012uniform}, an explicit left inverse for $\mathcal{R}$,
which we denote $\mathcal{R}^{-1}$, is provided when $0 \leq p < 1$
and $\text{supp}(f) \subset D$. The formula uses filtered
backprojection in a similar vein to the classical filtered
backprojection formula for the hyperplane Radon transform. For the
full expression for $\mathcal{R}^{-1}$, we refer the reader to \cite[Section 6]{palamodov2012uniform}. We now prove that $R$ and $\mathcal{R}$ are equivalent via diffeomorphism.

\begin{theorem}
\label{inv_cst_2}  Let $f$ be supported in $D$. 
Let $r(\vy)  = \sqrt{\alpha^2 + (1-|\vy|)^2}$. Then, 
\begin{equation}
Rf(\lambda,\Theta) = \frac{1}{p^{n-1}}\mathcal{R}\tilde{f}(\lambda,\Theta),
\end{equation}
where $\tilde{f}(\vy) = f\paren{ \frac{\vy}{p} }$.
\end{theorem}

\begin{proof}
Using the notation of \cite[Section 6]{palamodov2012uniform}, the defining function for the spheres of integration is
\begin{equation}
\Phi(\vx; \lambda,\Theta) = \lambda - \frac{p - (p\vx)\cdot \Theta}{1-|p\vx|^2}.
\end{equation}
Let $\Psi_1(\vx,\Theta) = \frac{p - (p\vx)\cdot \Theta}{1-|p\vx|^2}$ and $\Psi_2(\vy,\Theta) = \frac{p - \vy\cdot \Theta}{1-|\vy|^2}$. Then, we have
\begin{equation}
\begin{split}
Rf(\lambda,\Theta) &= \int_{\mathbb{R}^n}|\nabla_{\vx}\Psi_1|\delta\paren{ \lambda - \Psi_1(\vx,\Theta) }f(\vx)\mathrm{d}\vx\\
&= p\int_{\mathbb{R}^n}\left|\frac{\Theta}{1-|p\vx|^2} + \frac{2(p - (p\vx)\cdot\Theta)}{(1-|p\vx|^2)^2} \cdot (p\vx) \right|\delta\paren{ \lambda - \frac{p - (p\vx)\cdot \Theta}{1-|p\vx|^2} }f(\vx)\mathrm{d}\vx\\
&= \frac{1}{p^{n-1}}\int_{\mathbb{R}^n}\left|\frac{\Theta}{1-|\vy|^2} + \frac{2(p - \vy\cdot\Theta)}{(1-|\vy|^2)^2} \cdot \vy \right|\delta\paren{ \lambda - \frac{p - \vy\cdot \Theta}{1-|\vy|^2} }f\paren{ \frac{\vy}{p} }\mathrm{d}\vy\\
&= \frac{1}{p^{n-1}}\int_{\mathbb{R}^n}|\nabla_{\vy}\Psi_2|\delta\paren{ \lambda - \Psi_2(\vy,\Theta) }\tilde{f}(\vy)\mathrm{d}\vy\\
& = \frac{1}{p^{n-1}}\mathcal{R}\tilde{f}(\lambda,\Theta),
\end{split}
\end{equation}
where $\tilde{f}(\vy) = f\paren{ \frac{\vy}{p} }$ and we made the
substitution $\vy = p\vx$ in step 3. Note that $\tilde{f}$ is
supported in $D$ since $f$ is.
\end{proof}

\begin{corollary}
\label{corr_inv_pal}
Let $r(\vy)  = \sqrt{\alpha^2 + (1-|\vy|)^2}$ and let $f\in
L^2_c(D)$. Then 
\begin{equation}
f(\vx) = \mathcal{R}^{-1}\paren{ p^{n-1} Rf } (p\vx).
\end{equation}
\end{corollary}
\begin{proof}
Since $0< p = 1/\sqrt{1+\alpha^2} < 1$ and $\text{supp}(f) \subset D$, we can apply Theorem \ref{inv_cst_2} and $\mathcal{R}^{-1}$ to obtain the result.
\end{proof}

\begin{corollary}\label{cor:Sobolev estimate}
 Let $D$ be the open unit disk in $\rn$. For $\vy\in D$, let \[r(\vy) =
\sqrt{\alpha^2 + (1-|\vy|)^2}.\] Let $B,\ B_1$, and $B_2$ be compact
subsets of $D$ such that $B\subset \intt(B_1)$ and $B_1\subset
\intt(B_2)$. Let $\phi_1$ and $\phi_2$ be in $\Dc(D)$ with $\phi_1=1$
on $B,\ \supp(\phi_1)\subset B_1$ and $\phi_2=1$ on $B_1,\
\supp(\phi_2)\subset B_2$. Let $\Nc' = \phi_2\Nc\phi_1$. For each
$s\in \rr$, there is a constant $C >0$ such that for each $f\in
H^s(B)$,
\begin{equation}
\left\|f\right\|_{H^s(B)} \leq
C\left\|\mathcal{N'}f\right\|_{H^{s+1}(B_2)}.
\end{equation} 
\end{corollary}

\begin{proof}
By Theorem \ref{PDO_cst_2}, there exists an order $1$ parametrix,
$\Pc$, of $\mathcal{N}$.  Without loss of generality, one may assume
$\Pc$ is properly supported.
Then, for
$f\in \Ec'(D)$ one can write \bel{parametrix1-4.2}\Pc\Nc f = I f+ K
f\ee for some smoothing operator $K$.

Now let  $f$ be supported in $B$.   By
pseudolocality of \psido, one can write \eqref{parametrix1-4.2} as
\begin{equation}
\label{parametrix2-4.2}
\begin{split}
\Pc'\Nc' f &=
f + K'f,\ \ \text{where}\\
\Pc' = \paren{\phi_2\Pc }, \ \
\Nc'&= \paren{\phi_2\Nc\phi_1},\ \ K' =
\phi_2K\phi_1+\phi_2\Pc(\phi_2-1)\Nc\phi_1,
\end{split}
\end{equation}
and  $K'$ is smoothing for distributions supported on $B$ since
$(\phi_2-1)$ is zero on $B$. 

Since $\Pc'$ and $\Nc'$ are properly supported \psido, for each $s\in
\rr$, \newline $\Pc':H^s(B_2)\to H^{s-1}(B_2)$ and $\Nc':H^s(B_2)\to
H^{s+1}(B_2)$, and $K':H^s(B)\to H^{s}(B_2)$ are continuous and $K'$
is a compact operator since it is smoothing on $B$.

Now, we have
\begin{equation}
\begin{split}
\|f\|_{H^s(B)}&\leq \|\mathcal{P}'\mathcal{N}'f\|_{H^{s}(B_2)} +
\|K'f\|_{H^s(B_2)} \\
&\leq C\paren{\|\mathcal{N}'f\|_{H^{s+1}(B_2)} + \|K'f\|_{H^s(B_2)}}. 
\end{split}
\end{equation}

The rest of the proof, using Lemma \ref{stef}, follows the same as the
proof of Corollary \ref{corr_stab_1} except the {smooth cutoff} $h$ is not
required.\end{proof}

%
%
%

 As a direct consequence of the above Corollary, we can
infer that the reconstruction of $f$ on from $Rf$ is (strongly) stable {on $B$.}

We now consider an example geometry where $r(\vy)$ is constant.

%% file: example4.3.tex
\subsection{The constant \texorpdfstring{$r$}{r} case}\label{sec:URT}
In this section, we apply our microlocal theory to the case when
$r(\vy) = r>0$ is constant. This example has been considered previously
in the literature, e.g., in \cite{agranovsky2011support,
john2004plane, Q1993mor}, and may have applications in URT, where
spherical waves are used to reconstruct the image. 


To prove our first set of theorems, we let \bel{def D0
Y0}D_0=\sparen{\vx\in \rn\st \norm{\vx}<r},\ \ Y_0=
\sparen{\vy\in\rn:\norm{\vy}<2r},\ee so $Y_0$ is the set of all centers of
spheres of radius $r$ that meet $D_0$ as in \eqref{def D0 Y0}.

Under this setup, we now show that $R$ satisfies the Bolker condition.
\begin{theorem}
\label{thm_ult} Let $D_0$ and $Y_0$ be as given in \eqref{def D0 Y0} and let the
radius function $r(\vy) = r$ be constant. Then $R:\Ec'(D_0)\to
\Ec'(Y_0)$ is an elliptic FIO which satisfies the Bolker condition.
\end{theorem}

\begin{proof}
We have $\nabla_\vy r = 0 <1$, and thus $R$ is an elliptic FIO with
immersive $\Pi_L$ by Proposition \ref{fio_prop} and Theorem
\ref{imm_thm}.   

As noted in Remark \ref{rem:artifact}, for this transform, if $\vy\in
\rn$ and $\vx\in S(\vy)$, then the artifact point for $\vy$ is the
antipodal point $\vx -2(\vx-\vy)$. If $\vx\in D_0$, and $\vy\in Y_0$,
the antipodal
point is outside of $D_0$. Therefore $D_0 \cap \Ac(D_0) =
\emptyset$, and $\Pi_L$ is injective above $D_0$ by Proposition
\ref{prop_geo}, and the result follows.
\end{proof}

In this case, the backprojection operator is
\begin{equation}
R^*g(\vx) = \int_{S^{n-1}} g(\vx + r\xi)\mathrm{d}\xi,
\end{equation}  which is the integral of $g$ over $S(\vx)$.

\begin{theorem}\label{R*R elliptic psido}
Let $r(\vy) = r$ and let $f\in \Ec'(D_0).$
Then, $R^*R:\Ec'(D_0)\to \Dc'(D_0)$ is an elliptic \psido.
\end{theorem}
\begin{proof}
Since $\nabla_\vy r = 0$, the strong norm inequality \eqref{strong
norm} holds, and the result follows from Theorem \ref{thm_ult} and
Theorem \ref{thm:elliptic psido} because the uniqueness assumption of
Theorem \ref{inj_thm} holds for $D_0$ and $Y_0$.
\end{proof}


\subsubsection{An inversion method}
In Theorems \ref{thm_ult} and \ref{R*R elliptic psido} the centers
$\vy\in {Y_0}$ can be inside the reconstruction space ($D_0$). We
now discuss injectivity of $R$ when $\vy$ is outside the scanning region. Specifically, we consider the geometry of figure \ref{fig3}. {Later, in section \ref{stable_conds_urt}, we consider the stability of inverting $R$ when $\vy$ is outside of $D_0$.}

This transform in, e.g., $\rthree$ can be a model for URT when data
are taken at one {fixed time and with constant sound speed}, so
on spheres of fixed radius with centers outside the object, which we
assume is {supported on $D_0$.} In URT, $\vy$ represents the
ultrasound emitter/receiver location and $r(\vy) = r$ represents the
scanning depth at one specific time. 

Let $d\in (0,r)$. The target function, $f$, is supported on the
annulus \bel{def D' Y'}D' = \{\vx : d<|\vx| < r\},\ \ Y'=\sparen{\vy:
d+r<|\vy|<2r} \ee as is illustrated in figure \ref{fig3}.

We now provide an injectivity result for $R$ in the geometry of figure
\ref{fig3} using a spherical harmonic decomposition (this would be a polar Fourier series in
{$\rr^2$}) and Volterra integral equations. The proof follows similar
ideas to \cite{Q1983-rotation}.

Let $\vy\in Y'$, and choose the unique $s\in (d,r)$ and $\omega\in
S^{n-1}$ such that $\vy = (s+r)\omega$.
The spherical coordinates of each point  $\vz\in S(\vy)\cap D'$ are
$(\rho,\xi)$  where 
\bel{spherical coords}\rho =
\norm{\vz}\in [s,r],\ \ \xi = \vz/\norm{\vz}\in \{\eta\in S^{n-1}\st
\eta \cdot \omega \geq \frac{s+r}{2r}\}.\ee Therefore, $\xi$ is in a
spherical cap and the angle between $\xi$ and $\omega$ is \[\varphi=
\arccos(\xi\cdot\omega) \in\left[0,\cos^{-1}\paren{ \frac{s+r}{2r}
}\right]\] Here we use that $f$ is supported in $D'$ and the
notation in figure \ref{fig3}.

By the cosine rule, we have
\begin{equation}
r^2 = (r+s)^2 +\rho^2 - 2\rho(r+s)\cos\varphi,
\end{equation}
which implies $\rho^2 - 2(r+s)(\cos\varphi)\rho + (s^2+2rs) = 0$ and
$\cos\varphi = \frac{\rho^2 + s^2 + 2sr}{2\rho(s+r)}$.
 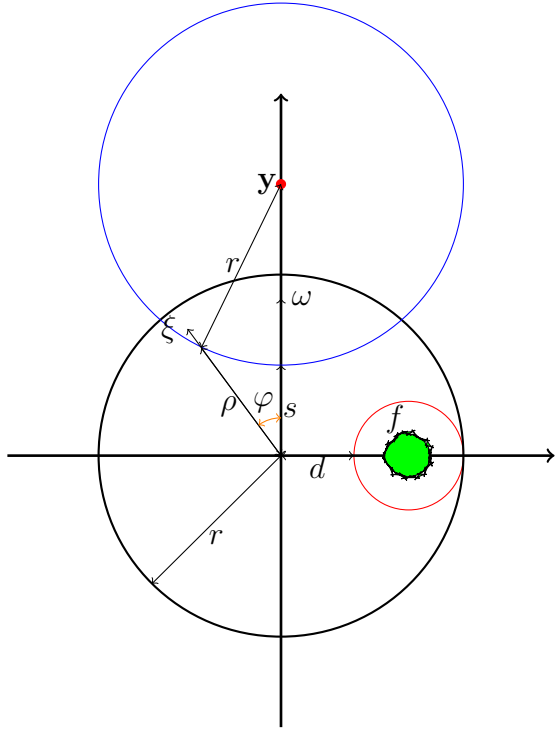
\begin{figure}[!h]
 \centering
 \begin{tikzpicture}[scale=0.6]
 \draw [red] (2.8,0) circle (1.2);
 \node at (2.5,0.8) {$f$};
 \draw [->,line width=1pt] (-6,0)--(6,0);
 \draw [->,line width=1pt] (0,-6)--(0,8);
 \draw [->] (0,0)--(0,3.47)node[right] {$\omega$};
 \draw [thick] (0,0) circle (4);
 \draw [<->] (0,0)--(-2.83,-2.83);
 \node at (-1.415,-1.8) {$r$};
 \draw [blue] (0,6) circle (4);
 \draw [red,fill=red] (0,6) circle (0.1);
 \node at (-0.3,6) {$\vy$};
 \draw[fill=green,rounded corners=1mm] (2.8,0) \irregularcircle{0.5cm}{0.5mm};
 \draw [<->] (0,0)--(0,2);

 \node at (0.2,1) {$s$};
 \draw [<->] (0,0)--(-1.75,2.38);
 \draw [->] (0,0)--(-2.05,2.8)node[left] {$\xi$};
 \node at (-1.15,1.1) {$\rho$};
 \draw [<->] (-1.75,2.38)--(0,6);
 \node at (-1.05,4.2) {$r$};
 \coordinate (D) at (-2.05,2.8);
 \coordinate (w) at (0,0);
 \coordinate (a) at (0,3.47);
 \draw [<->] (0,0)--(1.6,0);
 \node at (0.8,-0.25) {$d$};
 \draw pic[draw=orange, <->,"$\varphi$", angle eccentricity=1.5] {angle = a--w--D};
 \end{tikzpicture}
 \caption{Spherical ultrasound scanner.}
 \label{fig3}
 \end{figure}
For a unique  \bel{def:tau}\tau \in \omega^\perp :=\{\eta \in S^{n-1} : \eta \cdot \omega =
0\},\ \ 
\xi = \cos(\vp)\omega+\sin(\vp)\tau.\ee

Let $s\in{(d,r)}$, let $\omega\in S^{n-1}$, then, $(s,\omega)$ provide coordinates for  \[\vy =
(r+s)\omega\in Y'.\] For each $\vy\in Y'$, we parameterize $S(\vy)\cap D'$ in spherical
coordinates by
\begin{equation}\begin{gathered}
(\varphi,\tau)\mapsto (\rho,\xi) = \paren{ (r+s)\cos\varphi -
\sqrt{ (r+s)^2\cos^2\varphi - (s^2 +2rs) }, \cos(\varphi)\omega +
\sin(\varphi)\tau },\\
 \text{where}\ \ \vp\in \left[0,\cos^{-1}\paren{ \frac{s+r}{2r}
}\right],\ \tau\in \omega^\perp.\end{gathered}
\end{equation}
 See figure \ref{fig3}.  For each fixed $\vy$, we will use this
 parameterization in our proof.


 The standard measure on $S^{n-1}$ satisfies $\mathrm{d}\xi =
(\sin\varphi)^{n-2}\mathrm{d}\varphi \mathrm{d}\tau$. Let
\begin{equation}
\label{rho}
\rho = \rho(t,s) = (r+s)t - \sqrt{ (r+s)^2t^2 - (s^2 +2rs) },
\end{equation}
where $t =\cos \varphi = \xi \cdot \omega$. Then, the surface measure on $S$ is
\begin{equation}
\begin{split}
\mathrm{d}S &= (\rho \sin\varphi)^{n-2} \sqrt{ \rho(\cos\varphi,s)^2 + \sin^2\varphi\rho_t(\cos\varphi,s)^2 } \mathrm{d}\varphi\mathrm{d}\tau\\
&= \rho(\xi \cdot \omega,s)^{n-2} \sqrt{ \rho(\xi \cdot \omega,s)^2 + (1- (\xi \cdot \omega)^2)\rho_t(\xi \cdot \omega,s)^2 } \mathrm{d}\xi.
\end{split}
\end{equation}
Then,  for $f \in L^2_c\paren{ \{|\vx| <  r\} }$, $Rf$ has the alternate expression
\begin{equation}
\label{alt_R}
Rf(s,\omega) = \int_{\xi \in S^{n-1}\cap \{\xi \cdot \omega \geq \frac{s+r}{2r}\}} f(\rho(\xi \cdot \omega,s),\xi) \mathrm{d}S.
\end{equation}


\begin{theorem}
\label{inv_ult}
Let $r(\vy) = r$. Then, $Rf(\vy)$  for $\vy \in {Y'} $ determines
$f\in L^2_c(D')$ uniquely.
\end{theorem}
\begin{proof}

Let us decompose $f$ into a spherical harmonic expansion \[f(\rho,\xi)
= \sum_{lm}f_{lm}(\rho)Y_{lm}(\xi),\] where the $Y_{lm}$ form an
orthonormal basis of $L^2(S^{n-1})$ and each $Y_{lm}$ is a spherical
harmonic of degree $l$. Then, when $n\geq 3$, using equation
\eqref{alt_R}, we have
\begin{equation}
\label{n3}
\begin{split}
R\paren{f_{lm}  Y_{lm}}(s,\omega) &= \int_{\xi \in S^{n-1}\cap \{\xi \cdot \omega \geq \frac{s+r}{2r}\}} f_{lm}\paren{ \rho(\xi \cdot \omega,s) }Y_{lm}(\xi) \mathrm{d}S\\
&= c_n Y_{lm}(\omega) \int_{\frac{s+r}{2r}}^1 g(t,s) C_l^{\lambda}(t)f_{lm}(\rho(t,s))(1-t^2)^{\lambda - 1/2} \mathrm{d}t,
\end{split}
\end{equation}
where 
$$g(t,s) = \rho(t,s)^{n-2} \sqrt{ \rho(t,s)^2 + (1- t^2)\rho_t(t,s)^2 },$$
$c_n = \omega_{n-1}/C_l^{\lambda}(1)$, $\omega_n$ is the area of $S^{n-1}$ and the $C_l^{\lambda}$ are Gegenbauer polynomials degree $l$, order $\lambda = (n-2)/2$. The second step above follows from the Funk-Hecke Theorem \cite[pg 247]{ery}. When $n=2$, this becomes
\begin{equation}
\begin{split}
R\paren{f_{l} e^{-i l \xi }}(s,\omega) &= 2 e^{-i l \omega } \int_{\frac{s+r}{2r}}^1 g(t,s) T_l(t)f_{l}(\rho(t,s))(1-t^2)^{-1/2} \mathrm{d}t,
\end{split}
\end{equation}
where in this case $\xi, \omega \in [0,2\pi]$ parametrize $S^1$, the $f_l$ are the Fourier components of $f$, and $T_l$ is a Chebyshev polynomial degree $l$.

Defining $Rf_{lm}(s) = \int_{S^{n-1}} Rf(s,\omega) Y_{lm}(\omega) \mathrm{d}\omega$, we have
\begin{equation}
\label{volt_1}
Rf_{lm}(s) = c_n \int_{s}^r h(t(\rho,s),s) C_l^{\lambda}(t(\rho,s))f_{lm}(\rho)(1-t(\rho,s)^2)^{\lambda - 1/2} \mathrm{d}\rho,
\end{equation}
where
$$t = t(\rho,s) = \frac{\rho^2 + s^2 + 2sr}{2\rho(s+r)},$$
and
$$h(t,s) = \frac{g(t,s)}{\rho_t(t,s)} =  \rho(t,s)^{n-2} \sqrt{ \frac{\rho(t,s)^2}{\rho_t(t,s)^2} + (1- t^2)}.$$
We focus on the $n\geq 3$ case here as the $n = 2$ case is analogous. 

We have
$$\rho_t(t,s)^2 = (r+s)^2\paren{ 1 - \frac{(r+s) t}{\sqrt{t^2(r+s)^2 - (s^2+2rs)}} }^2 = \frac{ \rho^2(r+s)^2 }{ t^2(r+s)^2-(s^2+2rs) },$$
and thus,
\begin{equation}
K_1(\rho,s) = h(t(\rho,s),s) =  \rho^{n-2} \sqrt{ \frac{t(\rho,s)^2(r+s)^2-(s^2+2rs)}{(r+s)^2} + (1- t(\rho,s)^2)}.
\end{equation}
Now, we have
\begin{equation}
\begin{split}
1-t(\rho,s)^2 &= \frac{ (\rho^2 - s^2)\paren{ (2r+s)^2 - \rho^2 } }{ 4\rho^2(s+r)^2 }\\
&= (\rho - s) K_2(\rho,s),
\end{split}
\end{equation}
where $K_2 \in C^{\infty}(T)$ and $K_2 > 0$ on $T = \{d \leq s\leq r, s\leq \rho \leq r\}$. Therefore, $K_2$ must also be bounded away from zero on $T$. So,
\begin{equation}
\begin{split}
K_1(\rho,s) &=  \rho^{n-2} \sqrt{ t(\rho,s)^2 -\frac{(s^2+2rs)}{(r+s)^2} + (\rho-s)K_2(\rho,s) }\\
&= \frac{ \rho^{n-2} }{ 2\rho(r+s) } \sqrt{(\rho^2+s^2+2sr)^2 - 4\rho^2(s^2+2sr)+(\rho^2-s^2)((2r+s)^2-\rho^2) }\\
&= \frac{ r \rho^{n-2}}{ (r+s) },
\end{split}
\end{equation}
after reducing the polynomial under the square root in the second step. Thus, $K_1 \in C^{\infty}(T)$.

Let us define
$$K(\rho,s) = c_n K_1(\rho,s) \paren{K_2(\rho,s)}^{(n-3)/2} C_l^{\lambda}(t(\rho,s)).$$ 
For example, when $n=3$, $K$ has the simple expression
$$K(\rho,s) = 2\pi \cdot \frac{ r \rho}{ (r+s) } \cdot P_l\paren{ \frac{\rho^2 + s^2 + 2sr}{2\rho(s+r)} },$$
where $P_l$ is a Legendre polynomial degree $l$. For general $n$, $K
\in C^{\infty}(T)$ since \newline$t(\rho,s) \in C^{\infty}(T)$, and
$K_2 \in C^{\infty}(T)$ is bounded away from zero on $T$. Now,
\eqref{volt_1} becomes
\begin{equation}
\label{volt_2}
Rf_{lm}(s) = \int_{s}^r(\rho - s)^{(n-3)/2}K(\rho,s) f_{lm}(\rho)\mathrm{d}\rho,
\end{equation}
a generalized Abel equation with smooth kernel. Further, 
\begin{equation}
K_1(s,s) = \frac{rs^{n-2}}{(r+s)},
\end{equation}
which is strictly greater than zero on $[d,r]$. Therefore,
\begin{equation}
K(s,s) = \omega_{n-1} K_1(s,s) \paren{K_2(s,s)}^{(n-3)/2} > 0
\end{equation}
on $[d,r]$. We can now solve \eqref{volt_2} uniquely for $f_{lm}$, for $\rho \in [d,r]$ using \cite[Corollary 3.4]{web}. If $Rf_{lm} = 0$, then $f_{lm} = 0$ on $[d,r]$, and thus $\|f_{lm}\|_{L^2([d,r])} = 0$. It follows that, if $Rf = 0$, $f = 0$ in $L^2$ since all its harmonic components are zero.
\end{proof}

\subsubsection{Stability with limited data}
\label{stable_conds_urt}

We now introduce a set of centers that is sufficient for weak
stability, as we will show. Let \bel{def Y''}Y'' = \sparen{\vy\in
\rn\st r<\norm{\vy}<2r} {\supset Y'.}\ee Let $D'$ be given in \eqref{def D' Y'}.

In the last section, we proved that $R$ is invertible for functions
$f\in L^2_c(D')$ from data with $\vy\in Y'$ (see Theorem \ref{inv_ult}
and \eqref{def D' Y'}). Therefore, $R$ is invertible on $L^2_c(D')$
for $\vy$ in the larger set $ Y''$.



 For $\vx\in D'$, let \bel{def:H_r} H_r(\vx) = \sparen{\vy\in
S(\vx)\st \vy - \vx) \cdot (\vx/\norm{\vx}) \geq 0}.\ee Then
$H_r(\vx)$ is an closed hemisphere of points $\vy$ such that $\vx\in
S(\vy)$. Furthermore, a straightforward calculation shows the closest
point in $H_r(\vx)$ to the origin is  greater than distance $\sqrt{d^2+r^2}>r$
units from the origin and the farthest point is less than $2r$ units
from the origin, so \bel{H_r}H_r(\vx)\subset Y'', \ \ \text{and}\ \
H_r(\vx)\cap \cl(D')=\emptyset.\ee

\begin{theorem}
\label{lim_ult} Let $f$ be a distribution supported in $D'$. Let
$Rf(\vy)$ be known for $\vy\in Y''$. Then, the reconstruction of $f$
from $Rf$ is weakly stable everywhere in $D'$.  That is,
$R:\Ec'(D')\to \Dc'(Y'')$ is weakly stable.\end{theorem}

\begin{proof}  Let $\vx\in D'$. Let $f\in \Ec'(D')$. Because $H_r(\vx)\subset Y''$, by
\eqref{H_r}, $Rf(\vy)$ is known for all points in $H_r(\vx)$.

Let $\xi\in S^{n-1}$, then the line $\vx + \rr \xi$ intersects the
hemisphere $H_r(\vx)\subset Y''$ at one or two points. Then, \eqref{1}
of Definition \ref{stable_def} is satisfied and $\Pi_R:\Cc\to
\Omega\times \drn$ is surjective.

Condition \eqref{2} is proven in Theorem \ref{thm_ult}.
\end{proof}

%
%


\begin{remark}
Theorem \ref{lim_ult} shows that we can achieve a stable solution for
$R:\Ec'(D')\to \Dc'(Y'')$.  If the domain $D'$ is shrunk, then one
could need only a subset of $Y''$ for weak stability, and here is an
example.


Let \bel{def D}D= \{\vx\in \rn\st |\vx - ((r+d)/2,0,\ldots,0)| <
(r-d)/2\}. \ee Then $D\subset {D'}$. {An example cross-section of $D$ in $\rtwo$ is shown in figure \ref{fig3}, where the boundary of $D$ is highlighted as a red circle.} 

We do not need centers over all of $Y''$ to have weak stability on
$D$. We can largely use the $\vy$ on one side of $D$. This is because
many of the $S(\vy)$ with $\vy \in Y''$ do not intersect $D$, and this
effect is more pronounced as $d\to r$, where $d$ is effectively
controlling the size of the object supported in $D$. This has potential practical relevance, as one would not need to have emitters encircling the entire object for stable reconstruction if the object is small enough, e.g., in handheld URT this may be useful. 

Let $S_D$ be the set of $\vy \in Y''$ such that $S(\vy)$ also intersects $D$. An example $D$ and $S_D$ is shown in $\rtwo$ in figure \ref{fig4}, where $S_D$ is the yellow region and the boundary of $D$ is the white circle.
The arguments in Theorem \ref{lim_ult} show that data over $S_D$ is
enough for weak stability at each $\vx\in D$ since $D\subset D'$
{($h=0$ on $D'$ and Bolker holds since it holds for $R$).} Note that
$S_D$ is open because $D$ is open
\end{remark}

\begin{figure}[!h]
\centering \includegraphics[width=0.9\linewidth, height=5cm,
keepaspectratio]{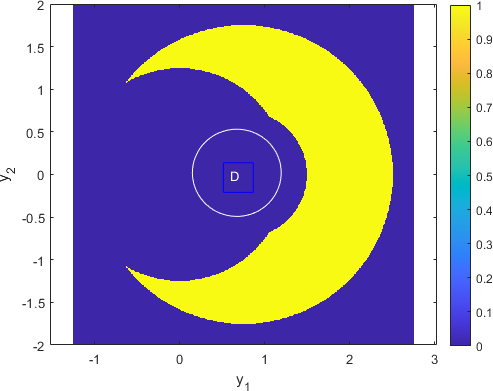} \caption{{In $\rtwo$, the
white circle is the boundary of $D$, and the set in yellow is $S_D$--the set of circle
centers $\vy \in Y'$ for which $S(\vy)\cap D \neq \emptyset$. Here $r=1.25$
and $d = 0.25$.}}\label{fig4}
\end{figure}

 We now aim to prove global (strong) stability on $D'$ given by
\eqref{def D' Y'} and $Y''$ given by \eqref{def Y''}.

\newcommand{\YDp}{Y_{D'}}

Let $\YDp = \bigcup_{\vx\in D'} H_r(\vx)$, then $\YDp$ is an annulus
and because
$D'$ is open, $\YDp$ is open.

 Let $h$ be a smooth nonnegative cutoff function which is supported on
$\sparen{\vy\in \rn\st r<\vy}$ and equal to $1$ on $\YDp$. This is
possible because $\YDp$ is bounded away from $\cl(D')$ as justified
just above \eqref{H_r}. Let $R^{\dagger} = hR$. Then, for $u\in \Ec'(D')$, $\supp(R^{\dagger}u)\subset
Y''$. The operator \[\mathcal{N} = R^*R^{\dagger}\] is a modified normal
operator that just uses data, $Ru$, on $Y''$. 


\begin{theorem}
\label{ell_thm_4}
Let $\mathcal{N} = R^*R^{\dagger}$. Then, $\mathcal{N} : H^s_c(D')\to
H^{s+1}_{\text{\rm loc}}(D')$  is an elliptic \psido order $-1$.
\end{theorem}

{Note that we are not saying that $\Nc$ is continuous. The proof
of Theorem \ref{corr_stab_4} will include a more restrictive
continuity result (see \eqref{Sobolev cont}).}

\begin{proof}
Similarly to the proof of Theorem \ref{ell_thm_1}, it is clear that
$hR : H^s_c(D')\to H^{s+1/2}_{c}(Y'')$ is an FIO with the
same phase and canonical relation as $R$ and that $R^{\dagger}$
satisfies the Bolker condition. It follows that the canonical relation
of $\mathcal{N}$ is a subset of the diagonal by Theorem \ref{lim_ult},
and thus $\mathcal{N}$ is a \psido. The order of $\mathcal{N}$ is $-1$
since the order of both $R^{\dagger}$ and $R^*$ is $-1/2$. 

By Theorem
\ref{lim_ult}, the right projection of $R^{\dagger}$ is surjective.  This
follows from the fact that for every $(\vx,\xi)\in D'\times \drn$,
there is a $\vy\in Y''$ such that $(\vx,\xi)$ is normal to $S(\vy)$. 

To prove that $\Nc$ is elliptic, we will go through the calculation of
its symbol. Let $(\vx,\xi)\in D'\times \drn$. Then the two points
$\vy_1 = \vx-r\xi/\norm{\xi}$ and $\vy_2 = \vx+r\xi/\norm{\xi}$ are
both in $Y_0$ and $(\vx,\xi)$ is normal to $S(\vy_j)$ for $j=1,2$.
Since $H_r(\vx)$ is a closed hemisphere, at least one of $\vy_j$ is in
$H_r(\vx)$. Without loss of generality, assume $\vy_1\in H_r(\vx)$.
Therefore, $h(\vy_1)=1$ and $h(\vy_2)\geq 0$. Let $\eta_j\in \drn$ be
chosen so the two preimages of $(\vx,\xi)$ in $\Cc$ are
\[\lambda_j=(\vy_j,\eta_j, \vx,\xi),\ \  j=1,2.\]
 The symbol of  $R^*R^{\dagger}$ is the sum of two
terms calculated using  these two preimages.

For each $\lambda_j$, one multiplies the weights for $R$ and $R^*$
evaluated at $(\vy_j, \vx)$ by $h(\vy_j)$. Since the weights on $R$
and $R^*$ are positive and $h(\vy_1)=1$, this product for $\lambda_1$ is positive.
Then, one multiplies by a positive Jacobian factor given in
\cite[equation (15) and proof of Theorem 2.1]{quinto} or Lemma 5.1 of
\cite{GKQR:ma}. This shows that the contribution to the symbol of
$\Nc$ at $\lambda_1$ is positive. The calculation for the contribution
at $\lambda_2$ is similar, and since $h(\vy_2)\geq 0$, this
contribution is nonnegative. Finally, one adds these two terms up.
Since the sum is positive, the symbol of $\Nc$ is positive, and $\Nc$
is elliptic. 
\end{proof}

We now prove a strong Sobolev inverse continuity result, showing that
inversion from $R$ is stable in a Sobolev sense.

\begin{corollary}
\label{corr_stab_4} 
Let $B,\ B_1$, and $B_2$ be compact
subsets of $D'$ such that $B\subset \intt(B_1)$ and $B_1\subset
\intt(B_2)$. Let $\phi_1$ and $\phi_2$ be in $\Dc(D')$ with $\phi_1=1$
on $B,\ \supp(\phi_1)\subset B_1$ and $\phi_2=1$ on $B_1,\
\supp(\phi_2)\subset B_2$. Let $\Nc' = \phi_2\Nc\phi_1$. For each
$s\in \rr$, there is a constant $C >0$ such that for each $f\in
H^s(B)$,
\begin{equation}
\left\|f\right\|_{H^s(B)} \leq
C\left\|\mathcal{N'}f\right\|_{H^{s+1}(B_2)}.
\end{equation} 
\end{corollary}

\begin{proof}
By Theorem \ref{ell_thm_4}, there exists an order $1$ parametrix,
$\Pc$, of $\mathcal{N}$.  Without loss of generality, one may assume
$\Pc$ is properly supported.
Then, for
$u\in \Ec'(D')$ one can write \bel{parametrix1-4.3}\Pc\Nc u = I u+ K
u\ee for some smoothing operator $K$.

Now let  $u$ be supported in $B$.   By
pseudolocality of \psido, one can write \eqref{parametrix1-4.3} as
\begin{equation}
\label{parametrix2-4.3}
\begin{split}
\Pc'\Nc' u &=
u + K'u,\ \ \text{where}\\
\Pc' = \paren{\phi_2\Pc }, \ \
\Nc'&= \paren{\phi_2\Nc\phi_1},\ \ K' =
\phi_2K\phi_1+\phi_2\Pc(\phi_2-1)\Nc\phi_1,
\end{split}
\end{equation}
and  $K'$ is smoothing for distributions supported on $B$ since
$(\phi_2-1)$ is zero on $B$. 

Since $\Pc'$ and $\Nc'$ are properly supported \psido, for each $s\in
\rr$, \bel{Sobolev cont} \Pc':H^s(B)\to H^{s-1}(B_2),\ \
\Nc':H^s(B)\to H^{s+1}(B_2), \ \ K':H^s(B)\to H^{s}(B_2)\ee are
continuous and $K'$ is a compact operator since it is smoothing on
$B$.

Now, we have
\begin{equation}
\begin{split}
\|u\|_{H^s(B)}&\leq \|\mathcal{P}'\mathcal{N}'u\|_{H^{s}(B_2)} +
\|K'u\|_{H^s(B_2)} \\
&\leq C\paren{\|\mathcal{N}'u\|_{H^{s+1}(B_2)} + \|K'u\|_{H^s(B_2)}}. 
\end{split}
\end{equation}
The rest of the proof follows the same arguments as the
proof of Corollary \ref{corr_stab_1}.\end{proof}


\noindent In the next section, we will present simulated image
reconstructions including ones using the limited data and the scanning
region shown in figure \ref{fig4}.

%% file: images_new.tex
\section{Image reconstructions}
\label{images}

In this section, we present image reconstructions
based on the scanning geometries discussed in section \ref{examples}.
We use algebraic methods for reconstruction. Let $A$ be the
discretized form of $R$, let $\vx$ be the vectorized image 
of the object to be reconstructed, and let $b$ be the data. Then, to
recover $\vx$, we aim to solve
\begin{equation}
\label{solve}
\argmin_{\vx \in\mathcal{X}} \|A\vx - b\|^2 + \lambda G(\vx),
\end{equation}
where $G$ represents a regularization penalty (e.g., total variation)
and $\mathcal{X}$ is the space of solutions (e.g., non-negative
functions). We mainly consider two methods for solving \eqref{solve},
namely Landweber \cite{landweber1951iteration} (implemented using the
code of \cite{AIRtools}), and a gradient decent method which uses
Total Variation (TV) regularization and non-negativity constraints,
specifically the code provided in \cite{ehrhardt2014joint}. In this
case, $\mathcal{X}$ is the space of vectors $\vx$ with
non-negative components. Then, $G(\vx) = \sqrt{\|\nabla\vx\|_2^2 +
\beta^2}$ is a modified TV norm, where $\beta$ is an additional
smoothing parameter that is introduced so that the gradient of $G$ is
defined at zero.

We generate noiseless data via $b = A\vx$, where $\vx$ is the ground truth. To avoid inverse crime, we use a finer pixel grid when generating $b$ than in the reconstructed image. Specifically, in this section, the image reconstructions have resolution $100\times 100$, but $b$ is generated using $105\times 105$ images. We add noise via
\begin{equation}
b_\epsilon = b + \gamma \times \frac{\|b\|_2}{\sqrt{k}}\eta,
\end{equation}
where $k$ is the length of $b$, $\eta \sim \mathcal{N}(0,1)$ is a vector of draws from standard normal distribution, and $\gamma$ controls the noise level. The images are reconstructed from $b_\epsilon$, and we measure least squares error using $\delta = \|\vx_\epsilon -\vx\|_2/\|\vx\|_2$, where $\vx$ is the ground truth, and $\vx_\epsilon$ is the reconstruction.

\subsection{Linear translation CST images}\label{sec:linear
translation} In this section, we present reconstructions in the linear
translation CST geometry analyzed in section \ref{sec:linearCST}. We
consider the half annulus phantom shown in figure \ref{F1}. We also
show the data sinogram and the sinogram multiplied by a smooth cutoff
$h$ as explained in section \ref{sec:linearCST}. We set $\alpha = 1$.
Note, the phantom is supported on $\{x_2 >0\}$ in line with the
setup
of section \ref{sec:linearCST}. 
\begin{figure}
\centering
\begin{subfigure}{0.24\textwidth}
\includegraphics[width=0.9\linewidth, height=3.2cm, keepaspectratio]{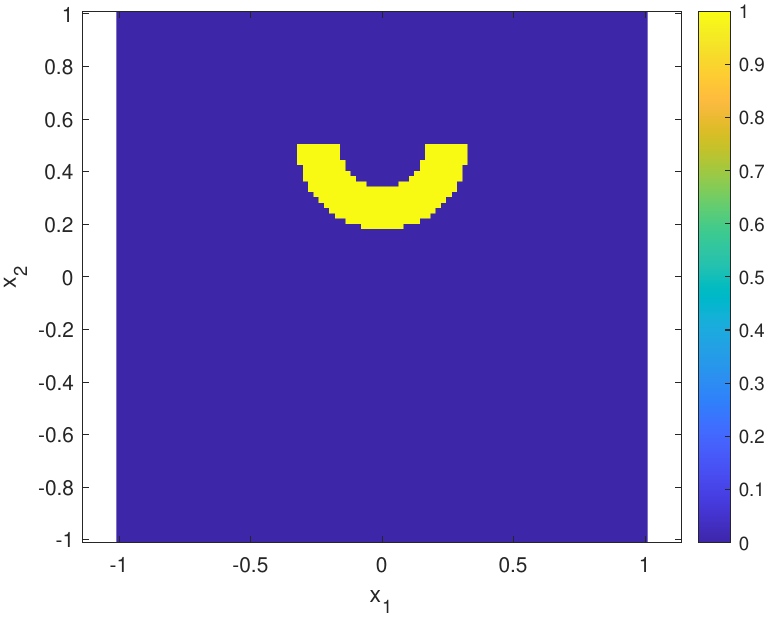}
\subcaption*{ground truth}
\end{subfigure}
\begin{subfigure}{0.24\textwidth}
\includegraphics[width=0.9\linewidth, height=3.2cm, keepaspectratio]{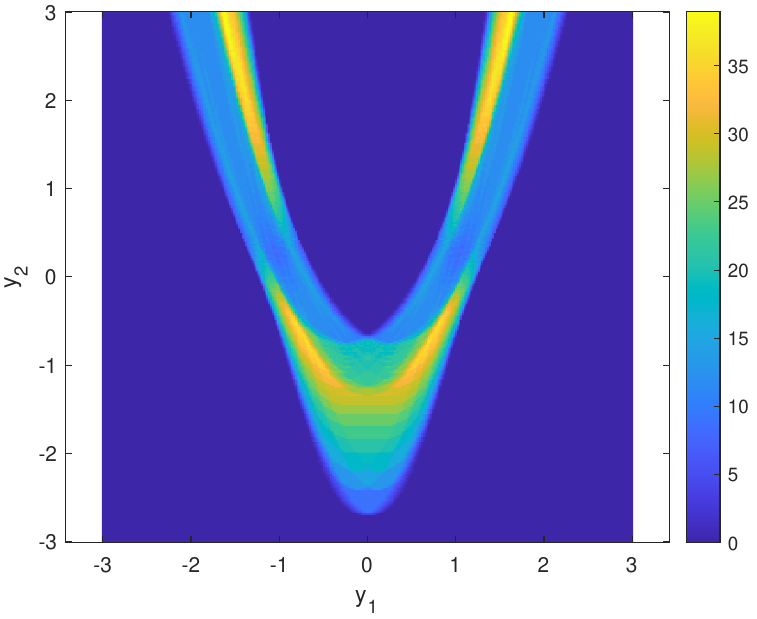}
\subcaption*{sinogram}
\end{subfigure}
\begin{subfigure}{0.24\textwidth}
\includegraphics[width=0.9\linewidth, height=3.2cm, keepaspectratio]{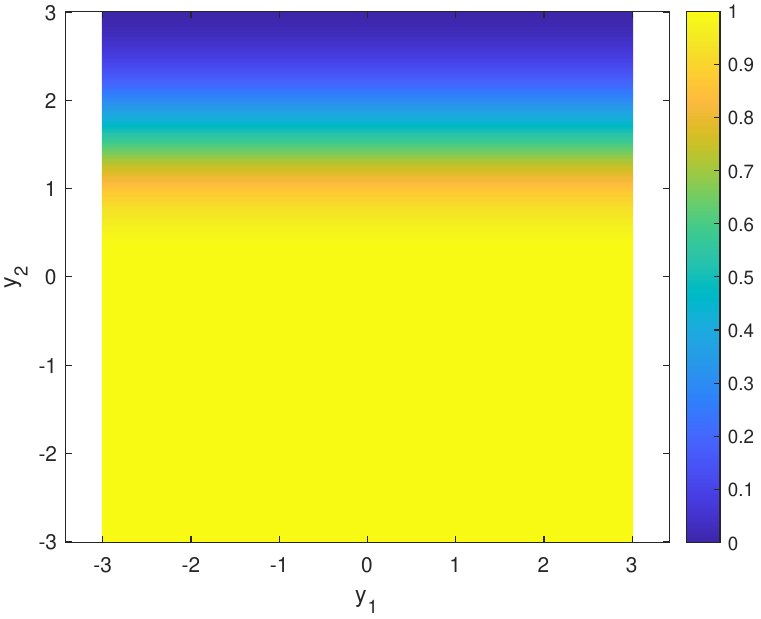}
\subcaption*{$h$}
\end{subfigure}
\begin{subfigure}{0.24\textwidth}
\includegraphics[width=0.9\linewidth, height=3.2cm, keepaspectratio]{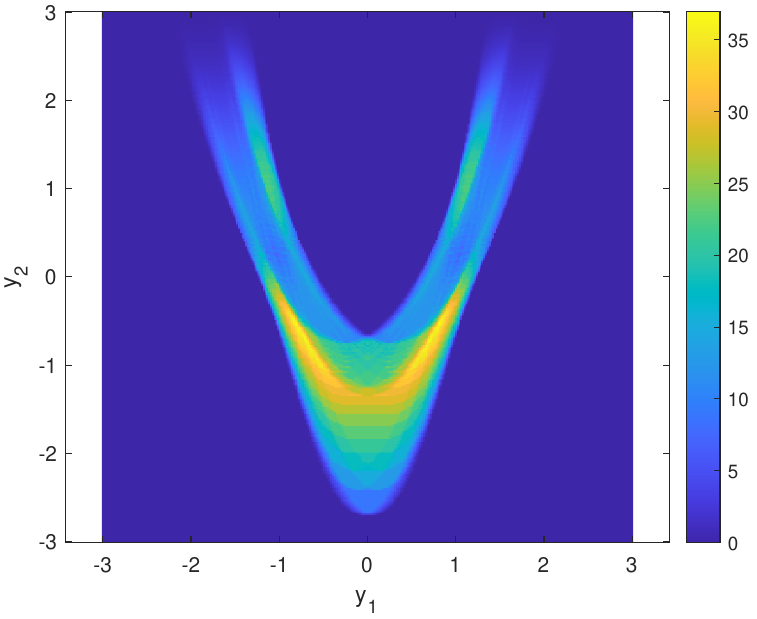}
\subcaption*{filtered sinogram}
\end{subfigure}
\caption{Ground truth half annulus phantom and sinogram $A\vx$
for the linear translation CST of section \ref{sec:linear translation}
(on the left). The smooth cutoff $h$ described in section \ref{sec:linearCST}
and the filtered sinogram $hA\vx$ (on the right). } \label{F1}
\end{figure}


We will compare reconstructions from the standard and smoothed sinograms. Here, in addition to Landweber and TV reconstructions, we also present Filtered Back-Projection (FBP) images to show the effects of filtering in the recovery of the image edges. We implement FBP via $\vx_\epsilon =  A^T\partial_{y_2}b$, where $\partial_{y_2}$ indicates the partial derivative of $b$ in the $y_2$ direction. This filter is a pseudodifferential operator and thus fits our theory by the H\"{o}rmander-Sato Lemma (Theorem \ref{thm:HS}), in the sense that this will not cause artifacts in the image. When reconstructing using filtered data $hb$, we implement FBP using $\vx_\epsilon =  \tilde{A}^T\partial_{y_2}(hb)$, where $\tilde{A} = hA$.
\begin{figure}
\centering
\begin{subfigure}{0.24\textwidth}
\includegraphics[width=0.9\linewidth, height=3.2cm, keepaspectratio]{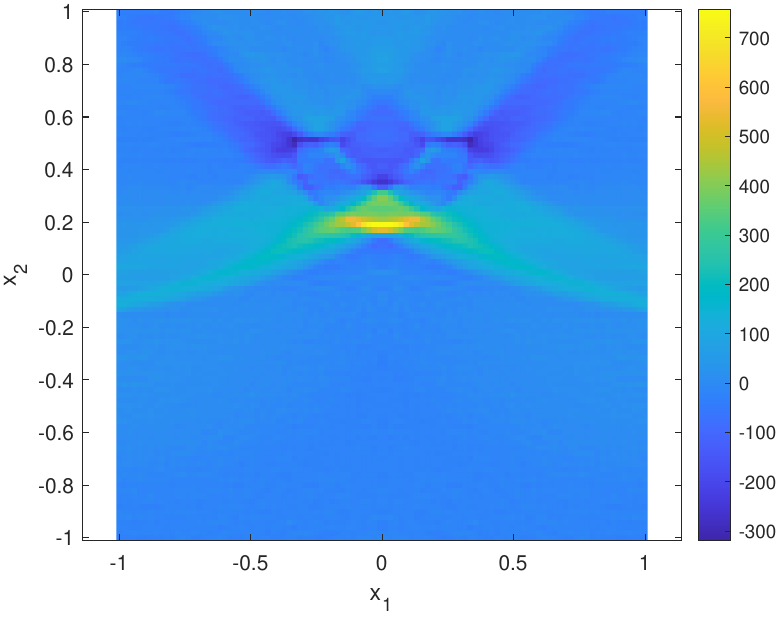}
\end{subfigure}
\begin{subfigure}{0.24\textwidth}
\includegraphics[width=0.9\linewidth, height=3.2cm, keepaspectratio]{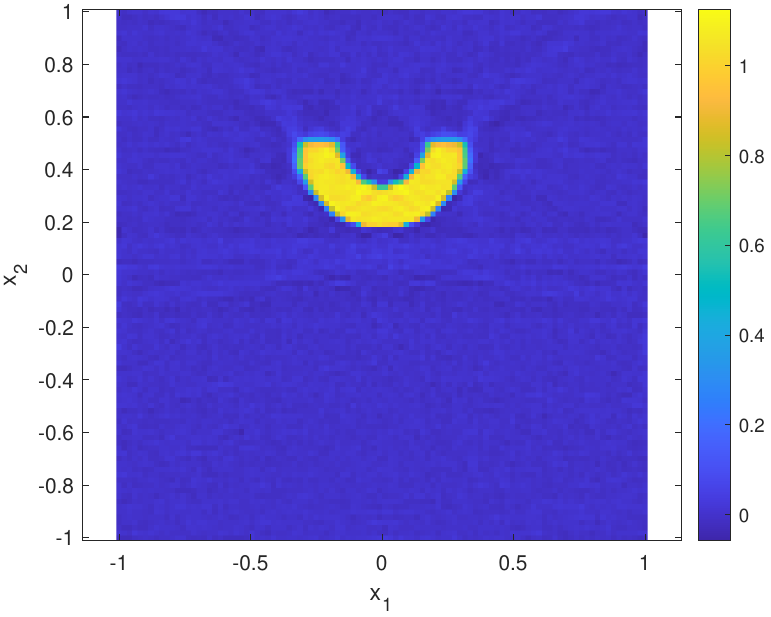}
\end{subfigure}
\begin{subfigure}{0.24\textwidth}
\includegraphics[width=0.9\linewidth, height=3.2cm, keepaspectratio]{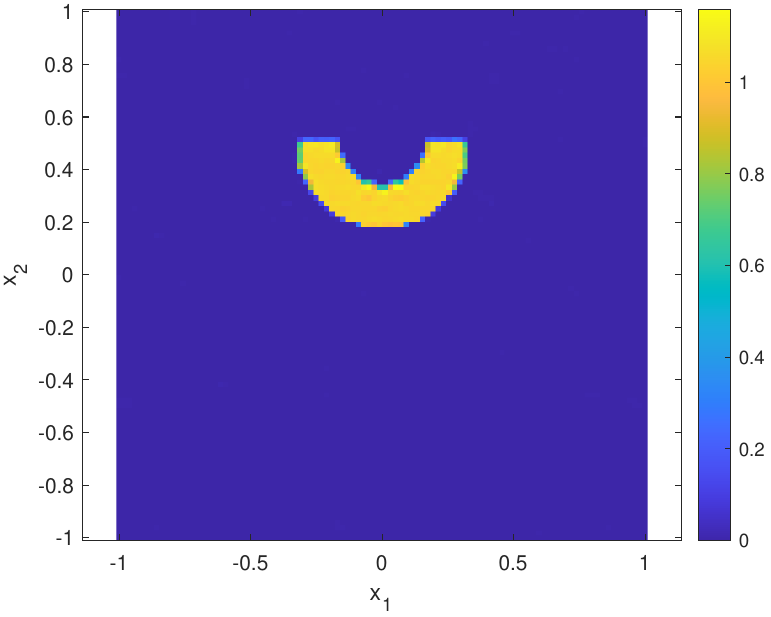}
\end{subfigure}
\\
\begin{subfigure}{0.24\textwidth}
\includegraphics[width=0.9\linewidth, height=3.2cm, keepaspectratio]{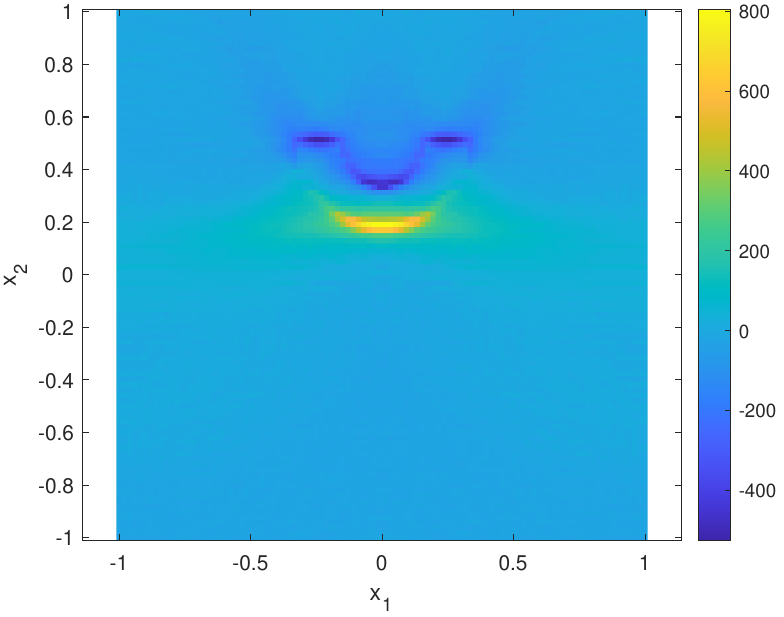}
\subcaption*{FBP}
\end{subfigure}
\begin{subfigure}{0.24\textwidth}
\includegraphics[width=0.9\linewidth, height=3.2cm, keepaspectratio]{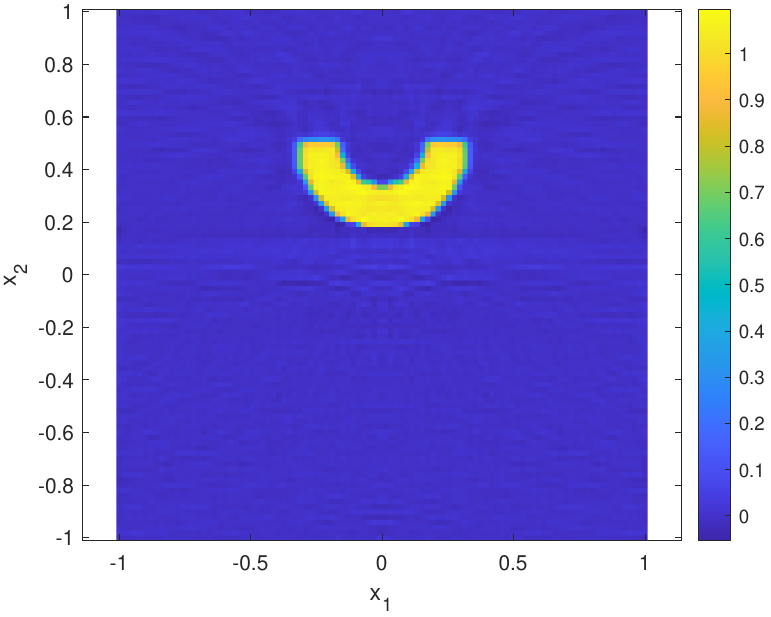}
\subcaption*{Landweber}
\end{subfigure}
\begin{subfigure}{0.24\textwidth}
\includegraphics[width=0.9\linewidth, height=3.2cm, keepaspectratio]{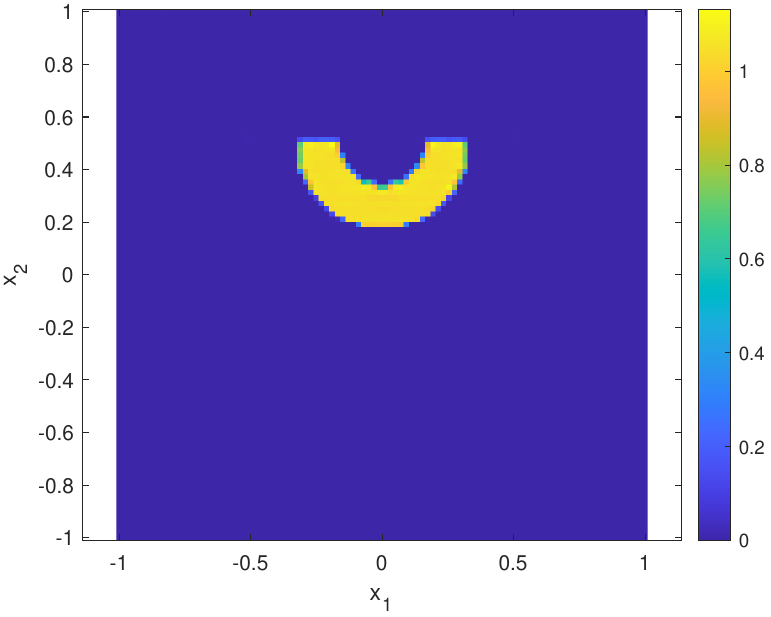}
\subcaption*{TV}
\end{subfigure}
\caption{Image phantom reconstructions using the unfiltered (top row)
and filtered (bottom row) sinograms for the linear translation CST geometry of
section \ref{sec:linear translation}. The noise level used for the FBP images in zero and set at $\gamma = 0.05$ for the Landweber and TV reconstructions.}
\label{F2}
\end{figure}
In the FBP reconstruction using the standard sinogram, there are streaking artifacts. These mainly occur along circles which correspond to $\vy$ near the sinogram edge, i.e., near $\{y_2 = 3\}$ in figure \ref{F1}. This is due to the sharp cropping of the sinogram along $\{y_2 = 3\}$. The artifacts are significantly reduced in the filtered sinogram reconstruction and the original edge map is more clearly represented. We see a similar effect in the Landweber reconstructions, although the least squares errors are comparable ($\delta = 0.21$ unfiltered, $\delta = 0.20$ filtered). The artifacts and background noise are suppressed further in the TV reconstructions ($\delta = 0.15$ in the unfiltered and filtered images), although these require more machinery to implement (e.g., gradient methods and non-negativity) when compared to Landweber.

\subsection{Rotational CST reconstructions}\label{sec:rotation}
Here, we present reconstructions based on the scanning geometry of section \ref{sec:circularCST}. We recover a half annulus phantom as previously. The phantom, its sinogram, and image reconstruction using Landweber and TV are shown in figure \ref{F3}. The annulus is slightly off center to avoid and rotation invariance.
\begin{figure}
\centering
\begin{subfigure}{0.24\textwidth}
\includegraphics[width=0.9\linewidth, height=3.2cm, keepaspectratio]{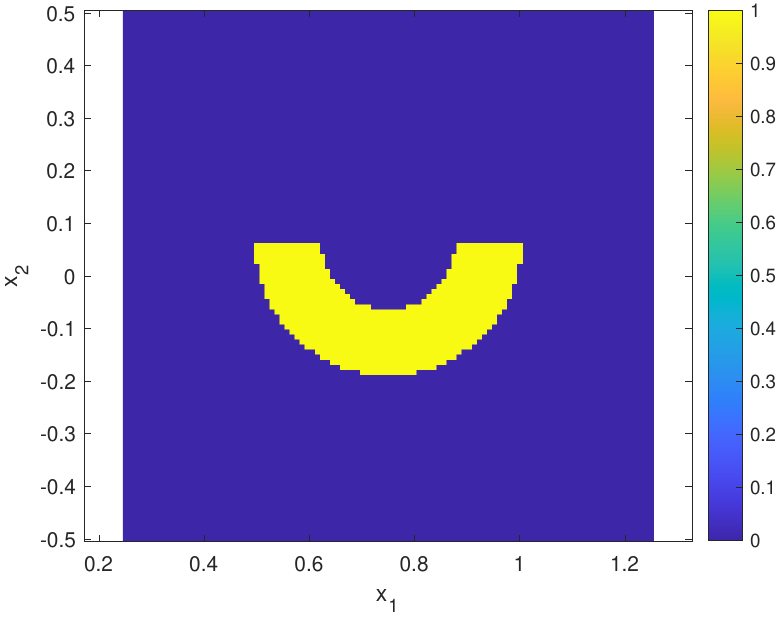}
\subcaption*{ground truth}
\end{subfigure}
\begin{subfigure}{0.24\textwidth}
\includegraphics[width=0.9\linewidth, height=3.2cm, keepaspectratio]{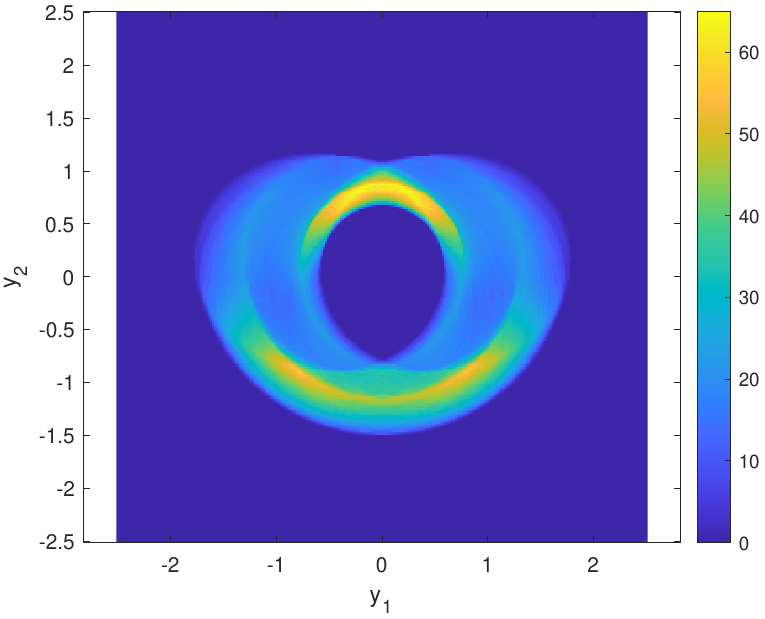}
\subcaption*{sinogram}
\end{subfigure}
\begin{subfigure}{0.24\textwidth}
\includegraphics[width=0.9\linewidth, height=3.2cm, keepaspectratio]{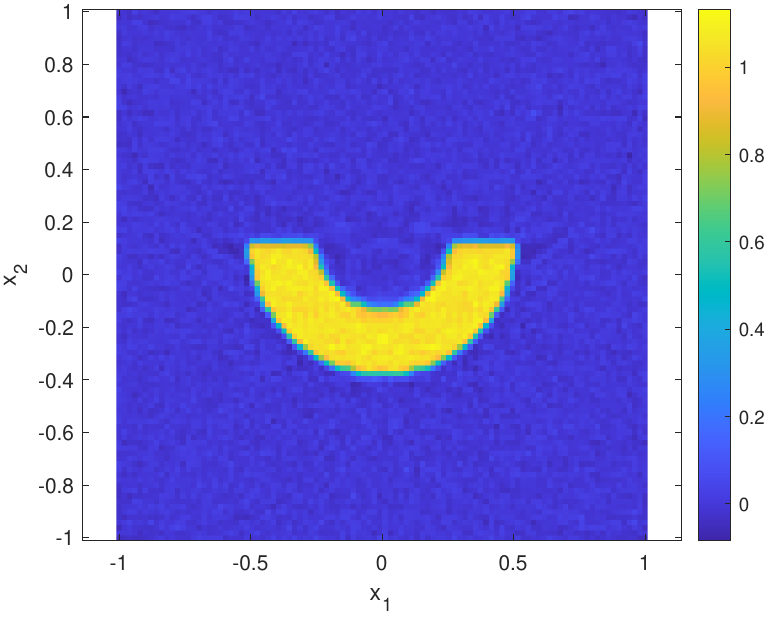}
\subcaption*{Landweber}
\end{subfigure}
\begin{subfigure}{0.24\textwidth}
\includegraphics[width=0.9\linewidth, height=3.2cm, keepaspectratio]{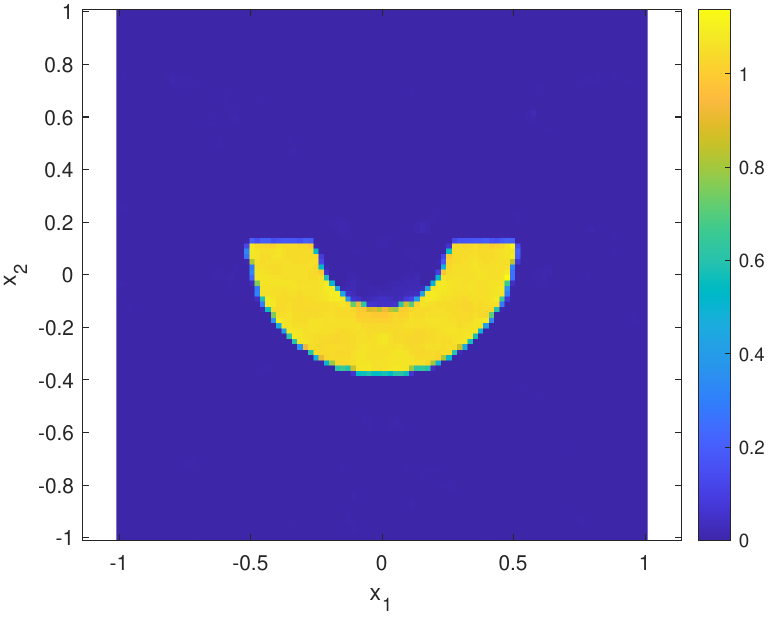}
\subcaption*{TV ($\delta = 0.16$)}
\end{subfigure}
\caption{{Ground truth annulus phantom and  sinogram for the rotational CST geometry of section
\ref{sec:rotation} (on the left). Reconstructions using the Landweber and TV
methods are shown in the right two figures. The noise level added to
the data was $\gamma = 0.05$.} }\label{F3}
\end{figure}

In contrast to the previous section, we do not see streaking artifacts in the Landweber reconstruction. This is because, in this geometry, the sinogram is compactly supported and thus cropping or smoothing is not necessary. There is background noise in the Landweber reconstruction ($\delta = 0.19$), which is to be expected given the significant level of noise added to the data ($\gamma = 0.05$). As in the previous example, the TV reconstruction with non-negativity constraints is most optimal ($\delta = 0.16$), and mitigates much of the background noise in the image.

\begin{remark}
Both CST geometries studied here (linear and rotational) offer a
stable solution by corollaries \ref{corr_stab_1} and \ref{corr_inv_pal}, although the linear translation system requires additional sinogram smoothing to achieve global estimates on Sobolev scale, and thus perhaps the rotation geometry we propose could be considered most optimal for CST imaging from a theoretical perspective. In addition to stability, practical application should also be factored in, e.g., ease of design, expected noise level ($\gamma$), and scan time etc. Such practical considerations are outside the scope of this paper, however, and will be the subject of future work.
\end{remark}

\subsection{The constant \texorpdfstring{$r$}{r} case}\label{sec:const
r} In this
example, we present image reconstructions based on the geometry of
section \ref{sec:URT} in the case when $n=2$. In particular, we
reconstruct an example $f$ {using sphere centers in $Y''$.} We
consider a half annulus phantom for reconstruction as in the previous
examples. See figure \ref{F4}.
\begin{figure}
\centering
\begin{subfigure}{0.24\textwidth}
\includegraphics[width=0.9\linewidth, height=3.2cm, keepaspectratio]{GT_cst_3}
\subcaption*{ground truth}
\end{subfigure}
\begin{subfigure}{0.24\textwidth}
\includegraphics[width=0.9\linewidth, height=3.2cm, keepaspectratio]{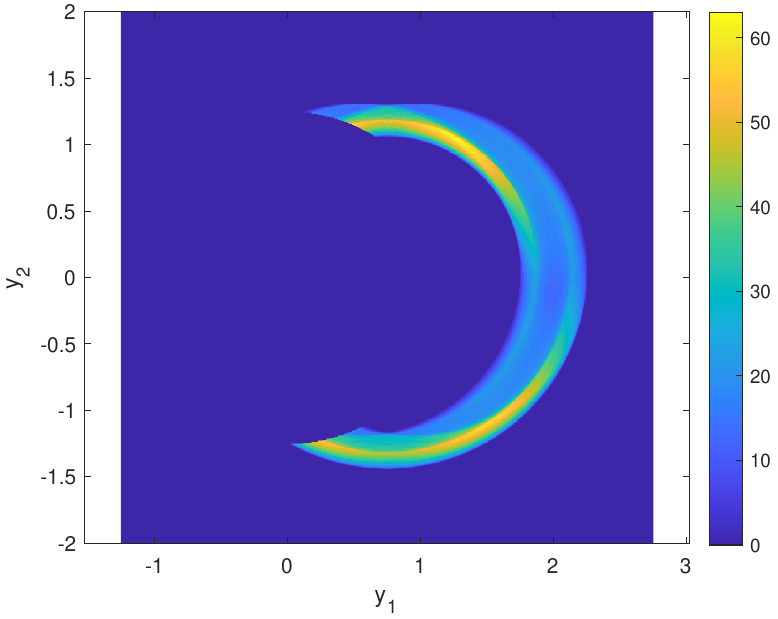}
\subcaption*{sinogram}
\end{subfigure}
\begin{subfigure}{0.24\textwidth}
\includegraphics[width=0.9\linewidth, height=3.2cm, keepaspectratio]{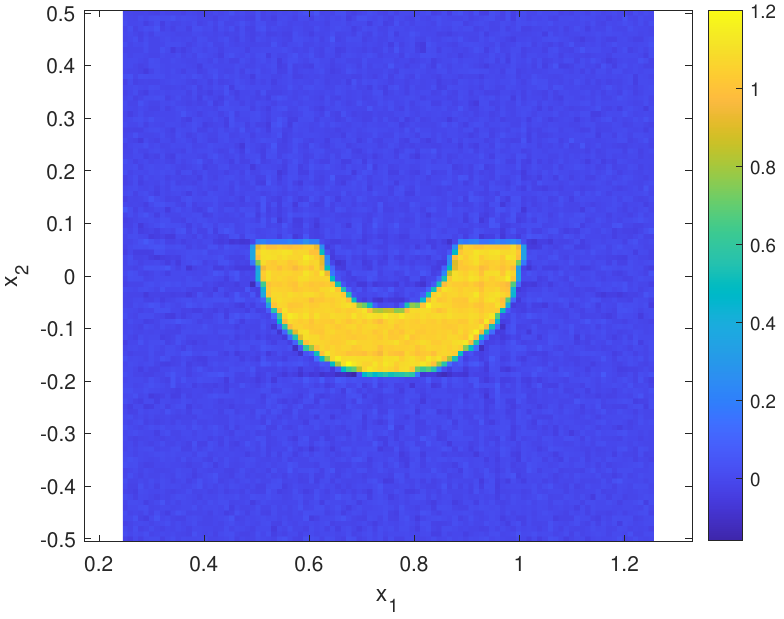}
\subcaption*{Landweber ($\delta = 0.19$)}
\end{subfigure}
\begin{subfigure}{0.24\textwidth}
\includegraphics[width=0.9\linewidth, height=3.2cm, keepaspectratio]{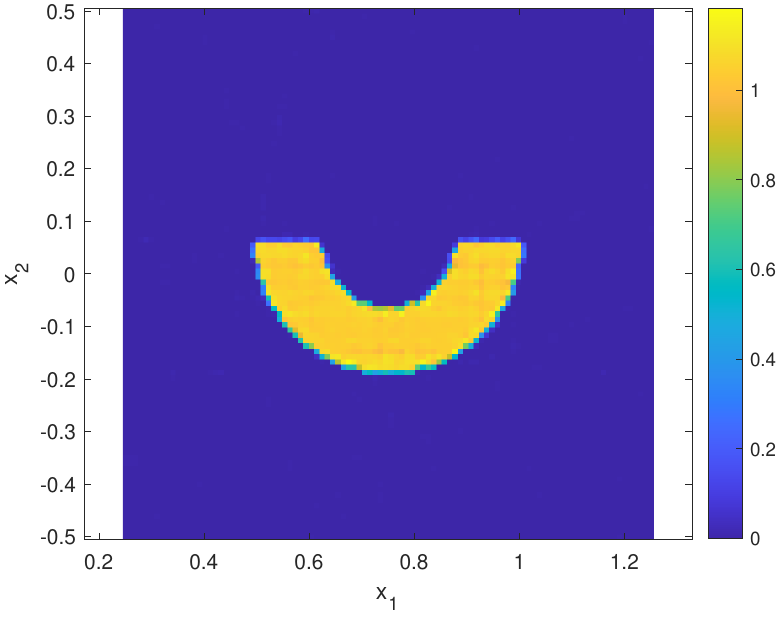}
\subcaption*{TV ($\delta = 0.17$)}
\end{subfigure}
\caption{{Ground truth half annulus phantom and sinogram for the
URT geometry of section \ref{sec:const r} (on the left). Data are
taken for $\vy\in Y''$, i.e., the $Rf$ values are set to zero in
$D_0$. Reconstructions are presented in the right two figures using the Landweber method and TV. Note that the reconstruction region is
$[0.25,1.25]\times[-0.5,0.5]$, $r = 1.25$ and $d = 0.25$. The noise
level added to the data was $\gamma = 0.05$. The least squares error
values ($\delta$) corresponding to the reconstructions are given in }}
\label{F4}
\end{figure}


The Landweber reconstruction is artifact free (i.e., no significant streaking) and the noise in the data appears only slightly amplified in the reconstruction, which is as expected as the inverse problem is only mildly ill-posed (Corollary \ref{corr_stab_4}). 
\begin{figure}
\centering
\begin{subfigure}{0.24\textwidth}
\includegraphics[width=0.9\linewidth, height=3.2cm, keepaspectratio]{sino_urt_2}
\end{subfigure}
\begin{subfigure}{0.24\textwidth}
\includegraphics[width=0.9\linewidth, height=3.2cm, keepaspectratio]{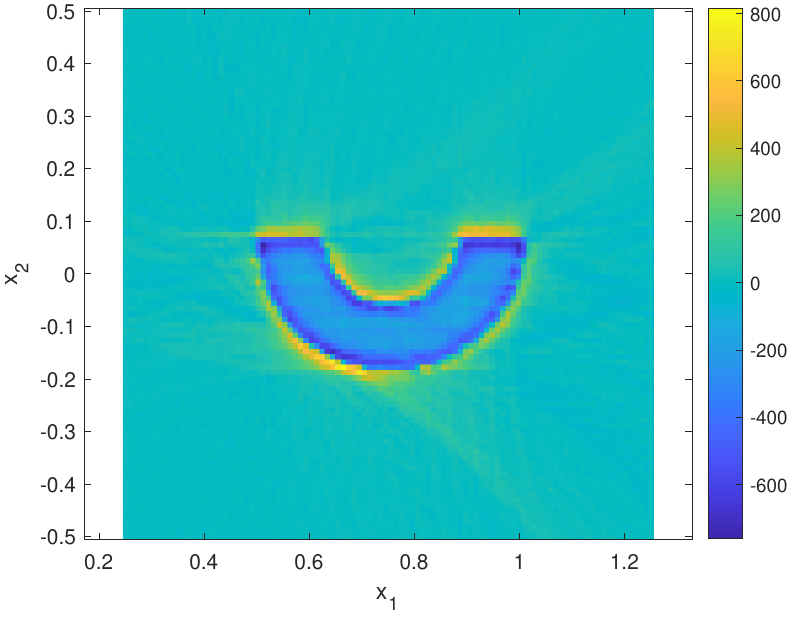}
\end{subfigure}
\begin{subfigure}{0.24\textwidth}
\includegraphics[width=0.9\linewidth, height=3.2cm, keepaspectratio]{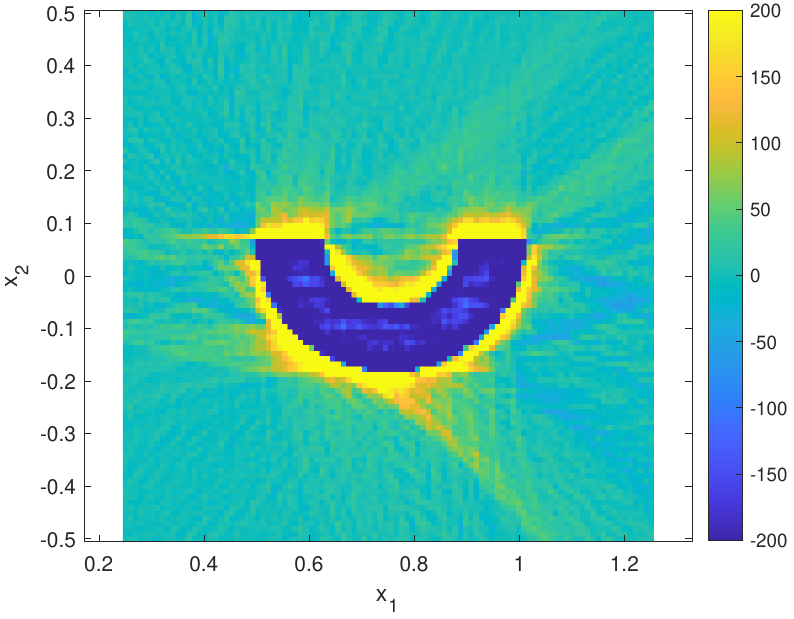} \label{F5c}
\end{subfigure}
\\
\begin{subfigure}{0.24\textwidth}
\includegraphics[width=0.9\linewidth, height=3.2cm, keepaspectratio]{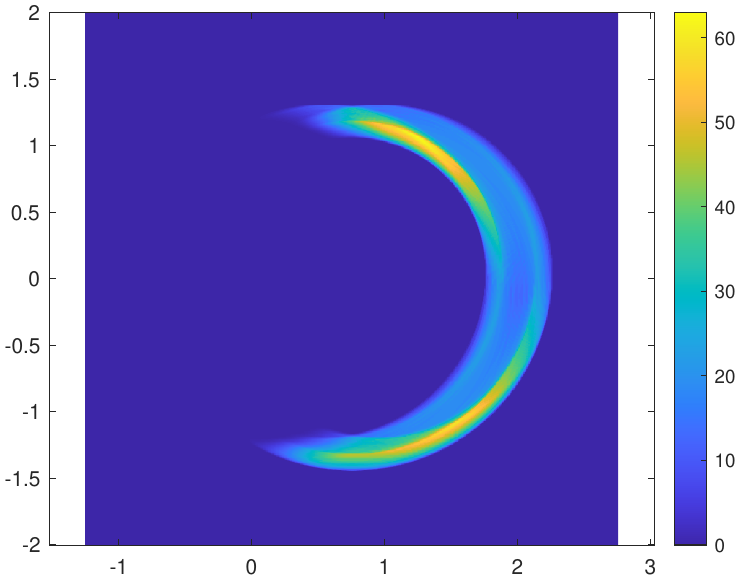}
\subcaption*{sinogram}
\end{subfigure}
\begin{subfigure}{0.24\textwidth}
\includegraphics[width=0.9\linewidth, height=3.2cm, keepaspectratio]{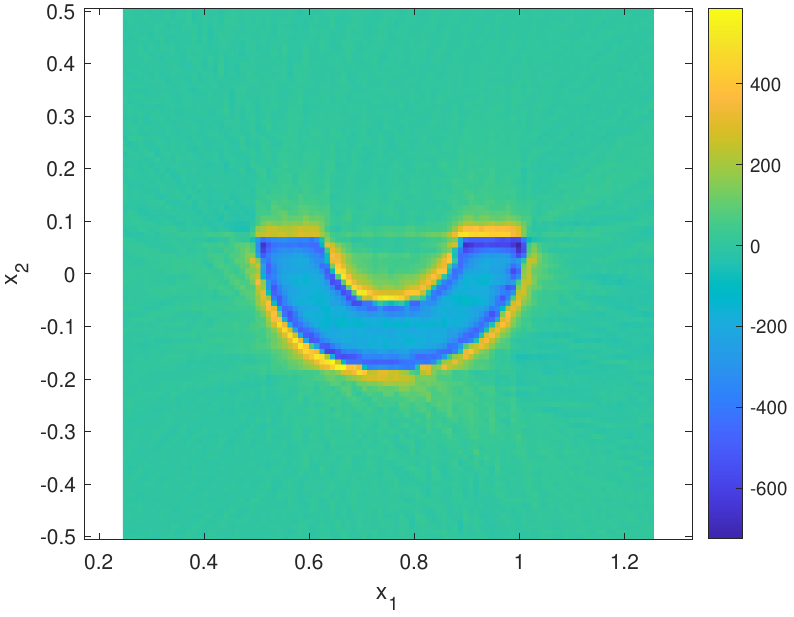}
\subcaption*{FBP}
\end{subfigure}
\begin{subfigure}{0.24\textwidth}
\includegraphics[width=0.9\linewidth, height=3.2cm, keepaspectratio]{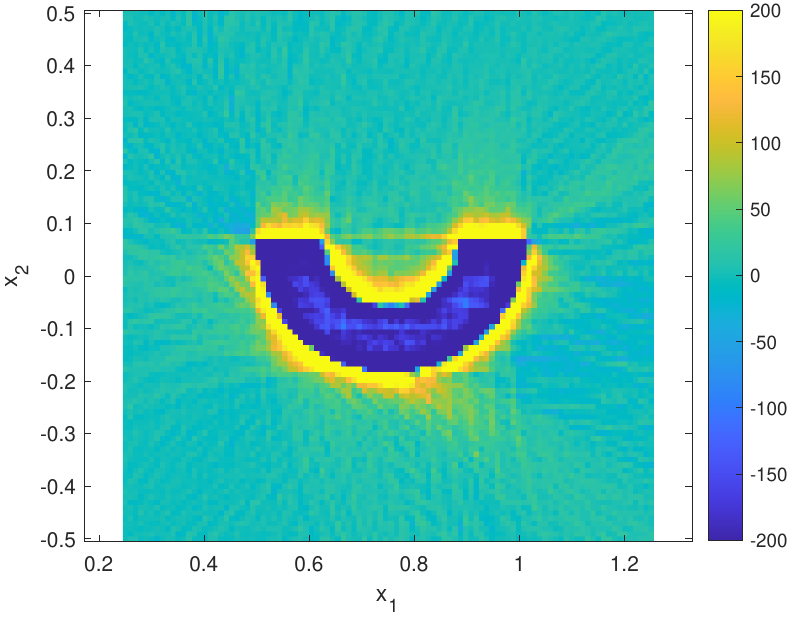} \label{F5f}
\subcaption*{FBP (limited colorbar)}
\end{subfigure}
\caption{FBP reconstructions of the half annulus phantom for the URT
geometry of section \ref{sec:const r}.  The top row has  a sharply
cropped sinogram,  and the bottom row has a  smoothly cropped sinogram using the cutoff $h$
as in Theorem \ref{ell_thm_4}. In the reconstructions in the
right-hand column we limit the colorbar so as to highlight the
artifacts better. The reconstructions with the original colorbars
(showing the full range of values) are in the middle column.}
\label{F5}
\end{figure}
We also do not detect the presence of edge of sinogram artifacts due
to the cropping of the sinogram space, and thus the Landweber and TV
methods appear to be sufficiently suppressing these artifacts in this
example. For comparison, we also provide FBP reconstructions in
figure \ref{F5}. Here, we reconstruct using {$R^*\Delta$}, where $\Delta$ is the Laplacian. We see in the reconstruction using the limited colorbar in the top-right of figure \ref{F5} that there are significant streaking artifacts due to the sharp cropping of the sinogram. These artifacts are suppressed in the bottom-right of figure \ref{F5} after we apply a smooth cutoff as in Theorem \ref{ell_thm_4}.


As in the previous two examples which were related to CST, the TV reconstruction here is most optimal and combats well the noise seen in the Landweber image. The parameters $r= 1.25$ and $d=0.25$ used here are the same as in figure \ref{fig4}. Thus, we only need the $\vy$ in the set {($S_D$)} shown in figure \ref{fig4} for stable reconstruction. This is of note as we only require circle centers $\vy$ largely on one side of the scanning target for stable reconstruction. 

%% file: conclusion.tex
\section{Conclusion} In this paper, we presented novel microlocal and
stability analyses of a spherical Radon transform, $R$, which
integrates $f$ over spheres with radii $r(\vy)$ which vary smoothly
with the sphere center, $\vy$. In particular, in Theorem \ref{inj_thm}
we provided conditions on $\nabla_\vy r$, and other geometric
requirements so that $R$ satisfies the Bolker condition. Then, under
the slightly stronger conditions in Theorem \ref{thm:elliptic psido},
we show that $R^*R$ is an elliptic \psido. This theory was extended in
Definition \ref{stable_def} to provide conditions for stable
reconstruction. In section \ref{examples}, we gave examples in CST and
URT which fit our framework, and we used our theory to derive Sobolev
stability estimates which show that reconstruction is only mildly
ill-posed. Inversion methods were also provided. For example, we used
the uniform inversion methods of \cite{palamodov2012uniform} to derive
a closed-form inversion formula for a rotational CST geometry. Our
theory was later validated in section \ref{images} using simulated
image reconstructions, and we were able to demonstrate stable
reconstructions of image phantoms in a variety of imaging modalities.

The rotational CST modality introduced here shows promise in terms of
inversion stability, e.g., we proved global stability estimates in
Sobolev space (Corollary \ref{cor:Sobolev estimate}). Here, the model
is linear, which in the context of CST means the ray attenuation is
neglected \cite{rigaud20183d}. In further work, we aim to consider a
non-linear variant of $R$ in CST which accounts for ray attenuation.
Specifically, we will use the microlocal theory developed here to
analyze how the singularities of $f$ relate to those in the data in
the non-linear formulation.

%% file: main-sphere.bbl
\begin{thebibliography}{10}

\bibitem{agranovsky2011support}
M.~Agranovsky and P.~Kuchment.
\newblock The support theorem for the single radius spherical mean transform.
\newblock {\em Mem. Differential Equations Math. Phys.}, 52:1--16, 2011.

\bibitem{agranovsky1996injectivity}
M.~L. Agranovsky and E.~T. Quinto.
\newblock Injectivity sets for the {R}adon transform over circles and complete systems of radial functions.
\newblock {\em Journal of Functional Analysis}, 139(2):383--414, 1996.

\bibitem{ABKQ2013}
G.~Ambartsoumian, J.~Boman, V.~P. Krishnan, and E.~T. Quinto.
\newblock Microlocal analysis of an ultrasound transform with circular source and receiver trajectories.
\newblock In {\em Geometric analysis and integral geometry}, volume 598 of {\em Contemp. Math.}, pages 45--58. Amer. Math. Soc., Providence, RI, 2013.

\bibitem{AFKNQ:common-midpoint}
G.~Ambartsoumian, R.~Felea, V.~P. Krishnan, C.~Nolan, and E.~T. Quinto.
\newblock A class of singular {F}ourier integral operators in synthetic aperture radar imaging.
\newblock {\em J. Funct. Anal.}, 264(1):246--269, 2013.

\bibitem{ambartsoumian2010inversion}
G.~Ambartsoumian, R.~Gouia-Zarrad, and M.~A. Lewis.
\newblock Inversion of the circular {R}adon transform on an annulus.
\newblock {\em Inverse Problems}, 26(10):105015, 2010.

\bibitem{andersson1988determination}
L.-E. Andersson.
\newblock On the determination of a function from spherical averages.
\newblock {\em SIAM Journal on Mathematical Analysis}, 19(1):214--232, 1988.

\bibitem{Caday:SAR}
P.~Caday.
\newblock Cancellation of singularities for synthetic aperture radar.
\newblock {\em Inverse Problems}, 31(1):015002, 22, 2015.

\bibitem{cebeiro2021three}
J.~Cebeiro, C.~Tarpau, M.~A. Morvidone, D.~Rubio, and M.~K. Nguyen.
\newblock On a three-dimensional compton scattering tomography system with fixed source.
\newblock {\em Inverse Problems}, 37(5):054001, 2021.

\bibitem{cormack1980radon}
A.~Cormack and E.~T. Quinto.
\newblock A {R}adon transform on spheres through the origin in $\mathbb{R}^n$ and applications to the {D}arboux equation.
\newblock {\em Transactions of the American Mathematical Society}, 260(2):575--581, 1980.

\bibitem{Co1963}
A.~M. Cormack.
\newblock {Representation of a function by its line integrals with some radiological applications}.
\newblock {\em J. Appl. Physics}, 34(9):2722--2727, 1963.

\bibitem{duistermaat1996fourier}
J.~J. Duistermaat and L.~Hormander.
\newblock {\em {F}ourier integral operators}, volume~2.
\newblock Springer, 1996.

\bibitem{ehrhardt2014joint}
M.~J. Ehrhardt, K.~Thielemans, L.~Pizarro, D.~Atkinson, S.~Ourselin, B.~F. Hutton, and S.~R. Arridge.
\newblock Joint reconstruction of {PET}-{MRI} by exploiting structural similarity.
\newblock {\em Inverse Problems}, 31(1):015001, 2014.

\bibitem{ery}
A.~Erdelyi, W.~Magnus, R.~Oberhettinger, and T.~F.
\newblock {\em Higher Transcendental Functions}, volume~II.
\newblock McGraw-Hill, 1953.

\bibitem{felea2013microlocal}
R.~Felea, R.~Gaburro, and C.~J. Nolan.
\newblock Microlocal analysis of sar imaging of a dynamic reflectivity function.
\newblock {\em SIAM Journal on Mathematical Analysis}, 45(5):2767--2789, 2013.

\bibitem{GKQR:ma}
C.~Grathwohl, P.~Kunstmann, E.~T. Quinto, and A.~Rieder.
\newblock {Microlocal analysis of imaging operators for effective common offset seismic reconstruction}.
\newblock {\em Inverse Problems}, 34(11):114001 (24 pages), 2018.

\bibitem{Gu1975}
V.~Guillemin.
\newblock {Some remarks on integral geometry}.
\newblock Technical report, MIT, 1975.

\bibitem{GS1977}
V.~Guillemin and S.~Sternberg.
\newblock {\em {Geometric Asymptotics}}.
\newblock American Mathematical Society, Providence, RI, 1977.

\bibitem{haltmeier2017spherical}
M.~Haltmeier and S.~Moon.
\newblock The spherical {R}adon transform with centers on cylindrical surfaces.
\newblock {\em Journal of Mathematical Analysis and Applications}, 448(1):567--579, 2017.

\bibitem{AIRtools}
P.~C. Hansen and J.~S. J\o{rgensen}.
\newblock A{IR} {T}ools {II}: algebraic iterative reconstruction methods, improved implementation.
\newblock {\em Numer. Algorithms}, 79(1):107--137, 2018.

\bibitem{Ho1971}
L.~H{\"o}rmander.
\newblock {Fourier Integral Operators, I}.
\newblock {\em Acta Mathematica}, 127:79--183, 1971.

\bibitem{hormanderI}
L.~H{\"o}rmander.
\newblock {\em The analysis of linear partial differential operators. {I}}.
\newblock Classics in Mathematics. Springer-Verlag, Berlin, 2003.
\newblock Distribution theory and {F}ourier analysis, Reprint of the second (1990) edition [Springer, Berlin].

\bibitem{hormanderIII}
L.~H\"{o}rmander.
\newblock {\em The analysis of linear partial differential operators. {III}}.
\newblock Classics in Mathematics. Springer, Berlin, 2007.
\newblock Pseudo-differential operators, Reprint of the 1994 edition.

\bibitem{hormander}
L.~H\"{o}rmander.
\newblock {\em The analysis of linear partial differential operators. {IV}}.
\newblock Classics in Mathematics. Springer-Verlag, Berlin, 2009.
\newblock Fourier integral operators, Reprint of the 1994 edition.

\bibitem{john2004plane}
F.~John.
\newblock {\em Plane waves and spherical means applied to partial differential equations}.
\newblock Courier Corporation, 2004.

\bibitem{klein2003inverting}
J.~Klein.
\newblock Inverting the spherical {R}adon transform for physically meaningful functions.
\newblock {\em arXiv preprint math/0307348}, 2003.

\bibitem{KrQu2011}
V.~P. Krishnan and E.~T. Quinto.
\newblock Microlocal aspects of common offset synthetic aperture radar imaging.
\newblock {\em Inverse Probl. Imaging}, 5(3):659--674, 2011.

\bibitem{kunyansky2007explicit}
L.~A. Kunyansky.
\newblock Explicit inversion formulae for the spherical mean {R}adon transform.
\newblock {\em Inverse problems}, 23(1):373, 2007.

\bibitem{landweber1951iteration}
L.~Landweber.
\newblock An iteration formula for {F}redholm integral equations of the first kind.
\newblock {\em American journal of mathematics}, 73(3):615--624, 1951.

\bibitem{lee2012smooth}
J.~M. Lee.
\newblock {\em Introduction to Smooth Manifolds, Second Edition}.
\newblock Number 218 in Graduate Texts in Mathematics. Springer, 2013.

\bibitem{natterer}
F.~Natterer.
\newblock {\em {The mathematics of computerized tomography}}.
\newblock Classics in Mathematics. Society for Industrial and Applied Mathematics (SIAM), New York, 2001.

\bibitem{Nguyen-Pham}
L.~V. Nguyen and T.~A. Pham.
\newblock Microlocal analysis for spherical {R}adon transform: two nonstandard problems.
\newblock {\em Inverse Problems}, 35(7):074001, 15, 2019.

\bibitem{norton}
S.~J. Norton.
\newblock {Compton scattering tomography}.
\newblock {\em Journal of applied physics}, 76(4):2007--2015, 1994.

\bibitem{palamodov2012uniform}
V.~P. Palamodov.
\newblock A uniform reconstruction formula in integral geometry.
\newblock {\em Inverse Problems}, 28(6):065014, 2012.

\bibitem{quinto}
E.~T. Quinto.
\newblock {The dependence of the generalized {R}adon transform on defining measures}.
\newblock {\em Trans. Amer. Math. Soc.}, 257:331--346, 1980.

\bibitem{Q1983-rotation}
E.~T. Quinto.
\newblock {The invertibility of rotation invariant {R}adon transforms}.
\newblock {\em J. Math. Anal. Appl.}, 94:602--603, 1983.

\bibitem{Q1993mor}
E.~T. Quinto.
\newblock {Pompeiu transforms on geodesic spheres in real analytic manifolds}.
\newblock {\em Israel J. Math.}, 84:353--363, 1993.

\bibitem{Q2006:supp}
E.~T. Quinto.
\newblock {Support Theorems for the Spherical {R}adon Transform on Manifolds}.
\newblock {\em International Mathematics Research Notices}, 2006:1--17, 2006.
\newblock Article ID = 67205.

\bibitem{RigaudComptonSIIMS2017}
G.~Rigaud.
\newblock Compton scattering tomography: feature reconstruction and rotation-free modality.
\newblock {\em SIAM J. Imaging Sci.}, 10(4):2217--2249, 2017.

\bibitem{rigaud20183d}
G.~Rigaud and B.~N. Hahn.
\newblock 3{D} {C}ompton scattering imaging and contour reconstruction for a class of {R}adon transforms.
\newblock {\em Inverse Problems}, 34(7):075004, 2018.

\bibitem{rubin2002inversion}
B.~Rubin.
\newblock Inversion formulas for the spherical {R}adon transform and the generalized cosine transform.
\newblock {\em Advances in Applied Mathematics}, 29(3):471--497, 2002.

\bibitem{Rudin:FA}
W.~Rudin.
\newblock {\em Functional analysis}.
\newblock McGraw-Hill Book Co., New York, 1973.
\newblock McGraw-Hill Series in Higher Mathematics.

\bibitem{stefanov2004stability}
P.~Stefanov and G.~Uhlmann.
\newblock Stability estimates for the {X}-ray transform of tensor fields and boundary rigidity.
\newblock {\em Duke Mathematical Journal}, 123(2):445--467, 2004.

\bibitem{SU:SAR2013}
P.~Stefanov and G.~Uhlmann.
\newblock Is a curved flight path in {SAR} better than a straight one?
\newblock {\em SIAM J. Appl. Math.}, 73(4):1596--1612, 2013.

\bibitem{truong2019compton}
T.-T. Truong and M.~K. Nguyen.
\newblock Compton scatter tomography in annular domains.
\newblock {\em Inverse Problems}, 2019.

\bibitem{webber2024surface}
J.~Webber, S.~Holman, and E.~T. Quinto.
\newblock Surface of revolution {R}adon transforms with centers on generalized surfaces in rn.
\newblock {\em SIAM Journal on Mathematical Analysis}, 2024.

\bibitem{webber2019compton}
J.~Webber and E.~Miller.
\newblock {Compton scattering tomography in translational geometries}.
\newblock Technical report, Tufts University, 2019.

\bibitem{webber2024generalized}
J.~W. Webber.
\newblock Generalized {A}bel equations and applications to translation invariant {R}adon transforms.
\newblock {\em Journal of Inverse and Ill-posed Problems}, 32(4):835--857, 2024.

\bibitem{webberholman}
J.~W. Webber and S.~Holman.
\newblock Microlocal analysis of a spindle transform.
\newblock {\em Inverse Problems \& Imaging}, 13(2):231--261, 2019.

\bibitem{webber2023ellipsoidal}
J.~W. Webber, S.~Holman, and E.~T. Quinto.
\newblock Ellipsoidal and hyperbolic {R}adon transforms; microlocal properties and injectivity.
\newblock {\em Journal of Functional Analysis}, page 110056, 2023.

\bibitem{web}
J.~W. Webber and E.~L. Miller.
\newblock Bragg scattering tomography.
\newblock {\em Inverse Problems and Imaging}, 15(4):683--721, 2021.

\end{thebibliography}
